\documentclass[12pt]{amsart}
\usepackage{amsfonts,amssymb, amscd, latexsym, graphicx, psfrag, color,float}
\usepackage{bbm}

\usepackage[all]{xy}

\usepackage{booktabs}
\usepackage{multirow}
\usepackage{mathrsfs}
\usepackage{verbatim}

\usepackage{microtype} 
\usepackage[margin=1in,marginparwidth=0.8in, marginparsep=0.1in]{geometry}

\usepackage[bookmarks=true, bookmarksopen=true,%
bookmarksdepth=3,bookmarksopenlevel=2,%
colorlinks=true,%
linkcolor=blue,%
citecolor=blue,%
filecolor=blue,%
menucolor=blue,%
urlcolor=blue]{hyperref}

\newtheorem{dummy}{dummy}[section]
\newtheorem{lemma}[dummy]{Lemma}
\newtheorem{theorem}[dummy]{Theorem}

\newtheorem{corollary}[dummy]{Corollary}
\newtheorem{proposition}[dummy]{Proposition}
\theoremstyle{definition}
\newtheorem{notation}[dummy]{Notation}
\newtheorem{definition}[dummy]{Definition}
\newtheorem{example}[dummy]{Example}
\newtheorem{remark}[dummy]{Remark}

\newtheorem*{prop*}{Proposition}

\newcommand{\bA}{\mathbb{A}}
\newcommand{\bC}{\mathbb{C}}

\newcommand{\bN}{\mathbb{N}}
\newcommand{\bP}{\mathbb{P}}

\newcommand{\bR}{\mathbb{R}}
\newcommand{\bZ}{\mathbb{Z}}

\newcommand{\cA}{\mathcal{A}}

\newcommand{\cE}{\mathcal{E}}
\newcommand{\cF}{\mathcal{F}}
\newcommand{\cG}{\mathcal{G}}

\newcommand{\cL}{\mathcal{L}}

\newcommand{\cO}{\mathcal{O}}

\newcommand{\cT}{\mathcal{T}}

\newcommand{\cX}{\mathcal {X}}

\newcommand{\op}{\operatorname}

\newcommand{\M}{M}

\newcommand{\La}{\Lambda}
\newcommand{\Ga}{\Gamma}
\newcommand{\ga}{\gamma}

\DeclareMathOperator{\sheafHom}{\ensuremath{\mathscr{H}\hspace{-.25ex}\mathit{om}}}

\newcommand{\Fuk}{\mathrm{Fuk}}

\newcommand{\Hom}{\mathrm{Hom}}

\renewcommand{\log}{{\op{log}}}

\newcommand{\Sh}{\mathit{Sh}}
\newcommand{\Loc}{\mathit{Loc}}

\newcommand{\coeffs}{{\mathbbm k}}

 \numberwithin{equation}{subsection}

 \newcommand{\CCC}{{\rm CCC}}

\newcommand{\Perf}{\mathrm{Perf}}
\newcommand{\Perfprop}{\mathrm{Perf}_{\mathrm{prop}}}
\newcommand{\Coh}{\mathrm{Coh}}
\newcommand{\QCoh}{\mathrm{QCoh}}
\newcommand{\IndCoh}{\mathrm{IndCoh}}

\newcommand{\congto}{\xrightarrow{\sim}}
\newcommand\onto{\twoheadrightarrow}
\newcommand\into{\hookrightarrow}

\newcommand{\stackfan}{\mathbf{\Sigma}}
\newcommand{\stackfano}{\stackfan^\circ}
\newcommand{\Sigmao}{\Sigma^\circ}
\newcommand{\whSigmao}{\widehat{\Sigma}^\circ}
\newcommand{\Legtimes}{\dot{\times}}
 
\newcommand{\newword}[1]{\textbf{\emph{#1}}}

\newcommand{\arrtip}{latex'}

\newcommand*{\hrlen}{10}
\newcommand*{\hramp}{3}
\newcommand*{\fcolor}{black!5}
\newcommand*{\bvertrad}{.09}
\newcommand*{\wvertrad}{.06}
\newcommand*{\ifac}{.26}

\usepackage{pgfplots}
\usepgfplotslibrary{fillbetween}
\usepackage{tikz}
\usetikzlibrary{matrix,shapes,arrows,arrows.meta,calc,topaths,intersections,hobby,positioning,decorations.pathreplacing,decorations.pathmorphing,fit,patterns}
\tikzset{
asdstyle/.style={blue,thick},
righthairs/.style={postaction={decorate,draw,decoration={border,amplitude=\hramp,segment length=\hrlen,angle=-90,pre=moveto,pre length=\hrlen/2}}},
lefthairs/.style={postaction={decorate,draw,decoration={border,amplitude=\hramp,segment length=\hrlen,angle=90,pre=moveto,pre length=\hrlen/2}}},
righthairsnogap/.style={postaction={decorate,draw,decoration={border,amplitude=\hramp,segment length=\hrlen,angle=-90}}},
lefthairsnogap/.style={postaction={decorate,draw,decoration={border,amplitude=\hramp,segment length=\hrlen,angle=90}}},
graphstyle/.style={thick},
arrowstyle/.style={thick,decorate,decoration={snake,amplitude=1.7,segment length=10pt,post length=.5mm,pre length=0}},
genmapstyle/.style={thick,-stealth'},
arrhdstyle/.style={thick},
exceptarcstyle/.style={red, ultra thick},
dualquiverstyle/.style={thick,->},
|/.tip={Bar[width=.8ex]}
}
 
\title{Kasteleyn Operators from Mirror Symmetry}
\author{David Treumann, Harold Williams and Eric Zaslow}

\begin{document}
	
\begin{abstract}
Given a consistent bipartite graph $\Gamma$ in $T^2$ with a complex-valued edge weighting $\cE$ we show the following two constructions are the same. The first is to form the Kasteleyn operator of $(\Gamma, \cE)$ and pass to its spectral transform, a coherent sheaf supported on a spectral curve in $(\bC^\times)^2$. The second is to form the conjugate Lagrangian $L \subset T^* T^2$ of $\Gamma$, equip it with a brane structure prescribed by $\cE$, and pass to its mirror coherent sheaf. This lives on a stacky toric compactification of $(\bC^\times)^2$ determined by the Legendrian link which lifts the zig-zag paths of $\Gamma$ (and to which the noncompact Lagrangian $L$ is asymptotic). We work in the setting of the coherent-constructible correspondence, a sheaf-theoretic model of toric mirror symmetry. We also show that tensoring with line bundles on the compactification is mirror to certain Legendrian autoisotopies of the asymptotic boundary of $L$. 
\end{abstract}

\maketitle

\tableofcontents

\section{Introduction}
\thispagestyle{empty}

In pioneering work, Kenyon-Okounkov-Sheffield \cite{KOS} 
showed that the statistical properties of dimer configurations on a doubly periodic bipartite graph in $\bR^2$ are largely determined by an algebraic curve: the spectral curve of its Kasteleyn operator $K(x,y)$. 
This is a matrix-valued Laurent polynomial that depends on a choice of edge weights on the associated finite graph $\Gamma$ in $T^2 = \bR^2/\bZ^2$. The vanishing locus of its determinant defines a curve $C \subset (\bC^\times)^2$, and its cokernel defines a sheaf supported on $C$. As this spectral data only depends on edge weights up to gauge transformations, its construction can be understood as a map from local systems on $\Gamma$ to coherent sheaves on $(\bC^\times)^2$. The purpose of this paper is to identify this spectral transform as an instance of homological mirror symmetry.

The algebro-geometric side of the mirror relation is a toric compactification $\cX_{\stackfan}$ of $T_N := (\bC^\times)^2$. Here $\stackfan$ denotes a complete (possibly stacky) fan in $N_\bR := \bR^2$ determined by the zig-zag paths of $\Gamma$, a configuration of immersed curves canonically associated to any bipartite surface graph. The symplectic counterpart of this compactification is a singular Legendrian $\Lambda_{\stackfan} \subset T^\infty T^2$ in the contact boundary of $T^* T^2$, described in \cite{FLTZ,FLTZ3}. Lagrangian branes asymptotic to $\Lambda_{\stackfan}$ are faithfully modeled by constructible sheaves on $T^2$ with microsupport asymptotic to $\Lambda_{\stackfan}$; see \cite{Nad09,NZ} or \cite{Gui12,JT17} for complementary approaches through Floer theory and pure sheaf theory, respectively. 
In this form the mirror equivalence $$\Perf(\cX_{\stackfan}) \cong \Sh^c_{\Lambda_{\stackfan}}(T^2)$$ between the dg categories of perfect complexes and constructible sheaves is referred to as the coherent-constructible correspondence (or $CCC$) \cite{FLTZ}, proved in the stated generality in \cite{Kuw16}.

Local systems on the graph $\Gamma$ give rise to Lagrangian branes in $T^* T^2$ via the construction of \cite{STWZ}: up to Hamiltonian isotopy there is a canonical embedded exact Lagrangian $L_\Gamma$ in $T^*T^2$ which deformation retracts onto $\Gamma$. In particular, there is an isomorphism $\Loc_1(\Gamma) \cong \Loc_1(L_\Gamma)$ between their algebraic tori of rank one local systems. The Lagrangian $L_\Gamma$ is noncompact but asymptotic to the Legendrian lift $\Lambda_\Gamma \subset T^\infty T^2$ of the zig-zag paths. It follows from the formalism of sheaf quantization that there is an associated embedding $\Loc_1(L_\Gamma) \into \Sh^c_{\Lambda_\Gamma}(T^2)$. The resulting objects were termed \emph{alternating sheaves} in \cite{STWZ}.

Provided $\Gamma$ satisfies a certain consistency condition, $\Lambda_\Gamma$ is Legendrian isotopic to the Legendrian link $\Lambda_{\stackfano}$ associated to the rays of $\stackfan$. Following the results of \cite{GKS}, such an isotopy quantizes to an equivalence $\Sh^c_{\Lambda_\Gamma}(T^2) \congto \Sh^c_{\Lambda_{\stackfano}}(T^2)$. On the other hand, $\Lambda_{\stackfano}$ is a subset of the singular Legendrian $\Lambda_{\stackfan}$, hence there is a fully faithful inclusion $\Sh^c_{\Lambda_{\stackfano}}(T^2) \into \Sh^c_{\Lambda_{\stackfan}}(T^2)$. The composition of these assembled maps with the coherent-constructible correspondence now defines a map from local systems on $\Gamma$ to coherent sheaves on $\cX_{\stackfan}$. Our main result is that this provides a geometric interpretation of the Kasteleyn operator $K(x,y)$:

\vskip0.15in
\begin{center}
\fbox{
\emph{the mirror map from local systems to coherent sheaves
is the spectral transform.}}
\end{center}
\vskip0.15in

\begin{theorem}[c.f. Theorem \ref{thm:mainresult}]\label{thm:mainthmintro}
Let $\Gamma \subset T^2$ be a consistent bipartite graph, $\stackfan$ the associated complete stacky fan. Then the following diagram commutes.
\[
\begin{tikzpicture}
[thick,>=\arrtip,every text node part/.style={align=center}]
\newcommand*{\xa}{2.8}; \newcommand*{\xb}{2.8}; \newcommand*{\xc}{2.5}; \newcommand*{\xd}{2.5}
\newcommand*{\ya}{0}; \newcommand*{\yb}{1.8};
\node[matrix] (left) at (0,0) {
	\node (a) at (0,0) {$\Sh^c_{\Lambda_\Gamma}(T^2)$};
	\node (b) at (\xa,0) {$\Sh^c_{\Lambda_{\stackfan}}(T^2)$};
	\node (c) at (\xa+\xb,0) {$\Perf(\cX_{\stackfan})$};
	\node (d) at (\xa+\xb+\xc,0) {$\Perf(T_N)$};
	\node (e) at (-\xd,-\ya-\yb) {$\Loc_1(\Gamma)$};
	\node (f) at (\xa+\xb+\xc,-\ya-\yb) {$\begin{Bmatrix}\text{pure sheaves of} \\ \text{dimension one}\end{Bmatrix}$};
	\node (g) at (-\xd,0) {$\Loc_1(L_\Gamma)$};
	\draw[->] (a) to (b); \draw[->] (b) to node[above] {$\sim$}  (c); \draw[->] (c) to (d);
	\draw[right hook->] (g) to  (a); \draw[->] (e) to node[above, rotate=90] {$\sim$} (g); \draw[->] (e) to node[below] {spectral transform of $K(x,y)$} (f); \draw[right hook->] (f) to (d);\\};
\end{tikzpicture}
\]
Here the bottom and top left maps are defined by a fixed Kasteleyn orientation. 
The top row is the composition of i) quantization of local systems on $L_\Gamma$ as alternating sheaves, ii) the GKS equivalence associated to a Legendrian isotopy $\Lambda_\Gamma \to \Lambda_{\stackfano}$ and the inclusion $\Sh^c_{\Lambda_{\stackfano}}(T^2) \subset \Sh^c_{\Lambda_{\stackfan}}(T^2)$, iii) the $CCC$, and iv) restriction to $T_N \subset \cX_{\stackfan}$.
\end{theorem}

The Theorem can also be summarized more coarsely, without referencing the compactification $\cX_{\stackfan}$, as follows: $L_\Gamma$ supports objects of the wrapped Fukaya category of $T^*T^2$, and mirror objects in $\Perf(T_N)$ are supported on the spectral curve. Indeed the notion of wrapping appears implicitly in the proof of Theorem \ref{thm:mainthmintro} in the guise of convolution with a free local system on~$T^2$, see Lemma \ref{lem:openorbitmain}.

If $\Gamma^b_0, \Gamma^w_0 \subset \Gamma_0$ denote the sets of black and white vertices of $\Gamma$, let us recall in more detail that the Kasteleyn operator 
$$ K(x,y): \bC[T_N]^{\Gamma^b_0} \to \bC[T_N]^{\Gamma^w_0} $$
is a matrix-valued function on $T_N$ whose entries are sign-twisted weighted edge counts. The edge weights depend on a choice of local system on $\Gamma$, and the signs are prescribed by a Kasteleyn orientation. This is an assignment $\Gamma_1 \to \{\pm 1\}$ satisfying certain conditions, and such an assignment can be identified with the choice of spin structure on $L_\Gamma$ used to fix certain signs in the map $\Loc_1(L_\Gamma) \into \Sh^c_{\Lambda_\Gamma}(T^2)$. The spectral transform of $K(x,y)$ is its cokernel as a map of $\bC[T_N]$-modules, and for generic edge weights is the pushforward of a line bundle on a smooth curve. 

\subsection*{Isotopies and integrability}
A main result of \cite{GK} is that the forgetful map from spectral data to the underlying spectral curve defines an algebraic completely integrable system with respect to the canonical Poisson structure on $\Loc_1(L_\Gamma)$. The coefficients of the defining equation of the curve give a collection of regular functions on $\Loc_1(L_\Gamma)$, well-defined up to an overall scalar. These provide Hamiltonians for the integrable system, and suitably normalized can be identified with summands of the partition function for the dimer model on $\Gamma$ (that is, they are weighted counts of perfect matchings, organized by their associated class in $H_1(T^2)$). 

Different incarnations of this integrable system have been studied from a wide range of points of view --- see for example \cite{B,CW,DM,FM,EFS,GSTV,FHM}. A generic fiber can be identified with a finite cover of the Jacobian of the closure $\overline{C}$ of $C$ in (the coarse moduli space of) $\cX_{\stackfan}$.
In \cite{KO,GK} such an identification is determined as follows: a choice of white vertex defines a section of $\mathrm{cok}\, K(x,y)$, and one pushes forward its vanishing divisor from $C$ to $\overline{C}$. 

In our framework, the identification of generic Liouville fibers with covers of Jacobians is determined by the choice of isotopy $\Lambda_\Gamma \to \Lambda_{\stackfano}$. That is, given a local system on $L_\Gamma$ each choice of isotopy determines an extension of $\mathrm{cok}\, K(x,y)$ to a sheaf on the corresponding closed curve $\overline{C}$. These sheaves are not in general isomorphic though their restrictions to $T_N$ are. 

To understand this ambiguity we consider the following elementary autoisotopies of $\Lambda_{\stackfano}$: each ray $\rho$ of $\stackfan$ determines a collection of pairwise isotopic components with parallel front projections in $T^2$, and we let $\sigma_\rho$ be the autoisotopy which moves these in their normal direction until they become cyclically permuted (see Figure \ref{fig:lambdarhot}). These act by autoequivalences on $\Sh^c_{\Lambda_{\stackfano}}(T^2)$ following \cite{GKS}. On the other hand, also associated to $\rho$ is a line bundle $\cL_\rho$ on $\cX_{\stackfan}$ --- when $\cX_{\stackfan}$ is a variety this is just the line bundle $\cO(- D_\rho)$ defined by the toric divisor $D_\rho$, and is a root of $\cO(- D_\rho)$ when $D_\rho$ has nontrivial stabilizers. 
		
\begin{proposition}[c.f. Proposition \ref{prop:mirrorisotopies}]\label{prop:isotopyintro}
	The autoisotopy $\sigma_\rho$ acts on $\Sh^c_{\Lambda_{\stackfano}}(T^2)$ by the restriction of the autoequivalence of $\Sh^c_{\Lambda_{\stackfan}}(T^2) \cong \Perf(\cX_{\stackfan})$ given by tensoring with $\cL_\rho$.
\end{proposition}

The finite covers appearing in the above discussion and in \cite[Sec. 1.4.3]{GK} reflect the discrepancy between the stack $\cX_{\stackfan}$ and its coarse moduli space: sheaves on the former record extra information about the action of finite stabilizers, while the Jacobian of a non-stacky spectral curve does not record this data. We discuss this issue in more detail in Section \ref{sec:satellites}.

The Poisson commutativity of the Goncharov-Kenyon Hamiltonians in the canonical Poisson structure on $\Loc_1(L_\Gamma)$ should admit an alternative derivation from the present point of view. On one hand, it is known that the relative Picard variety of a family of smooth curves in a toric surface has a natural Poisson structure in which the forgetful map to the Hilbert scheme is a Lagrangian fibration (this is a special case of \cite[Section 8]{DM}). On the other hand, it is understood by work of \cite{BD}, pursued in the present context in \cite{ST}, that the Poisson structure on this space is essentially intrinsic to the underlying category whose moduli we are considering. Thus in our example the natural Poisson structures from the coherent and constructible descriptions of the category should coincide, and a Lagrangian fibration for one is a Lagrangian fibration for the other. 

\subsection*{Cluster structures} At any square face of $\Gamma$ we may perform a local move to produce a new bipartite graph $\Gamma'$, see Figure \ref{fig:squaremove}. The dual graph of $\Gamma$ is naturally a quiver (we orient it so that its edges pass a white vertex on their right) and when $\Gamma$ undergoes a square move its dual graph undergoes a quiver mutation. Note that holonomies around the faces of $\Gamma$ provide distinguished coordinates on $\Loc_1(\Gamma)$ (satisfying the single relation that their product is equal to $1$). A key result of \cite{STWZ} is that alternating sheaves before and after a square move are related by a commutative diagram 
\[
\begin{tikzpicture}
[thick,>=\arrtip,every text node part/.style={align=center}]
\newcommand*{\xa}{2.8}; \newcommand*{\xb}{2.8}; \newcommand*{\xc}{2.5}; \newcommand*{\xd}{2.5}
\newcommand*{\ya}{1.5}; \newcommand*{\yb}{1.8};
\node[matrix] (left) at (0,0) {
	\node (a) at (0,0) {$\Loc_1(L_\Gamma)$};
	\node (b) at (0,-\ya) {$\Loc_1(L_{\Gamma'})$};
	\node (c) at (\xa,0) {$\Sh^c_{\Lambda_{\Gamma}}(T^2)$};
	\node (d) at (\xa,-\ya) {$\Sh^c_{\Lambda_{\Gamma'}}(T^2).$};
	\draw[dashed,->] (a) to  (b); \draw[->] (c) to node[below, rotate=90] {$\sim$} (d);
	\draw[right hook->] (b) to (d); \draw[right hook->] (a) to (c);\\};
\end{tikzpicture}
\]
Here the right map is the equivalence defined by a canonical local isotopy $\Lambda_{\Gamma} \to \Lambda_{\Gamma'}$ and the left map is the cluster $\cX$-transformation associated to the given quiver mutation. Composition with this local isotopy identifies the sets of isotopies from $\Lambda_{\stackfano}$ to $\Lambda_\Gamma$ and $\Lambda_{\Gamma'}$, respectively. We immediately obtain the following result.

\begin{corollary}\label{cor:clusterintro} Let $\Gamma$, $\Gamma' \subset T^2$ be two consistent bipartite graphs differing by a square move. Then we have a commutative diagram
	\[
	\begin{tikzpicture}
[thick,>=\arrtip,every text node part/.style={align=center}]
\newcommand*{\xa}{2.8}; \newcommand*{\xb}{2.8}; \newcommand*{\xc}{2.5}; \newcommand*{\xd}{2.5}
\newcommand*{\ya}{1.5}; \newcommand*{\yb}{1.8}; \newcommand*{\xe}{5.6};
\node[matrix] (left) at (0,0) {
	\node (a) at (0,0) {$\Loc_1(L_\Gamma)$};
	\node (b) at (0,-\ya) {$\Loc_1(L_{\Gamma'})$};
	\node (c) at (\xa,0) {$\Sh^c_{\Lambda_{\Gamma}}(T^2)$};
	\node (d) at (\xa,-\ya) {$\Sh^c_{\Lambda_{\Gamma'}}(T^2)$};
	\draw[dashed,->] (a) to  (b); \draw[->] (c) to node[below, rotate=90] {$\sim$} (d);
	\draw[right hook->] (b) to (d); \draw[right hook->] (a) to (c);

		\node (e) at (\xe,-\ya*.5) {$\Perf(\cX_{\stackfan})$};
		\draw[->] (c) to (e); 
		\draw[->] (d) to (e); \\};
	\end{tikzpicture}
	\]
	where the left map is the cluster $\cX$-transformation associated to the given quiver mutation, and the maps to $\Perf(\cX_{\stackfan})$ are defined as in Theorem \ref{thm:mainthmintro} by compatible isotopies from $\Lambda_\Gamma$, $\Lambda_{\Gamma'}$ to $\Lambda_{\stackfano}$. 
\end{corollary}

The further corollary that the spectral curve is preserved by cluster transformations, hence up to a scalar so are the Goncharov-Kenyon Hamiltonians, is proved combinatorially in \cite[Theorem 4.7]{GK}. 

The images of $\Loc_1(\Gamma)$ in $\Sh^c_{\Lambda_{\stackfano}}(T^2)$ for various choices of graph $\Gamma$ and isotopy $\Lambda_\Gamma \to \Lambda_{\stackfano}$ have the common property that they consist of sheaves whose microstalks along $\Lambda_{\stackfano}$ form a rank one local system. Such sheaves were called microlocal rank one in \cite{STZ} and simple sheaves in \cite{KS94}. It follows from Proposition \ref{prop:isotopyintro} that the moduli space of microlocal rank one sheaves contains a countable family of components each of which has a (partial) cluster $\cX$-structure. On the constructible side the components are indexed by the Euler characteristic of the stalk of a sheaf at any point of $T^2$. Passing to the coherent side by Theorem \ref{thm:mainthmintro}, the image of a microlocal rank one sheaf in $\Perf(\cX_{\stackfan})$ is generically a line bundle supported on a smooth curve, and the components are indexed by the degree of this bundle.

A sequence of square moves that takes a bipartite graph $\Gamma$ back to itself yields an autoisotopy of of $\Lambda_\Gamma$, hence an autoequivalence of $\Sh^c_{\Lambda_\Gamma}(T^2)$. At the level of  moduli spaces this is an example of the automorphism of a cluster variety attached to a mutation periodic sequence. That in the present case such automorphisms act on spectral data by tensoring with line bundles, and in particular preserve the Goncharov-Kenyon Hamiltonians, was observed in \cite{GK} (see also \cite{FM}). This also follows immediately from our mirror-symmetric description of the spectral transform: any isotopy $\Lambda_\Gamma \to \Lambda_{\stackfano}$ intertwines the autoisotopy defined by a periodic sequence of square moves with a composition of the $\sigma_\rho$, hence by Proposition \ref{prop:isotopyintro} it acts on spectral data by tensoring with a line bundle. 

\subsection*{Further context} Our results complement many well established other connections between dimer models and mirror symmetry. The Legendrian $\Lambda_{\stackfan}$ can be identified with a skeleton of a generic fiber of the Hori-Vafa potential $W \in \bC[T_M]$ \cite{HV}, where $T_M$ is the dual torus of $T_N \subset \cX_{\stackfan}$ \cite{RSTZ,GS17,Zho18b}. This is a Laurent polynomial whose Newton polygon has vertices on the rays of $\stackfan$. 

Many aspects of mirror symmetry for $\cX_{\stackfan}$ and the total space $Y$ of its anticanonical bundle can be conveniently described in terms of a consistent bipartite graph $\Gamma^\vee$ for which $W^{-1}(0) \subset T_M$ is a spectral curve. The derived category $\Coh(Y)$ of coherent sheaves on $Y$ is equivalent to the derived category of modules for a Jacobian algebra $J$ of the dual quiver of $\Gamma^\vee$, while $\cX_{\stackfan}$ itself is derived equivalent to the module category of a subalgebra of $J$ \cite{HHV,IU09}. Mirror symmetrically, $\Gamma^\vee$ can be understood as encoding the intersection pattern of a collection of Lagrangian three-spheres  in the mirror $\{ W = uv \} \subset T_M \times \bC^2$ of $Y$, and $J$ as encoding relations in its Fukaya category \cite{FHKV,FU10}. This summary only scratches the surface of an extensively developed circle of ideas: an incomplete sampling of references includes \cite{HK,ORV,FHVWK,FHMSVW,HV07,FV,UY,Sze,Boc,BM,Dav,MR,Bro,NN,Nag}. Note that in the present paper, while we are also interested in the $B$-side of mirror symmetry for $\cX_{\stackfan}$, it is Laurent polynomials on $T_N \subset \cX_{\stackfan}$ rather than on the dual torus $T_M$ which play the leading role -- these have Newton polygons with edges normal to the rays of $\stackfan$, as opposed to vertices on these rays.

We also note that the combinatorial construction of Lagrangians relevant to mirror symmetry has been the topic of other recent and ongoing works, see \cite{Mik,Mat18a,Mat18b,Hic18}. Here we apply the construction of \cite{STWZ} in a similar spirit, but with a tropical coamoeba roughly in the role played by a tropical amoeba in the cited works. Results similar to our Proposition \ref{prop:isotopyintro} were also recently obtained in \cite{Han18}, but with the role of Legendrian isotopies replaced monodromies of coefficients in Landau-Ginzburg potentials. It is plausible to us that brane brick models \cite{FLS16a,FLS16b,FLSV} (see also \cite{FU14}) and their generalizations provide a combinatorial framework around which various aspects of the present work could be extended to higher dimensions. 

Finally, it is an elementary check, see for example \cite{FHKV,GK}, that the conjugate Lagrangians~$L$ and their mirror smooth spectral curves $C$ all have the same topological type.  One expects an explanation for this in terms of hyperk\"ahler rotation.  Indeed, the spectral curves are Lagrangian with respect to the holomorphic symplectic form $d \log(x) \wedge d \log(y)$, which is compatible with the flat hyperk\"ahler metric associated to any definite inner product on $\bZ^2$.  When $\bZ^2$ is the square lattice and the spectral curve is a Harnack curve, the description of $C$ via the Ronkin function \cite{KO} can be used to see that, in one of its symplectic structures, a natural symplectomorphism from $T_N$ to $T^* T^2$ carries $C$ to a Lagrangian asymptotic to a Legendrian link isotopic to~$\Lambda_{\stackfano}$. 

\subsection*{Organization} The structure of the paper is as follows. In Section~\ref{sec:CCC} we review the coherent-constructible correspondence. We prove a general formula computing the restriction of a coherent sheaf on a toric stack to the open torus orbit in terms of the constructible side of the $CCC$. In Section~\ref{sec:Legendrians} we study Legendrian links arising from toric stacks of dimension two and prove that a certain consistency condition on a Legendrian link $\Lambda \subset T^\infty T^2$ is a necessary and sufficient condition for it to be (cusplessly) Legendrian isotopic to one arising from a toric DM stack. In Section~\ref{sec:bipartite} we review how such Legendrians appear in the theory of dimer models. In Section~\ref{sec:Kasteleyn} we prove our main theorem, and in Section~\ref{sec:isotopies} establish a mirror relationship between Legendrian isotopies and tensor products with line bundles. Finally, in Section~\ref{sec:satellites} we characterize the mirror operation of pushing forward to a coarse moduli space as arising from the action of a specific Legendrian degeneration on constructible sheaves, clarifying the relationship between the present work and related ones in which stacks do not appear. 

\subsection*{Notation}
\label{subsec:notation}

Throughout, we fix a coefficient field $\coeffs$.  Some of the references we rely on assume for simplicity that $\coeffs$ is the field of complex numbers, though this hypothesis does not play any role in our constructions and arguments.  In fact with a little more effort one could replace $\coeffs$ with a more general ring. We write $\Sh(M)$ for the unbounded dg derived category of sheaves of $\coeffs$-vector spaces on a manifold $M$ with possibly nonempty boundary; we refer to an object of $\Sh(M)$ simply as a sheaf. We write $\Sh^c(M)$ for the subcategory of sheaves with constructible cohomology with respect to some Whitney stratification. Given a sheaf $\cF \in \Sh(M)$ we write $SS(\cF) \subset T^*M$ for the microsupport or singular support of $\cF$. Throughout all functors will be assumed derived unless otherwise stated. 

Given a conic Lagrangian subset $L \subset T^*M$ we write $\Sh_L(M) \subset \Sh(M)$ for the full subcategory of sheaves with microsupport contained in $L$, similarly for $\Sh^c_L(M)$. We additionally have the subcategory $\Sh^w_L(M) \subset \Sh_L(M)$ of compact objects, also called wrapped constructible sheaves \cite{Nad16}. Occasionally we cite basic results from \cite{KS94} which are stated for $\Sh^c_L(M)$ or $\Sh^c(M)$ but known to also hold for $\Sh_L(M)$ or $\Sh(M)$; we refer to \cite[Sec.~2]{JT17} or \cite{RS18} for a general discussion, citing \cite{KS94} without comment in the text.  We write $\Loc_1(M) \subset \Loc(M) := \Sh_{M}(M)$ for the subcategory of local systems with rank one stalks concentrated in degree zero, and somewhat abusively also for the moduli space of such objects (which is an algebraic torus when $M$ is a compact torus). 

We write $T^\infty M$ for the cosphere bundle of $M$, which we view as the fiberwise boundary of the fiberwise spherical compactification of $T^* M$. Given a subset $\Lambda \subset T^\infty M$, we write $\Sh_\Lambda(M) \subset \Sh(M)$ for the full subcategory of sheaves with microsupport contained in the union of the zero section and the cone over $\Lambda$, similarly for $\Sh^c_\Lambda(M)$, $\Sh^w_\Lambda(M)$. We write $SS^\infty(\cF) \subset T^\infty M$ for the asymptotic microsupport of $\cF$, i.e. the intersection with $T^\infty M$ of the closure of $SS(\cF)$ in the fiberwise spherical compactification of $T^*M$. 

Given a scheme or algebraic stack $Y$ we write $\Coh(Y)$ for the bounded dg derived category of coherent sheaves on $Y$, $\IndCoh(Y)$ for the ind-completion thereof, $\QCoh(Y)$ for the unbounded dg derived category of quasicoherent sheaves, and $\Perf(Y)$ for the subcategory of perfect complexes. 

\vskip.2in
\noindent{\bf Acknowldgements.}
We are grateful to Peng Zhou for discussions about singular support of degenerations. We also warmly thank Vivek Shende for the collaboration \cite{STWZ}, on which this project builds. D.T. is supported by NSF grant DMS-1510444.
H.W. was supported by NSF grant DMS-1801969 and NSF Postdoctoral Research Fellowship DMS-1502845.
E.Z. has been supported by NSF grants DMS-1406024 and DMS-1708503.

\section{The coherent-constructible correspondence}\label{sec:CCC}
\label{subsec:rttot}

In this section we explain how to compute the restriction of a coherent sheaf on a toric variety to its open torus orbit in terms of the constructible side of the CCC. We begin by fixing our notation and reviewing the statement of the CCC.

We fix dual lattices $M$ and $N$ of rank $n$, setting $M_{\bR} := M \otimes \bR$, $N_{\bR} := N \otimes \bR$.  We define a compact torus and a dual algebraic torus by 
\begin{equation}
\label{eq:TnTN-notation}
T^n := M_{\bR}/M, \qquad T_N := N \otimes \coeffs^{\times}
\end{equation}
and write $p:M_{\bR} \to T^n$ for the universal covering homomorphism.

Given a fan $\Sigma$ in $N_\bR$ we write $X_\Sigma$ for the associated toric partial compactification of $T_N$. We also consider extended data $\stackfan = (\Sigma, \widehat{\Sigma}, \beta)$, where $\widehat{\Sigma}$ is a fan in an auxiliary lattice $\widehat{N}$ 
and $\beta: \widehat{N} \to N$ is a homomorphism with finite cokernel and which induces a combinatorial equivalence between $\widehat{\Sigma}$ and $\Sigma$. We refer to $\stackfan$ as a stacky fan, though several variants appear in the literature \cite{BCS05,GS15,Tyo12}. Associated to $\stackfan$ is a toric DM stack $\cX_{\stackfan}$ which has no nontrivial stabilizers on the open subspace $T_N \subset \cX_{\stackfan}$ and whose coarse moduli space is $X_\Sigma$ (see e.g. \cite[\S 5]{Kuw16}). When $\beta$ is an isomorphism $\cX_{\stackfan}$ has no nontrivial stabilizers at all and coincides with the variety $X_\Sigma$.

\begin{notation}
\label{not:cohcoh}
We will be most interested in the following categories of sheaves on $\cX_{\stackfan}$.
\begin{enumerate}
	\item $\Perfprop(\cX_{\stackfan})$, the dg category of perfect complexes with proper support,
	\item $\Coh(\cX_{\stackfan})$, the bounded dg derived category of coherent sheaves.
\end{enumerate}
\end{notation}

To $\stackfan$ we also associate a conic Lagrangian $L_{\stackfan}$ in $T^* T^n \cong M_\bR/M \times N_\bR$ as follows. By assumption, $\beta$ induces a correspondence between the cones $\hat{\sigma} \in \widehat{\Sigma}$ and the cones $\sigma \in \Sigma$. We set $N_\sigma = \beta(\widehat{N} \cap \mathrm{span}(\hat{\sigma}))$ and $M_\sigma = \Hom(N_\sigma, \bZ)$. Given $\chi \in M_\sigma$ we set
$$ \sigma_\chi^\perp := \{ m \in M_\bR | \langle m, s \rangle = \langle \chi, s \rangle \text{ for any } s \in \beta(\widehat{N} \cap \hat{\sigma}) \}. $$
Each $\sigma_\chi^\perp$ is a translate of $\sigma^\perp := \sigma_0^\perp$ and only depends on the image of $\chi$ in the cokernel of the natural map $M \to M_\sigma$. We can now define
\begin{equation}
\label{eq:LSigma}
L_{\stackfan} := \bigcup_{\sigma \in \Sigma} \bigcup_{\chi \in M_\sigma} p(\sigma_\chi^\perp) \times (-\sigma).
\end{equation}
Note that when $\beta$ is an isomorphism, hence $\cX_{\stackfan} \cong X_\Sigma$, we have $\sigma_\chi^\perp = \sigma^\perp$ for all $\chi$. In this case we may simply write $L_\Sigma$ for $L_{\stackfan}$.

Since $L_{\stackfan}$ is conic it defines a Legendrian $\Lambda_{\stackfan}$ in the cosphere bundle $T^\infty T^n$, the quotient by $\bR_+$ of the complement of the zero section in $T^* T^n$. We write $T^\infty T^n,$ as we view the cosphere bundle as the boundary of the fiberwise spherical compactification of $T^*T^n$ (in particular, we do not fix a choice of contact form on $T^\infty T^n$, merely its contact distribution).

\begin{figure}
	\begin{tikzpicture}
	\newcommand*{\edgelen}{4}; \newcommand*{\vertrad}{.12}; \newcommand*{\crad}{.05}
	\newcommand*{\gspace}{.9}; \newcommand*{\angdelta}{22.5};
	\node (b) [matrix] at (0,0) {
		\coordinate (tl) at (0,\edgelen); \coordinate (tr) at (\edgelen,\edgelen);
		\coordinate (bl) at (0,0); \coordinate (br) at (\edgelen,0);
		\foreach \a/\b in {tl/bl, tr/br, br/tl} {\draw[asdstyle,righthairs] (\a) to (\b);}
		\foreach \a/\b in {tr/tl, br/bl} {\draw[asdstyle,lefthairs] (\a) to (\b);}
		\foreach \c in {tl, tr, bl, br} \foreach \ang in {0,...,7}
			{\draw[blue] ($(\c)+(\ang*\angdelta:\vertrad)$) to ($(\c)+(\ang*\angdelta+180:\vertrad)$);}
	\\};
	\node (c) [matrix] at (5.3,0) {
		\coordinate (tl) at (0,\edgelen); \coordinate (tr) at (\edgelen,\edgelen);
		\coordinate (bl) at (0,0); \coordinate (br) at (\edgelen,0); 
		\coordinate (tm) at (\edgelen/2,\edgelen); \coordinate (bm) at (\edgelen/2,0);
		\coordinate (mm) at (\edgelen/2,\edgelen/2);
		\foreach \a/\b in {tl/bl, tr/br} {\draw[asdstyle,righthairs] (\a) to (\b);}
		\foreach \a/\b in {tr/tl, br/bl, bm/tm, tl/br} {\draw[asdstyle,lefthairs] (\a) to (\b);}
		\foreach \c in {tl, tr, bl, br} \foreach \ang in {0,...,7} 
			{\draw[blue] ($(\c)+(\ang*\angdelta:\vertrad)$) to ($(\c)+(\ang*\angdelta+180:\vertrad)$);}
		\foreach \c in {tm, bm} \foreach \ang in {8,...,12} 
			{\draw[blue,thick] ($(\c)+(\ang*\angdelta:\vertrad+.02)$) to (\c);}
		\foreach \ang in {2,...,8} 
			{\draw[blue,thick] ($(mm)+(\ang*\angdelta:\vertrad+.02)$) to (mm);}
	\\};
	\node (a) [matrix] at (-5.3,0) {
		\foreach \x in {0,...,4} \foreach \y in {0,...,4}
			{\fill (\x*\gspace,\y*\gspace) circle (\crad);}
		\foreach \x/\y in {4.4/2,2/4.4,-.3/-.3}
			{\draw[-stealth',thick] (2*\gspace,2*\gspace) to (\x*\gspace,\y*\gspace);}
		\foreach \x/\y in {4/2,2/3,1/1}
			{\fill[red] (\x*\gspace,\y*\gspace) circle (\crad*1.5);}
		\node (e1) at (2.7*\gspace,3.4*\gspace) {$\beta(e_1)$};
		\node (e2) at (4*\gspace,1.6*\gspace) {$\beta(e_2)$};
		\node (e3) at (.7*\gspace,1.5*\gspace) {$\beta(e_3)$};
	\\};
	\node (atxt) at (-5.3, -2.7) {$\Sigma \subset N_\bR$};
	\node (btxt) at (0, -2.7) {$\Lambda_\Sigma \subset T^\infty T^2$};
	\node (ctxt) at (5.3, -2.7) {$\Lambda_{\stackfan} \subset T^\infty T^2$};
	\end{tikzpicture}
	\caption{The Legendrians in $T^\infty T^2$ associated to $X_\Sigma \cong \bP^2$ and a stack $\cX_{\stackfan}$ whose coarse moduli space is $\bP^2$.}
\end{figure}
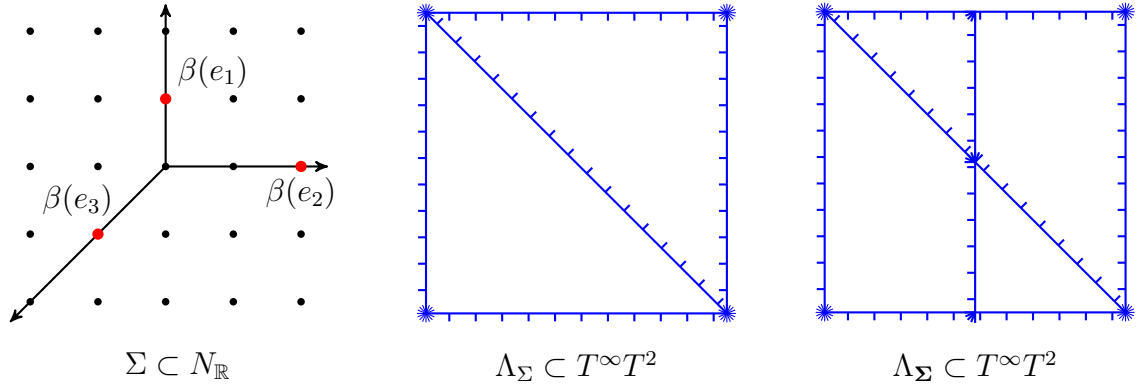\label{fig:stackyP2}

\begin{example}
	When $M \cong \bZ^2$ and $\Sigma$ is the complete fan with three rays generated respectively by $(1,0)$, $(0,1)$, and $(-1,-1)$, then $X_\Sigma \cong \bP^2$. In Figure \ref{fig:stackyP2} we depict the Legendrian $\Lambda_\Sigma \subset T^\infty T^2$, which is the union of the cocircle above $p(0)$ and three circles which project to the geodesics $\{p(\sigma^\perp)\}_{\sigma \in \Sigma(1)}$ (where $\Sigma(1)$ denotes the set of rays of $\Sigma$). Here and elsewhere we convey a Legendrian $\Lambda \subset T^\infty T^2$ by drawing its front projection $\pi(\Lambda) \subset T^2$ together with hairs indicating the codirections which comprise $\Lambda$ itself.

	 Now define $\beta: \widehat{N} \cong \bZ^3 \to N$ by setting 
	 $$ \beta(e_1) = (0,1), \quad \beta(e_2) = (2,0), \quad \beta(e_3) = (-1,-1),$$ 
	 and let $\widehat{\Sigma} \subset \widehat{N}_\bR$ be the (non-complete) fan formed by the walls of the positive octant. In this case $\cX_{\stackfan}$ is a stacky $\bP^2$ in which one component of the toric boundary has a generic $\bZ_2$ stabilizer. The Legendrian $\Lambda_{\stackfan}$ is the union of $\Lambda_\Sigma$, another circle which projects to a translate of $p(\sigma^\perp)$ (where $\sigma$ is the ray generated by $(1,0)$), and two intervals in the cocircles above $(1/2,0)$ and $(1/2,1/2)$.
\end{example} 

\begin{notation}
\label{not:shsh}
We will work with the following categories of sheaves on $T^n$ \cite{KS94,Nad16}.
\begin{enumerate}
	\item $\Sh^c_{\Lambda_{\stackfan}}(T^n)$, the dg category of constructible sheaves with asymptotic microsupport contained in $\Lambda_{\stackfan}$, 
	\item $\Sh^w_{\Lambda_{\stackfan}}(T^n)$, the dg category of wrapped constructible sheaves with asymptotic microsupport contained in $\Lambda_{\stackfan}$.
\end{enumerate}
\end{notation}
Wrapped constructible sheaves are by definition the compact objects in the unbounded dg derived category $\Sh_{\Lambda_{\stackfan}}(T^n)$ of sheaves with prescribed microsupport but no further size restriction; in particular the stalks of a wrapped constructible sheaf $\cF$ need not be finite-dimensional. 
Convolution on $T^n$ induces symmetric monoidal structures on  $\Sh^c_{\Lambda_{\stackfan}}(T^n)$ and $\Sh_{\Lambda_{\stackfan}}(T^n)$, and $\Sh^w_{\Lambda_{\stackfan}}(T^n) \subset \Sh_{\Lambda_{\stackfan}}(T^n)$ is closed under convolution when $\cX_{\stackfan}$ is smooth. 

The following result of Kuwagaki builds on and complements the results of \cite{Bon06,FLTZ,FLTZ3,Tre10,SS14,Kuw17,Zho17}, which taken together categorify the familiar relation between polytopes and line bundles on toric varieties 
\begin{theorem}[\cite{Kuw16}]\label{thm:CCC}
	For any stacky fan $\stackfan$ there is a commutative diagram
	\[
	\hspace*{-50pt}
	\begin{tikzpicture}
	[thick,>=\arrtip]
	\node (a) at (-1.8,0) {$\CCC_{\cX_{\stackfan}}:$};
	\node (a) at (0,0) {$\Coh(\cX_{\stackfan})$};
	\node (b) at (4,0) {$\Sh^w_{\Lambda_{\stackfan}}(T^n)$};
	\node (c) at (0,-1.5) {$\Perfprop(\cX_{\stackfan})$};
	\node (d) at (4,-1.5) {$\Sh^c_{\Lambda_{\stackfan}}(T^n).$};
	\draw[->] (a) to node[above] {$\sim$} (b);
	\draw[right hook->] (d) to (b);
	\draw[right hook->] (c) to (a);
	\draw[->] (c) to node[above] {$\sim$} (d);
	\end{tikzpicture}
	\]
	where the top and bottom functors are equivalences.
	The bottom is a monoidal equivalence with respect to the tensor product on $\cX_{\stackfan}$ and the convolution product on $T^n$, as is the top when $\cX_{\stackfan}$ is smooth. 
\end{theorem}

\begin{example}\label{ex:surfaceCCC}
	Suppose that $M \cong \bZ^2$ and that $\Sigma$ is a complete fan, in which case $X_\Sigma$ is a complete toric surface. We have the following key examples of the coherent-constructible correspondence:
	\begin{enumerate}
		\item Suppose $\cF \in \Coh(X_\Sigma)$ is the structure sheaf of a point $(x, y) \in T_N \subset X_\Sigma$. Then $\CCC_{X_\Sigma}(\cF)$ is a local system on $T^2$ with holonomies $x$ and $y$ around the given generators of $M \cong \pi_1(T^2)$, placed in degree $-2$. 
		\item Suppose $\cF \in \Coh(X_\Sigma)$ is an ample line bundle on $X_\Sigma$. Then there is a polygon $P \subset M_\bR$ such that $\CCC_{X_\Sigma}(\cF) \cong p_* i_! \omega_{P^\circ}$, where $\omega_{P^\circ}$ is the dualizing sheaf on the interior $P^\circ$ of $P$, $i: P^\circ \into M_\bR$ the inclusion, and $p: M_\bR \onto T^2$ the projection. The inward normal fan of $P$ is a coarsening of $\Sigma$.
	\end{enumerate}
\end{example}

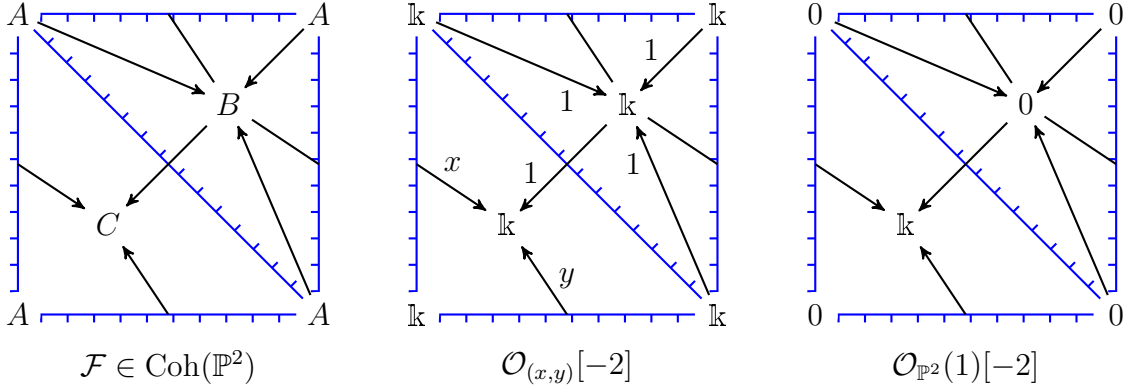
\begin{figure}
	\begin{tikzpicture}
	\newcommand*{\edgelen}{4}; \newcommand*{\wcircrad}{.3}
	\newcommand*{\xytrtri}{.7}; \newcommand*{\xybltri}{.3};
	\newcommand*{\ashort}{8}; \newcommand*{\bshort}{4};
	\node (a) [matrix] at (-5.3,0) {
		\coordinate (tl) at (0,\edgelen); \coordinate (tr) at (\edgelen,\edgelen);
		\coordinate (bl) at (0,0); \coordinate (br) at (\edgelen,0);
		\foreach \a/\b in {tl/bl, tr/br, br/tl} {\draw[asdstyle,righthairs] (\a) to (\b);}
		\foreach \a/\b in {tr/tl, br/bl} {\draw[asdstyle,lefthairs] (\a) to (\b);}
		\foreach \c in {tl, tr, bl, br} {
			\fill[white] (\c) circle (\wcircrad);
			\node at (\c) {$A$};}
		\node (trtri) at (\xytrtri*\edgelen, \xytrtri*\edgelen) {$B$};
		\node (bltri) at (\xybltri*\edgelen, \xybltri*\edgelen) {$C$};
		\draw[genmapstyle, shorten <=\ashort] (tl) to (trtri);
		\draw[genmapstyle, shorten <=\ashort] (br) to (trtri);
		\draw[genmapstyle, shorten <=\ashort, shorten >=-\bshort+1] (tr) to (trtri);
		\draw[thick] (trtri) to (\edgelen,.5*\edgelen); 
		\draw[genmapstyle] (0,.5*\edgelen) to (bltri);
		\draw[thick] (trtri) to (.5*\edgelen,\edgelen); 
		\draw[genmapstyle] (.5*\edgelen,0) to (bltri);
		\draw[genmapstyle, shorten >=-\bshort+2] (trtri) to (bltri);
		\\};
	\node (b) [matrix] at (0,0) {
		\coordinate (tl) at (0,\edgelen); \coordinate (tr) at (\edgelen,\edgelen);
		\coordinate (bl) at (0,0); \coordinate (br) at (\edgelen,0);
		\foreach \a/\b in {tl/bl, tr/br, br/tl} {\draw[asdstyle,righthairs] (\a) to (\b);}
		\foreach \a/\b in {tr/tl, br/bl} {\draw[asdstyle,lefthairs] (\a) to (\b);}
		\foreach \c in {tl, tr, bl, br} {
			\fill[white] (\c) circle (\wcircrad);
			\node at (\c) {$\coeffs$};}
		\node (trtri) at (\xytrtri*\edgelen, \xytrtri*\edgelen) {$\coeffs$};
		\node (bltri) at (\xybltri*\edgelen, \xybltri*\edgelen) {$\coeffs$};
		\draw[genmapstyle, shorten <=\ashort] (tl) to (trtri);
		\draw[genmapstyle, shorten <=\ashort] (br) to (trtri);
		\draw[genmapstyle, shorten <=\ashort, shorten >=-\bshort] (tr) to (trtri);
		\draw[thick] (trtri) to (\edgelen,.5*\edgelen); 
		\draw[genmapstyle] (0,.5*\edgelen) to (bltri);
		\draw[thick] (trtri) to (.5*\edgelen,\edgelen); 
		\draw[genmapstyle] (.5*\edgelen,0) to (bltri);
		\draw[genmapstyle, shorten >=-\bshort] (trtri) to (bltri);
		\node at (.12*\edgelen,.5*\edgelen) {$x$};
		\node at (.5*\edgelen,.12*\edgelen) {$y$};
		\foreach \x/\y in {.38/.47, .50/.71, .72/.5, .78/.88} {\node at (\x*\edgelen,\y*\edgelen) {$1$};}
		\\};
	\node (c) [matrix] at (5.3,0) {
		\coordinate (tl) at (0,\edgelen); \coordinate (tr) at (\edgelen,\edgelen);
		\coordinate (bl) at (0,0); \coordinate (br) at (\edgelen,0);
		\foreach \a/\b in {tl/bl, tr/br, br/tl} {\draw[asdstyle,righthairs] (\a) to (\b);}
		\foreach \a/\b in {tr/tl, br/bl} {\draw[asdstyle,lefthairs] (\a) to (\b);}
		\foreach \c in {tl, tr, bl, br} {
			\fill[white] (\c) circle (\wcircrad);
			\node at (\c) {$0$};}
		\node (trtri) at (\xytrtri*\edgelen, \xytrtri*\edgelen) {$0$};
		\node (bltri) at (\xybltri*\edgelen, \xybltri*\edgelen) {$\coeffs$};
		\draw[genmapstyle, shorten <=\ashort] (tl) to (trtri);
		\draw[genmapstyle, shorten <=\ashort] (br) to (trtri);
		\draw[genmapstyle, shorten <=\ashort, shorten >=-\bshort] (tr) to (trtri);
		\draw[thick] (trtri) to (\edgelen,.5*\edgelen); 
		\draw[genmapstyle] (0,.5*\edgelen) to (bltri);
		\draw[thick] (trtri) to (.5*\edgelen,\edgelen); 
		\draw[genmapstyle] (.5*\edgelen,0) to (bltri);
		\draw[genmapstyle, shorten >=-\bshort] (trtri) to (bltri);
		\\};
	\node (atxt) at (-5.3, -2.7) {$\cF \in \Coh(\bP^2)$};
	\node (btxt) at (0, -2.7) {$\cO_{(x,y)}[-2]$};
	\node (ctxt) at (5.3, -2.7) {$\cO_{\bP^2}(1)[-2]$};
	\end{tikzpicture}
	\caption{Constructible mirrors of some coherent sheaves on $\bP^2$. From left to right: a general object of $\Coh(\bP^2)$, a (shifted) skyscraper at a point $(x,y) \in \bA^2 \subset \bP^2$, and the (shifted) line bundle $\cO_{\bP^2}(1)[-2]$.}
\end{figure}\label{fig:CCC4P2}

\begin{example}
	When $\Sigma$ is the fan of $\bP^2$, an object of $\Sh_{\Lambda_\Sigma}(T^2)$ is determined by its stalks at $p(0)$ and a generic point in each component of $T^2 \smallsetminus \pi(\Lambda_\Sigma)$ --- where $\pi: T^\infty T^2 \onto T^2$ denotes the projection --- together with six generization maps satisfying obvious relations, see Figure \ref{fig:CCC4P2}. In this way the coherent-constructible correspondence recovers the Beilinson description of $\Coh(\bP^2)$ in terms of quiver representations. 
\end{example}

The following lemma will be used in Section \ref{sec:Kasteleyn}. Note that the identification $N \cong \Hom(M,\bZ)$ induces an identification of $\coeffs[T_N]$ with the group algebra of $M \cong \pi_1(T^n)$, and a corresponding equivalence of dg categories
\begin{equation}
\label{eq:shift-this}
\coeffs[T_N]\text{-mod} \cong \Loc(T^n)
\end{equation}
The coherent-constructible correspondence is the composition of this with the shift-by-$n$ functor $F \mapsto F[n]$ (a normalization prescribed by the requirement that the equivalence be monoidal).

\begin{lemma}\label{lem:openorbitmain}
	Let $\cF \in \Sh^w_{\Lambda_{\stackfan}}(T^n)$  be a wrapped constructible sheaf (as in \ref{not:shsh}). Then the restriction of $\CCC^{-1}_{\cX_{\stackfan}}(\cF) \in \Coh(\cX_{\stackfan})$ to $T_N$ is isomorphic as a $\coeffs[T_N]$-module to $\Gamma_c(p^*\cF)$ with the natural action of $\pi_1(T^n)$ by deck transformations. 
\end{lemma}

\begin{proof}
Let $\omega_{M_{\bR}}$ denote the dualizing complex on $M_{\bR}$, so that an orientation of $M_{\bR}$ gives an isomorphism $\omega_{M_{\bR}} \cong \coeffs_{M_{\bR}}[n]$.  We recall that for $\cG \in \Coh(\cX_{\stackfan})$ we have
$$ \CCC_{T_N}(i^*_{T_N} \cG) \cong \CCC_{\cX_{\stackfan}}(\cG) \star p_!(\omega_{M_{\bR}}) $$
in $\Sh_{T^n}(T^n) \cong \Loc(T_N)$ (\cite[Th. 3.8]{FLTZ}, \cite[Cor. 12.8]{Kuw16}). That is, the coherent-constructible correspondence intertwines restriction to $T_N$ and convolution with a shifted free local system on $T^n$.  For any sheaf $\cF \in \Sh(T^n)$, a base-change argument (which we give below) shows that the stalk of the local system $\cF \star p_! \omega_{M_{\bR}}$ at zero is naturally identified with $\Gamma_c((p^* \cF)\otimes \omega_{M_{\bR}})$.  After choosing an orientation of $M_{\bR}$, the projection formula \cite[Prop. 2.6.6]{KS94} gives
\[
\Gamma_c((p^* \cF) \otimes \omega_{M_{\bR}}) = \Gamma_c((p^* \cF) \otimes \pi^* \coeffs[n]) \cong \Gamma_c(p^*\cF) \otimes \coeffs[n] = \Gamma_c(p^* \cF)[n]
\]
Since $\CCC_{T_N}$ is the shift-by-$n$ of the equivalence of categories induced by the isomorphism of rings $\coeffs[T_N] \cong \coeffs[\pi_1(T^n)]$, we may prove the Proposition by giving an isomorphism 
\begin{equation}
\label{eq:base-change}
\Gamma_c((p^* \cF) \otimes \omega_{M_{\bR}}) \cong (\cF \star p_! \omega_{M_{\bR}})\vert_0
\end{equation}
and computing the deck group action.

The left half of the diagram below shows a Cartesian square of continuous maps between locally compact spaces.  The base-change isomorphism associated to this diagram is \eqref{eq:base-change}, indicated in the right-half of the diagram.
\[
\begin{tikzpicture}
[thick,>=\arrtip,every text node part/.style={align=center}]
\node[matrix] (left) at (0,0) {
	\node (a) at (0,0) {$M_\bR$};
	\node (b) at (4,0) {$T^n \times M_\bR$};
	\node (c) at (0,-3) {$\{0\}$};
	\node (e) at (4,-1.5) {$T^n \times T^n$};
	\node (d) at (4,-3) {$T^n$};
	\draw[->] (a) to node[above] {$t \mapsto (p(t), -t)$} (b);
	\draw[->] (b) to node[right] {$id \times p$} (e); \draw[->] (e) to node[right] {$m$} (d);
	\draw[->] (a) to (c); \draw[->] (c) to (d);\\};
\node[matrix] (right) at (7,-.3) {
	\node (a) at (0,0) {$p^*\cF \otimes \omega_{M_{\bR}}$};
	\node (b) at (4,0) {$\cF \boxtimes \omega_{M_\bR}$};
	\node (c) at (0,-3) {$\Gamma_c(p^*\cF \otimes \omega_{M_{\bR}})$ \\ $\cong (\cF \star p_! \omega_{M_{\bR}})|_0$};
	\node (e) at (4,-1.5) {$\cF \boxtimes p_! \omega_{M_\bR}$.};
	\node (d) at (4,-3) {$\cF \star p_! \omega_{M_{\bR}}$.};
	\draw[|->] (b) to (a); \draw[|->] (b) to (e); \draw[|->] (e) to (d);
	\draw[|->] (a) to (c); \draw[|->] (d) to (c);\\};
\end{tikzpicture}
\]

For $m \in M$, let $\tau_m:\M_{\bR} \to M_{\bR}:t \mapsto t+m$ denote the translation by $m$.  Since $p \circ \tau_m = p$ we have a natural isomorphism $\tau_m^* p^* \cF \cong p^* \cF$, and an adjoint isomorphism
$p^* \cF \cong \tau_{m,*} p^* \cF = \tau_{m,!} p^* \cF$.  We also have an isomorphism
$\omega_{M_{\bR}} \stackrel{\sim}{\to} \tau_{m,!} \omega_{M_{\bR}}$.
 natural isomorphism $\tau_m^! \omega_{M_{\bR}} \cong \omega_{M_\bR}$ and an adjoint isomorphism $\omega_{M_{\bR}} \xrightarrow{\sim} \tau_{m,!} \omega_{M_{\bR}}$, giving the equivariant structure on $\omega_{M_{\bR}}$.  Together these define a map
\begin{equation}
\label{eq:deck-act}
\Gamma_c(p^* \cF \otimes \omega_{M_{\bR}}) \to \Gamma_c(\tau_{m,!} (p^* \cF \otimes M_{\bR})) = \Gamma_c(p^* \cF \otimes \omega_{M_{\bR}})
\end{equation}
(using $\Gamma_c \circ \tau_{m,!} = \Gamma_c$) which give the deck action on $\Gamma_c(p^* \cF \otimes \omega_{M_{\bR}})$.  Since $(p(t+m),-(t+m)) = (p(t),-t-m)$, the top map in the diagram intertwines $\tau_m$ with $\mathrm{id} \times \tau_{-m}$, and 
$
p^* \cF \otimes \omega_{M_{\bR}} \to \tau_{m,!}( p^* \cF\otimes \omega_{M_{\bR}})
$
is pulled back from 
\[
\cF \boxtimes \omega_{M_{\bR}} \xrightarrow{\sim} (\mathrm{id} \times \tau_{-m})_! (\cF \boxtimes \omega_{M_{\bR}}).
\]
Writing $\tau_{-m}:p_! \omega_{M_{\bR}} \to p_! \omega_{M_{\bR}}$ for the image of $\omega_{M_{\bR}} \to \tau_{-m,!} \omega_{M_{\bR}}$ under $p_!$, (making use of $p_! \tau_{-m,!} = p_!$)
it follows that \eqref{eq:deck-act} coincides with 
\[
(\mathrm{id} \star \tau_{-m})\vert_0:(\cF \star p_! \omega_{M_{\bR}})\vert_0 \to (\cF \star p_! \omega_{M_{\bR}})\vert_0
\]
under the base-change isomorphism.
\end{proof}

\section{From fans to Legendrian links}\label{sec:Legendrians}
\label{sec:from-fans-to}

We now specialize to the case $M \cong \bZ^2$ and fix a complete stacky fan $\stackfan$. In Example~\ref{ex:surfaceCCC} we saw how to describe mirrors of skyscrapers sheaves and line bundles on $\cX_{\stackfan}$ in terms of the coherent-constructible correspondence. The remainder of the paper can be understood as explaining how the mirrors of coherent sheaves supported on hypersurfaces in $\cX_{\stackfan}$ may be described using dimer models. 

We first note that a generic hypersurface $C$ in the underlying variety $X_\Sigma$ will intersect the toric boundary $D := X_\Sigma \smallsetminus T_N$ only along the smooth part of $D$. In other words, $C$ will lie in the smooth, non-complete toric surface $X_{\Sigmao}$ defined by the subfan $\Sigmao$ of $\Sigma$ formed by its 1-dimensional cones (rays) together with the origin. Set-theoretically $X_{\Sigmao}$ is the union of $T_N$ and a copy of $\coeffs^\times$ for each ray of $\Sigma$. Letting $\stackfano = (\Sigmao, \whSigmao, \beta)$, the same is true of $\cX_{\stackfano}$, the difference between $\cX_{\stackfano}$ and $X_{\Sigmao}$ being that in the former the $\coeffs^\times$ attached to the ray $\sigma$ has the cyclic group $C_\sigma := (N \cap \mathrm{span}(\sigma))/(\beta(\widehat{N}) \cap \mathrm{span}(\sigma))$ as a stabilizer. 

The Legendrian $\Lambda_{\Sigmao} \subset T^\infty T^2$ is a link with a connected component for each ray of $\Sigma$ (recall again that we set $T^2 := M_\bR/M$). These project onto geodesics passing through the origin in $T^2$, and the singular Legendrian $\Lambda_\Sigma$ is the union of $\Lambda_{\Sigmao}$ with the fiber of $T^\infty T^2$ at the origin. The Legendrian $\Lambda_{\Sigmao}$ is contained in $\Lambda_{\stackfano}$, which is also a link and whose remaining components project to geodesics not passing through the origin. While the components of $\Lambda_{\Sigmao}$ are mutually nonisotopic, $\Lambda_{\stackfano}$ has $|C_\sigma|$ pairwise isotopic components associated to each $\sigma \in \Sigma$. 

An important feature of microlocal sheaf theory is that the action of contact isotopies on the cosphere bundle of a manifold $M$ quantizes to an action on its sheaf category $\Sh(M)$ \cite{GKS}. By a contact isotopy we mean a family $\{\phi_t\}_{t \in I}$ of contactomorphisms $\phi_t: T^\infty M \to T^\infty M$ which assemble into a smooth map $T^\infty M \times I \to T^\infty M$, and such that $\phi_0$ is the identity (here $I := [0,1]$). On a practical level, this means that in studying sheaves microsupported on $\Lambda_{\stackfano}$ we have the freedom to isotope $\Lambda_{\stackfano}$ as we wish and instead study sheaves microsupported on the new Legendrian. More formally we have the following result, where $\dot{T}^*M \subset T^*M$ is the complement of the zero section. 

\begin{theorem}\label{thm:GKS} \cite{GKS}
	Let $\{\phi_t\}_{t \in I}$ be a contact isotopy of $T^\infty M$. Then there is a unique locally bounded sheaf $K_{\phi_I} \in \Sh^c(I \times M \times M)$ such that
	\begin{enumerate}
		\item the intersection of $\{t\} \times T^\infty M \times T^\infty M$ with the projection of $SS(K_\Phi) \cap I \times \dot{T}^*M \times \dot{T}^*M$ to $I \times T^\infty M \times T^\infty M$ is equal to the graph of $\phi_t$, 
		\item the restriction of $K_{\phi_I}$ to $\{0\} \times M \times M$ is the constant sheaf of the diagonal.
	\end{enumerate}
	Convolution with the restriction $K_{\phi_t}$ of $K_{\phi_I}$ to $\{t\} \times M \times M$ defines an autoequivalence of $\Sh(M)$ such that $SS(K_{\phi_t}(\cF)) = \phi_t(SS(\cF))$ for any $\cF \in \Sh(M)$ and $t \in I$. 
\end{theorem}

It follows that for any Legendrian $\Lambda \subset T^\infty M$ convolution with $K_{\phi_t}$ restricts to an equivalence $\Sh_{\Lambda}(M) \congto \Sh_{\phi_t(\Lambda)}(M)$. In fact, by \cite[Prop. 3.12]{GKS} this restriction only depends on the Legendrian isotopy $\{\phi_t(\Lambda)\}_{t \in I}$ and can be alternatively described as follows, where $M_t := M \times \{t\} \subset M \times I$ (see also \cite[Theorem 3.1]{Zho18}, \cite[Sec. 2.11]{JT17}). Here given a smooth family $\{\Lambda_t\}_{t \in I}$ of Legendrians in $T^\infty M$ we write $\Lambda_I \subset T^\infty (M \times I)$ for the unique Legendrian whose projection to $(T^\infty M) \times I$ is the given family (see e.g. \cite[Sec. A.2]{GKS}). 

\begin{corollary}\label{cor:LegendrianGKS}
	\cite{GKS} Let $\{\Lambda_t\}_{t \in I}$ be a Legendrian isotopy in $T^\infty M$. Then for any $t \in I$ the restriction functor $i_{M_t}^*: \Sh_{\Lambda_I} (M\times I) \to \Sh_{\Lambda_t}(M)$ is an equivalence. The composition $i_{M_t}^* \circ (i_{M_0}^*)^{-1}: \Sh_{\Lambda_0}(M) \congto \Sh_{\Lambda_t}(M)$ is isomorphic to convolution with the GKS kernel $K_{\phi_t}$ for any contact isotopy $\{\phi_t\}_{t \in I}$ such that $\Lambda_t = \phi_t(\Lambda_0)$ for all $t \in I$.
\end{corollary}

With this Corollary in hand we will write $K_{\{\Lambda_t\}}: \Sh_{\Lambda_0}(M) \congto \Sh_{\Lambda_1}(M)$ for the equivalence associated to a Legendrian isotopy $\{\Lambda_t\}_{t \in I}$. Following the discussion preceding Theorem \ref{thm:GKS}, we would now like to characterize Legendrian links in $T^\infty T^2$ which are Legendrian isotopic to some $\Lambda_{\stackfano}$. If we restrict our attention to isotopies which do not create cusps in the front projection, we will see in Proposition \ref{prop:isotopies} that the resulting links are exactly those which satisfy the following condition. 

\begin{definition}\label{def:consistent}
	A Legendrian link $\La \subset T^\infty T^2$ is \newword{consistent} if 
	\begin{itemize}
		\item its projection $\La \to T^2$ is an immersion (in particular, its image has no cusps),
		\item no component has a homotopically trivial projection,
		\item no component has a projection whose lift to the universal cover of $T^2$ has a self-intersection,
		\item no pair of components have projections whose lifts to the universal cover bound a parallel bigon.
	\end{itemize}
\end{definition}

Here and elsewhere we say a bigon formed by two strands of the front projection of a Legendrian is anti-parallel (below left) if the induced co-orientations of the edges are both inward or both outward, and is parallel (below right) otherwise. We allow bigons which have other strands (or different parts of the same strands) meeting their interior, i.e. the pictures below may be embedded into a larger immersed curve but would still be referred to as bigons.

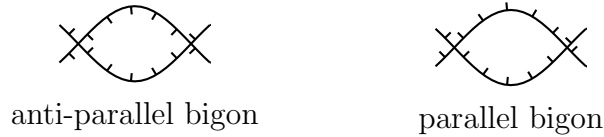
\[
\begin{tikzpicture}
\node[matrix] (a) at (0,0) {
	\draw[lefthairs, thick, domain=60:300 ] plot (\x/120,{.5*cos(\x)});
	\draw[righthairs, thick, domain=60:300 ] plot (\x/120,{-.5*cos(\x)});\\};
\node[matrix] (b) at (5,0) {
	\draw[lefthairs, thick, domain=60:300 ] plot (\x/120,{.5*cos(\x)});
	\draw[lefthairs, thick, domain=60:300 ] plot (\x/120,{-.5*cos(\x)});\\};
\node (atxt) at ($(a.south)+(0,-.3)$) {anti-parallel bigon};
\node (atxt) at ($(b.south)+(0,-.3)$) {parallel bigon};
\end{tikzpicture}
\]

As discussed in Section \ref{sec:bipartite}, the term \textit{consistent} is adapted from the literature on dimer models. It is clear $\Lambda_{\stackfano}$ is consistent, and that any Legendrian isotopy which preserves the first condition in Definition \ref{def:consistent} preserves the rest. To argue that all consistent Legendrians are isotopic to ones associated to stacky fans we will apply the following useful property, which ensures that isotopies of individual components of a consistent Legendrian extend to well-behaved isotopies of the entire link. Recall that we write $\pi$ for the projection $T^\infty T^2 \to T^2$.

\begin{lemma}\label{lem:bigonremoval}
	Let $\La \subset T^\infty T^2$ be a consistent Legendrian link, and $\La_i$ a component of $\La$. Let $\La'_i \subset T^\infty T^2$ be a Legendrian such that $\ga := \La_i \smallsetminus (\La_i \cap \La'_i)$ and $\ga' := \La'_i \smallsetminus (\La_i \cap \La'_i)$ are intervals whose projections form the boundary of an embedded bigon $B \subset S$, and such that the projections of $\ga'$ and $\La$ intersect at finitely many points. Then $\La'_i$ extends to a consistent Legendrian $\La' \subset T^\infty T^2$ which is the result of an isotopy $\La \to \La'$ that takes $\La_i$ onto $\La'_i$, is stationary outside a neighborhood of $B$, and moves $\La$ minimally in the following sense: any point of $\La$ whose projection lies on $\pi(\ga')$ is held fixed unless it is an endpoint of an interval $\ga'' \subset \La$ whose projection lies in $B$ and forms a parallel bigon with a subinterval of $\pi(\ga')$.
\end{lemma}

\begin{proof}
	We induct on the number of embedded bigons in $B$ with boundary formed by the projections of two intervals in $\La \cup \ga'$. By perturbing $\La$ we may assume that $B \smallsetminus \pi(\La)$ has finitely many components. An embedded bigon is a union of such components, so in particular there are finitely many.
	
	If there are no embedded bigons in $B$ other than $B$ itself, then the intersection of $\pi(\La)$ with $B$ consists of a collection of embedded arcs connecting its two edges, any two arcs intersecting at most once. Any isotopy $\pi(\La_i) \to \pi(\La'_i)$ that creates no tangencies to these arcs and is stationary on $\pi(\La_i \smallsetminus \ga)$ lifts to a Legendrian isotopy $\La_i \to \La'_i$ which avoids the other components of $\La$, hence extends to an isotopy $\La \to \La' := (\La \smallsetminus \ga) \cup \ga'$ whose result is consistent if $\La$ is.
	
	If on the other hand there are bigons properly contained in $B$, choose a bigon $B'$ which is minimal in the sense that no other bigons are properly contained in it. If one edge of $B'$ is a subinterval of $\pi(\ga')$, then similarly to the previous paragraph we can isotope the other edge to lie just outside of $B$ without creating tangencies to the rest of $\pi(\La)$. The Legendrian lift of this again extends to an isotopy of $\La$ whose result is consistent and whose projection has strictly fewer bigons contained in $B$.
	
	If both edges of $B'$ lie along $\pi(\La)$ then again there is an isotopy that carries one edge across $B'$ to lie just outside the other edge, and which creates no tangencies among strands of $\pi(\La)$ crossing the interior of $B'$. A tangency is created between the two edges, but since $\La$ is consistent the co-orientations of these edges are opposite. Thus this again lifts to a Legendrian isotopy that decreases the number of bigons and preserves consistency.
\end{proof}

\begin{proposition}\label{prop:isotopies}
	If a Legendrian link $\Lambda \subset T^\infty T^2$ is consistent then it can be Legendrian isotoped so that its front projection is a union of geodesics. In particular, it is isotopic to a Legendrian of the form $\Lambda_{\stackfano}$ for a unique complete stacky fan $\stackfan$.
\end{proposition}

\begin{proof}
	It suffices by induction to show there is an isotopy of $\La$ which carries an arbitrary component $\La_i$ with non-geodesic projection to a Legendrian $\La'_i$ with geodesic projection, and which leaves stationary all components whose projections are already geodesic. Note that since the projection of $\La_i$ is a homotopically nontrivial simple closed curve, it is clear that there exists a Legendrian $\La'_i$ with geodesic projection which is Legendrian isotopic to $\La_i$ (through an isotopy that is allowed to pass through other components of $\Lambda$) --- let us fix such a $\La'_i$,  which without loss of generality we may choose so that $\pi(\La'_i) \cap \pi(\La)$ is finite and $\pi(\La'_i) \cap \pi(\La_i)$ is nonempty. We now build the desired isotopy of $\Lambda$ inductively by composing a sequence of isotopies, each of which decreases the number of intersection points of $\pi(\La_i)$ and $\pi(\La'_i)$, until they are in a position where Lemma \ref{lem:bigonremoval} can be used to align them. 
	
	
	Consider first that since $\pi(\La'_i) \cap \pi(\La_i)$ is nonempty there must exist an embedded bigon $B \subset T^2$ whose boundary is the union of the projections of intervals $\ga \subset \La_i$ and $\ga' \subset \La'_i$, and which contains no smaller bigons of this kind. Since the signs of the intersections of $\pi(\La_i)$ with $\pi(\La'_i)$ sum to zero, we can find two of opposite signs which are adjacent in the natural order along $\La_i$ and let $\ga$ be the interval between these. Since $T^2 \smallsetminus \La'_i$ is an annulus, one of the two components of $T^2 \smallsetminus (\La'_i \cup \ga)$ is a bigon $B$. As two bigons of this kind are either disjoint or one is properly contained in the other, we can choose $B$ to contain no other such bigons. 
	
	Suppose that $\pi(\La'_i) \cap \pi(\La_i)$ consists of more than two points. By Lemma \ref{lem:bigonremoval} there is an isotopy of $\La$ which is supported in a neighborhood of $B$ and which carries $\ga$ so that its projection lies just outside the opposite edge of $B$ (so to be precise, we apply Lemma \ref{lem:bigonremoval} so that the interval $\gamma$ in its statement is slightly larger than what we call $\gamma$ here, as we want the endpoints of our $\gamma$ to be pushed off of $\pi(\La'_i)$). Moreover, during the induction in the proof of the Lemma -- wherein smaller bigons are erased by pulling one of their edges across the other one -- we can always avoid moving components of $\Lambda$ with geodesic projection. This follows since at most one edge of a bigon can be straight, and we get to choose which edge to move at each step in the induction.
	This isotopy clearly decreases the number of intersection points of $\pi(\La_i)$ and $\pi(\La'_i)$ by two. 
	
	Suppose now that $\pi(\La'_i) \cap \pi(\La_i)$ consists of two points. We proceed as above, except we choose the initial isotopy of $\La$ to carry $\ga$ so that its projection lies just outside the opposite edge of $B$ except at a single point, where $\pi(\ga)$ is tangent to $\pi(\La'_i)$. But now $\pi(\La_i)$ and $\pi(\La'_i)$ bound a single large bigon whose two corners wrap around $T^2$ and touch at this point. A final application of Lemma \ref{lem:bigonremoval} as above lets us isotope $\La$ onto $\La'_i$ without disturbing any components with geodesic projection. 
	
	Finally, it is straightforward to see that a Legendrian $\Lambda$ with geodesic front projection is isotopic to a suitable $\Lambda_{\stackfano}$. First, let $\Sigma$ be the complete fan whose rays are those which define the conormal lifts of the components of $\Lambda$. Now let $\widehat{N} = \bigoplus_{\rho \in \Sigma(1)} \bZ e_\rho$ and let $\widehat{\Sigma}$ be the complete fan in $\widehat{N}$ whose rays are $\{\bR_{\geq 0} e_\rho\}$. Finally, let $\beta$ take $e_\rho$ to the $m_\rho$th multiple of the generator of $N \cap \rho$, where $m_\rho$ is the number of components of $\Lambda$ which are lifted from their front projection by the conormal direction $\rho$.
\end{proof}

\section{Legendrian links from bipartite graphs}\label{sec:bipartite}

We now explain how interesting isotopy representatives of a consistent Legendrian link $\Lambda_{\stackfano} \subset T^\infty T^2$ may be obtained systematically from bipartite graphs.

Let $\Gamma \subset T^2$ be an embedded bipartite graph with vertices colored black and white.
The \newword{zig-zag paths} of $\Gamma$ are a collection of immersed curves determined up to isotopy by the following conditions: they lie in an open set that retracts onto $\Gamma$, their crossings all lie on edges of $\Gamma$ with a unique crossing on each edge, and these crossings are the only points where the zig-zags meet $\Gamma$.
We label the components of the complement of the zig-zag paths as white, black, or null according to whether they contain a white vertex, a black vertex, or no vertices.

\begin{definition} \cite[\S 4]{STWZ}
	The \newword{alternating Legendrian} $\Lambda_\Gamma$ associated to $\Gamma$ is the Legendrian lift of its zig-zag paths, co-oriented so that the boundaries of black and white regions are co-oriented inward and outward, respectively. 
\end{definition}

Alternating Legendrians are distinguished representatives of their Legendrian isotopy class in that $\La_\Gamma$ has a canonical exact embedded Lagrangian filling $L_\Gamma \subset T^*T^2$ \cite{STWZ}. This Lagrangian deformation retracts onto $\Gamma$ and its image in $T^2$ is the union of the black and white regions. Sheaf quantization of $L_\Gamma$ yields a fully faithful functor $\Loc_1(L_\Gamma) \into \Sh_{\La_\Gamma}(T^2)$ from rank one local systems on $L_\Gamma$ to sheaves whose microsupport at infinity is contained in $\Lambda_\Gamma$. This functor can be described Floer theoretically via the equivalence $\Sh(T^2) \cong \Fuk_{inf} T^*T^2$ of \cite{NZ,Nad09} or sheaf theoretically as in \cite{JT17,Gui12}. We refer to objects in the image of $\Loc_1(L_\Gamma) \to \Sh_{\La_\Gamma}(T^2)$ as \newword{alternating sheaves}.

 A sheaf $\cA \in \Sh_{\Lambda_\Gamma}(T^2)$ is alternating if and only if it fits into a triangle
\begin{equation}\label{eq:altpres}
i_* \coeffs_B \to i_! \coeffs_W[2] \to \cA \to i_* \coeffs_B[1]
\end{equation}
where the left hand map, viewed as a section of $\sheafHom(i_* \coeffs_B, i_! \coeffs_W[2])$, has nonzero stalk at each zig-zag crossing. Here $\coeffs_B$ and $\coeffs_W$ are the constant sheaves on the union of the black and white regions of $T^2 \smallsetminus \pi(\La_\Gamma)$, respectively.

Note that there is a $\bZ_2^{|\Ga_1|}$-torsor of trivializations
\[
\sheafHom(i_* \coeffs_B, i_! \coeffs_W[2]) \cong \bigoplus_{\substack{\text{crossings} \\ p \,\in\, \pi(\La_\Ga)}} \coeffs_p
\]
of the Hom sheaf on the left which arise by base change from $\bZ$ to $\coeffs$. 
Here the right-hand side is the direct sum of skyscraper sheaves supported at the crossings of $\pi(\La_\Ga)$. A choice of such trivialization identifies isomorphism classes of alternating sheaves with $(\coeffs^\times)^{\Ga_1}/(\coeffs^\times)^{\Ga_0}$, hence with the torus $\Loc_1(\Gamma)$ of rank one local systems on $\Ga$. Since $L_\Gamma$ retracts onto $\Gamma$ we have $\Loc_1(\Gamma) \cong \Loc_1(L_\Gamma)$, hence a trivialization as above defines an embedding $\Loc_1(L_\Gamma) \into \Sh_{\Lambda_\Gamma}(T^2)$ whose essential image is the subcategory of alternating sheaves. 
In fact, there is a standard choice of such trivialization: the $\bZ_2^{|\Ga_1|}$-torsor above is canonically identified with a choice of component of $\Lambda_\Gamma$ above each crossing, and we can consistently choose the ``left'' component at each crossing \cite{STWZ} (in fact consistently choosing the ``right'' component results in the same embedding $\Loc_1(L_\Gamma) \into \Sh_{\Lambda_\Gamma}(T^2)$).

We say that the bipartite graph $\Gamma \subset T^2$ is \newword{consistent} if the Legendrian $\Lambda_\Gamma$ is. This restates \cite[Def. 3.5]{IU}, which in turn is one of several related formulations \cite{MR,Dav,Bro,HV07}, and indeed we have defined the notion of consistent Legendrian to make this so. If $\Gamma$ is consistent, then by Proposition \ref{prop:isotopies} there is a unique complete stacky fan $\stackfan$ such that $\Lambda_\Gamma$ and $\Lambda_{\stackfano}$ are Legendrian isotopic. 

Not all consistent Legendrians are isotopic to ones which arise from bipartite graphs. Those that do are characterized by the following property. First, we can use the orientation of $T^2$ to turn co-oriented curves into oriented curves --- we orient a co-oriented curve $\gamma$ so that the conormal hairs point \emph{right}. This allows us to define the homology class of a co-oriented curve. Since the zig-zag paths bound the union of the black and white regions their homology classes sum to zero. It follows from \cite[Theorem 2.5]{GK} and Proposition \ref{prop:isotopies} that any consistent Legendrian satisfying this condition is Legendrian isotopic to $\Lambda_\Gamma$ for some bipartite graph $\Gamma$. We also have the following notion --- see e.g. \cite{GK}. 

\begin{definition} 
	The \newword{Newton polygon} $P \subset M_\bR \cong H_1(T^2; \bR)$ of $\Gamma$ is the convex lattice polygon, unique up to translation, whose set of counterclockwise-oriented primitive edge vectors are exactly the homology classes of the set of zig-zag paths of $\Gamma$.
\end{definition}
 
 Note that the underlying fan $\Sigma \subset N_\bR$ associated to $\Lambda_{\stackfano} \cong \Lambda_{\Gamma}$ is the inward normal fan of $P$, and $\stackfan$ further records the length of each edge, measured in primitive vectors. 

\begin{example}\label{ex:P1graph}

\begin{figure}
\begin{tikzpicture}
\newcommand*{\edgelen}{4}; \newcommand*{\vertrad}{.12}; \newcommand*{\crad}{.05}
\newcommand*{\gspace}{1.1}; \newcommand*{\angdelta}{22.5};
\newcommand*{\scl}{1.5}; \newcommand*{\zx}{0}; \newcommand*{\ax}{4};
\newcommand*{\bx}{8}; \newcommand*{\cx}{12};
\node (z) [matrix] at (\zx,0) {
	\draw[rotate=180] (0,0)--(2*\scl,0)--(2*\scl,2*\scl)--(0,2*\scl)--(0,0);
	\draw[rotate=180] (.66*\scl,.66*\scl)--(1.33*\scl,1.33*\scl);
	\draw[rotate=180] (1.33*\scl,1.33*\scl)--(1*\scl,2*\scl);
	\draw[rotate=180] (1*\scl,0)--(.66*\scl,.66*\scl);
	\draw[rotate=180] (.66*\scl,.66*\scl)--(0,1*\scl);
	\draw[rotate=180] (1.33*\scl,1.33*\scl)--(2*\scl,1*\scl);
	\draw[rotate=180,black,fill=white] (.66*\scl,.66*\scl) circle (.1cm);
	\draw[rotate=180,black,fill] (1.33*\scl,1.33*\scl) circle (.1cm);
	\\};
\node (a) [matrix] at (\ax,0) {
	\draw[rotate=180] (0,0)--(2*\scl,0)--(2*\scl,2*\scl)--(0,2*\scl)--(0,0);
	\draw[rotate=180] (.66*\scl,.66*\scl)--(1.33*\scl,1.33*\scl);
	\draw[rotate=180] (1.33*\scl,1.33*\scl)--(1*\scl,2*\scl);
	\draw[rotate=180] (1*\scl,0)--(.66*\scl,.66*\scl);
	\draw[rotate=180] (.66*\scl,.66*\scl)--(0,1*\scl);
	\draw[rotate=180] (1.33*\scl,1.33*\scl)--(2*\scl,1*\scl);
	\draw[rotate=180,black,fill=white] (.66*\scl,.66*\scl) circle (.1cm);
	\draw[rotate=180,black,fill] (1.33*\scl,1.33*\scl) circle (.1cm);
	\draw[rotate=180,lefthairs,blue,thick] (.9*\scl,2*\scl)--(.9*\scl,0);
	\draw[rotate=180,lefthairs,blue,thick] (0,.9*\scl)--(2*\scl,.9*\scl);
	\draw[rotate=180,lefthairs,blue,thick] (2*\scl,1.1*\scl)--(1.1*\scl,2*\scl);
	\draw[rotate=180,lefthairs,blue,thick] (1.1*\scl,0)--(0,1.1*\scl);
	\\};
\node (b) [matrix] at (\bx,0) {
	\foreach \x in {0,...,2} \foreach \y in {0,...,2}
	{\fill (\x*\gspace,\y*\gspace) circle (\crad);}
	\draw[thick] (1*\gspace,1*\gspace) to (2*\gspace,1*\gspace) to (1*\gspace,2*\gspace) to (1*\gspace,1*\gspace);
	\\};
\node (c) [matrix] at (\cx,-.32) {
	\foreach \x in {1,...,3} \foreach \y in {1,...,3}
	{\fill (\x*\gspace,\y*\gspace) circle (\crad);}
	\foreach \x/\y in {3.4/2,2/3.4,.7/.7}
	{\draw[-stealth',thick] (2*\gspace,2*\gspace) to (\x*\gspace,\y*\gspace);}
	\foreach \x/\y in {3/2,2/3,1/1}
	{\fill[red] (\x*\gspace,\y*\gspace) circle (\crad*1.5);}
	\node (e1) at (1.5*\gspace,2.7*\gspace) {$\beta(e_1)$};
	\node (e2) at (3*\gspace,1.6*\gspace) {$\beta(e_2)$};
	\node (e3) at (.8*\gspace,1.5*\gspace) {$\beta(e_3)$};
	\\};

\newcommand*{\txtht}{-2.1};
\node (ztxt) at (\zx, \txtht) {$\Gamma \subset T^2$};
\node (atxt) at (\ax, \txtht) {$\Lambda_{\stackfano} \subset T^\infty T^2$};
\node (btxt) at (\bx, \txtht) {$P \subset M_\bR$};
\node (ctxt) at (\cx, \txtht) {$\Sigma \subset N_\bR$};
\end{tikzpicture}
\caption{The projection of the hexagonal lattice to a minimal fundamental domain.}\label{fig:P1graph}
\end{figure}
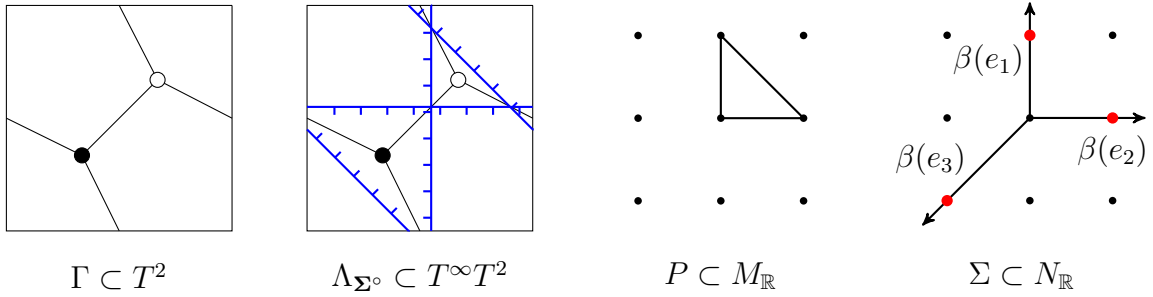

Let $\Gamma$ be the projection of the hexagonal lattice in $\bR^2$ to a minimal fundamental domain, as pictured at the left of Figure \ref{fig:P1graph}. The Legendrian $\Lambda_\Gamma$ has three components, the associated homology classes of which are $(1,0),$ $(0,-1)$ and  $(-1,1)$. The Newton polygon $P$ is a right triangle whose inward normal fan is the fan $\Sigma$ of $\bP^2$. Since the edges of $P$ are all primitive (equivalently, no two components of $\Lambda_\Gamma$ are isotopic), the associated stacky fan $\stackfan$ is trivial (i.e. $\cX_{\stackfan}$ is isomorphic to its coarse moduli space) and the $\beta(e_i)$ are just the primitive generators of the rays of $\Sigma$. We will see in Example \ref{ex:mainex} how, after contact isotopy, homological mirror symmetry matches constructible sheaves with
singular support contained in $\Lambda_\Gamma$ with coherent sheaves on $\cX_{\stackfan} \cong \bP^2$. 
\end{example}

\begin{figure}
\begin{tikzpicture}
\newcommand*{\edgelen}{4}; \newcommand*{\vertrad}{.12}; \newcommand*{\crad}{.05}
\newcommand*{\gspace}{1.1}; \newcommand*{\angdelta}{22.5};
\newcommand*{\scl}{1.5};
\newcommand*{\ax}{0}; \newcommand*{\bx}{5.3}; \newcommand*{\cx}{10.4};
\node (a) [matrix] at (\ax,0) {
	\draw[rotate=180] (0,0)--(4*\scl,0)--(4*\scl,2*\scl)--(0,2*\scl)--(0,0);
	\draw[rotate=180] (.66*\scl,.66*\scl)--(1.33*\scl,1.33*\scl);
	\draw[rotate=180] (1.33*\scl,1.33*\scl)--(1*\scl,2*\scl);
	\draw[rotate=180] (1*\scl,0)--(.66*\scl,.66*\scl);
	\draw[rotate=180] (.66*\scl,.66*\scl)--(0,1*\scl);
	\draw[rotate=180] (1.33*\scl,1.33*\scl)--(2*\scl,1*\scl);
	\draw[rotate=180] (2*\scl+.66*\scl,.66*\scl)--(2*\scl+1.33*\scl,1.33*\scl);
	\draw[rotate=180] (2*\scl+1.33*\scl,1.33*\scl)--(2*\scl+1*\scl,2*\scl);
	\draw[rotate=180] (2*\scl+1*\scl,0)--(2*\scl+.66*\scl,.66*\scl);
	\draw[rotate=180] (2*\scl+.66*\scl,.66*\scl)--(2*\scl+0,1*\scl);
	\draw[rotate=180] (2*\scl+1.33*\scl,1.33*\scl)--(2*\scl+2*\scl,1*\scl);
	\draw[rotate=180,black,fill=white] (.66*\scl,.66*\scl) circle (.1cm);
	\draw[rotate=180,black,fill] (1.33*\scl,1.33*\scl) circle (.1cm);
	\draw[rotate=180,black,fill=white] (2*\scl+.66*\scl,.66*\scl) circle (.1cm);
	\draw[rotate=180,black,fill] (2*\scl+1.33*\scl,1.33*\scl) circle (.1cm);
	\draw[rotate=180,lefthairs,blue,thick] (.9*\scl,2*\scl)--(.9*\scl,0);
	\draw[rotate=180,lefthairs,blue,thick] (0,.9*\scl)--(2*\scl,.9*\scl);
	\draw[rotate=180,lefthairs,blue,thick] (2*\scl,1.1*\scl)--(1.1*\scl,2*\scl);
	\draw[rotate=180,lefthairs,blue,thick] (1.1*\scl,0)--(0,1.1*\scl);
	\draw[rotate=180,lefthairs,blue,thick] (2*\scl+.9*\scl,2*\scl)--(2*\scl+.9*\scl,0);
	\draw[rotate=180,lefthairs,blue,thick] (2*\scl+0,.9*\scl)--(2*\scl+2*\scl,.9*\scl);
	\draw[rotate=180,lefthairs,blue,thick] (2*\scl+2*\scl,1.1*\scl)--(2*\scl+1.1*\scl,2*\scl);
	\draw[rotate=180,lefthairs,blue,thick] (2*\scl+1.1*\scl,0)--(2*\scl+0,1.1*\scl);
	\\};
\node (b) [matrix] at (\bx,0) {
	\foreach \x in {0,...,2} \foreach \y in {0,...,2}
	{\fill (\x*\gspace,\y*\gspace) circle (\crad);}
	\draw[thick] (1*\gspace,0*\gspace) to (2*\gspace,0*\gspace) to (1*\gspace,2*\gspace) to (1*\gspace,0*\gspace);
	\\};
\node (c) [matrix] at (\cx,-.32) {
	\foreach \x in {0,...,4} \foreach \y in {1,...,3}
	{\fill (\x*\gspace,\y*\gspace) circle (\crad);}
	\foreach \x/\y in {4.4/2,2/3.4,-.3/.85}
	{\draw[-stealth',thick] (2*\gspace,2*\gspace) to (\x*\gspace,\y*\gspace);}
	\foreach \x/\y in {4/2,2/3,0/1}
	{\fill[red] (\x*\gspace,\y*\gspace) circle (\crad*1.5);}
	\node (e1) at (1.5*\gspace,2.7*\gspace) {$\beta(e_1)$};
	\node (e2) at (4*\gspace,1.6*\gspace) {$\beta(e_2)$};
	\node (e3) at (-.1*\gspace,1.5*\gspace) {$\beta(e_3)$};
	\\};

\node (atxt) at (\ax, -2.1) {$\Lambda_{\stackfano} \subset T^\infty T^2$};
\node (btxt) at (\bx, -2.1) {$P \subset M_\bR$};
\node (ctxt) at (\cx, -2.1) {$\Sigma \subset N_\bR$};
\end{tikzpicture}
\caption{The projection of the hexagonal lattice to a double cover of a minimal fundamental domain.}\label{fig:dblcover}
\end{figure}
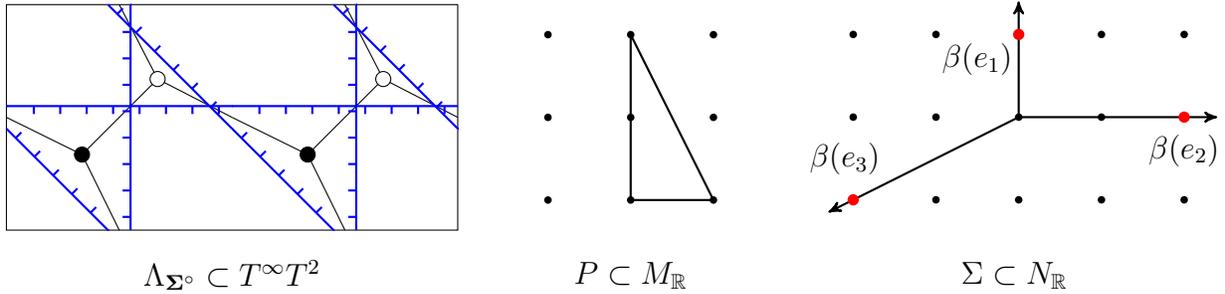

\begin{example}
	Let us also treat a double cover of the previous example: the projection of the hexagonal lattice to a fundamental domain which is twice as wide.
	The Legendrian $\Lambda_\Gamma$ now has four components whose homology classes are $(0,1), (0,1), (-1,0)$, and $(1,-2)$, see Figure \ref{fig:dblcover}. The inward normal fan $\Sigma$ of the new Newton polygon $P$ has three rays with generators $(1,0)$, $(0,1)$, and $(-2,1)$. The stacky fan $\stackfan$ is now nontrivial, with 
	$\beta$ mapping the three generators of $\widehat{N} \cong \bZ^3$ respectively onto the primitive vectors $(0,1)$, $(1,-2)$ and the nonprimitive vector $(2,0)$.
\end{example}

\section{The Kasteleyn operator and the mirror map}\label{sec:Kasteleyn}

If $\Gamma \subset T^2$ is a consistent bipartite graph we have seen that there is a unique complete stacky fan $\stackfan$ such that the alternating Legendrian $\Lambda_\Gamma$ is Legendrian isotopic to $\Lambda_{\stackfano}$. Following \cite{GKS} such an isotopy quantizes to an equivalence $\Sh^c_{\La_\Gamma}(T^2) \cong \Sh^c_{\Lambda_{\stackfano}}(T^2)$. The coherent-constructible correspondence provides a mirror description of the latter category as $\Perf_{prop} (\cX_{\stackfano})$. On the other hand, quantization of the conjugate Lagrangian $L_\Gamma$ as alternating sheaves yields an embedding $\Loc_1(L_\Gamma) \into \Sh^c_{\La_\Gamma}(T^2)$. In this section we show that the composition
$$ \Loc_1(L_\Gamma) \into \Sh^c_{\La_\Gamma}(T^2) \cong \Sh^c_{\Lambda_{\stackfano}}(T^2) \cong \Perf_{prop}(\cX_{\stackfano}) $$
is naturally described by the Kasteleyn operator, whose definition we first recall. 

Let $\cL \in \Loc_1(\Gamma)$ be a rank one local system and choose trivializations of the stalks of $\cL$ at the vertices of $\Ga$. We write $\cL(E)$ for the parallel transport across an edge $E$ from its black endpoint to its white endpoint. We also let $\Gamma_0^b, \Gamma_0^w \subset \Gamma_0$ denote the sets of black and white vertices, respectively. 

We choose generators $x, y$ of $M$, providing a trivialization $\coeffs M \cong \coeffs[x^{\pm1}, y^{\pm 1}]$. We also choose simple closed curves $\ga_x$, $\ga_y$ in $T^2$ representing the Poincar\'{e} duals of $x$ and $y$ with respect to some choice of orientation. We assume that $\ga_x$, $\ga_y$ avoid $p(0)$ and all vertices of $\Gamma$, and that each of them intersect any edge of $\Ga$ at most once. Given an edge $E$ of $\Gamma$, we define $\langle \ga_x, E \rangle$ to be $1$ (resp. $-1$, $0$) if $E$ crosses $\gamma_x$ positively (resp. negatively, not at all) when oriented from black to white ($\langle \ga_y, E \rangle$ is defined the same way). 

We will also need the notion of a Kasteleyn orientation $\kappa$ of $\Gamma$. This is a function $\kappa: \Gamma_1 \to \{ \pm 1\}$ such that the product of the values of $\kappa$ around a face of $\Gamma$ is $-1$ (resp. $1$) if the number of edges on its boundary is $0 \text{ mod } 4$ (resp. $2 \text{ mod } 4$).

\begin{definition}
Given a Kasteleyn orientation $\kappa$, the \newword{Kasteleyn operator} $K_\cL$ of $\cL$ is the $(\Gamma^w_0 \times \Gamma^b_0)$-matrix-valued Laurent polynomial whose $(v_b, v_w)$ entry is
\[
(K_\cL)_{(v_b, v_w)} = \sum_{\substack{\text{$E$ incident} \\ \text{to $v_b$, $v_w$}}} \cL(E)\kappa(E) x^{\langle \ga_x, E \rangle} y^{\langle \ga_y, E \rangle}.
\]
\end{definition}

Since the entries of $K_\cL$ are elements of $\coeffs M \cong \coeffs[T_N]$ we can regard it as a homomorphism
$$ K_\cL : \coeffs[T_N]^{\Ga_0^b} \to \coeffs[T_N]^{\Ga_0^w} $$
of free $\coeffs[T_N]$-modules. 
By the spectral transform of $K_\cL$ we mean its cokernel in $\coeffs[T_N]\mathrm{-mod}$. While the entries of $K_\cL$ depend on the gauge fixing and choice of $\ga_x$ and $\ga_y$, these ambiguities can be absorbed by automorphisms of $\coeffs[T_N]^{\Ga^b_0}$ and $\coeffs[T_N]^{\Ga^w_0}$ hence the spectral transform is independent of them. 

The spectral transform is a pure sheaf on $T_N$ of dimension one, i.e. it has no subsheaves with zero-dimensional support. It is supported on the spectral curve $C$, which is the vanishing locus of the determinant of $K_\cL$.  While $C$ need not be reduced or smooth in general, for generic $\cL$ it is and in this case the spectral transform of $K_\cL$ is the pushforward of a line bundle from $C$ to $T_N$.

Since the edges of $\Gamma$ are in bijection with the crossings of $\pi(\Lambda_\Gamma)$, we can view the Kasteleyn orientation $\kappa$ as the data of a trivialization
\begin{equation}\label{eq:triv}
\sheafHom(i_* \coeffs_B, i_! \coeffs_W[2]) \cong \bigoplus_{\substack{\text{crossings} \\ p \,\in\, \pi(\La_\Ga)}} \coeffs_p
\end{equation}
as in Section \ref{sec:bipartite} (recall that $B, W \subset T^2$ denote the unions of the black and white regions of $T^2 \smallsetminus \pi(\Lambda_\Gamma)$). We simply interpret the signs in $\kappa$ as twisting the standard trivialization by multiplication on the right-hand side of (\ref{eq:triv}). Thus $\kappa$ fixes a choice of signs in the embedding of $\Loc_1(L_\Gamma)$ into $\Sh^c_{\Lambda_\Gamma}(T^2)$ as alternating sheaves.

\begin{theorem}\label{thm:mainresult}
	Let $\Gamma \subset T^2$ be a consistent bipartite graph, $\stackfan$ the associated complete stacky fan, and $\{\Lambda_t\}_{t \in I}$ a Legendrian isotopy with $\Lambda_0 = \Lambda_\Gamma$ and $\Lambda_1 = \Lambda_{\stackfano}$. Then the following diagram commutes.
	\[
	\begin{tikzpicture}
	[thick,>=\arrtip,every text node part/.style={align=center}]
	\newcommand*{\xa}{3.2}; \newcommand*{\xb}{4.0}; \newcommand*{\xc}{3.2}; \newcommand*{\xd}{2.5}
	\newcommand*{\ya}{0}; \newcommand*{\yb}{1.8};
	\node[matrix] (left) at (0,0) {
		\node (a) at (0,0) {$\Sh^c_{\Lambda_\Gamma}(T^2)$};
		\node (b) at (\xa,0) {$\Sh^c_{\Lambda_{\stackfano}}(T^2)$};
		\node (c) at (\xa+\xb,0) {$\Perf_{prop}(\cX_{\stackfano})$};
		\node (d) at (\xa+\xb+\xc,0) {$\Perf_{prop}(T_N)$};
		\node (e) at (-\xd,-\ya-\yb) {$\Loc_1(\Gamma)$};
		\node (f) at (\xa+\xb+\xc,-\ya-\yb) {$\begin{Bmatrix}\text{pure sheaves of} \\ \text{dimension one}\end{Bmatrix}$};
		\node (g) at (-\xd,0) {$\Loc_1(L_\Gamma)$};
		\draw[->] (a) to node[above] {$K_{\{\Lambda_t\}}$} node[below] {$\sim$}(b); 
		\draw[->] (b) to node[above] {$CCC_{\cX_{\stackfano}}^{-1}$} node[below] {$\sim$}  (c); 
		\draw[->] (c) to node[above] {$i_{T_N}^*$} (d);
		\draw[right hook->] (g) to  (a); \draw[->] (e) to node[above, rotate=90] {$\sim$} (g); \draw[->] (e) to node[below] {spectral transform of $K_\cL$} (f); \draw[right hook->] (f) to (d);\\};
	\end{tikzpicture}
	\]
	Here the bottom and top left maps are defined by any fixed Kasteleyn orientation. 
\end{theorem}
\begin{proof}
	
Let $\cL \in \Loc_1(\Gamma)$ be a local system and $\cA \in \Sh^c_{\Lambda_\Gamma}(T^2)$ the associated alternating sheaf. We first claim that 
\begin{equation}\label{eq:invariance}
p_! \omega_{M_\bR} \star \cA \cong p_! \omega_{M_\bR} \star K_{\{\Lambda_t\}}(\cA).
\end{equation} Let $\cA_I$ denote the image of $\cA$ under the inverse GKS equivalence $(i_{T^2_0}^*)^{-1}: \Sh^c_{\Lambda_0}(T^2) \congto \Sh^c_{\Lambda_I}(T^2 \times I)$ of Corollary \ref{cor:LegendrianGKS}. It suffices to show that $p_! \omega_{M_\bR} \star \cA_I := (m \times id_I)_! (p_! \omega_{M_\bR} \boxtimes \cA_I) \in \Sh(T^2 \times I)$ is locally constant: by base change $i_{T^2_0}^* (p_! \omega_{M_\bR} \star \cA_I)  \cong (p_! \omega_{M_\bR} \star \cA)$ and $i_{T^2_1}^* (p_! \omega_{M_\bR} \star \cA_I)  \cong p_! \omega_{M_\bR} \star K_{\{\Lambda_t\}}(\cA)$, but if $p_! \omega_{M_\bR} \star \cA_I$ is locally free $i_{T^2_t}^* (p_! \omega_{M_\bR} \star \cA_I) \cong i_{T^2_{t'}}^* (p_! \omega_{M_\bR} \star \cA_I)$ for all $t, t'$.

We can bound the singular support of the proper pushforward $p_! \omega_{M_\bR} \star \cA_I$ using \cite[Prop. 5.4.4]{KS94}. In the case at hand it says that
\begin{equation*}
\begin{aligned}
SS(p_! \omega_{M_\bR} \star \cA_I) \subset &\{(p(m), n, t, \tau) \in T^*T^2 \times T^*I \text{ such that } \exists(p(m_1), n) \in SS(p_! \omega_{M_\bR}),\\ &\quad  (p(m_2), n, t, \tau) \in SS(\cA_I) \text{ with }p(m_1+m_2)=p(m)\}.
\end{aligned}
\end{equation*}
Since $p_! \omega_{M_\bR}$ is locally free $n = 0$. On the other hand nonzero covectors in $SS(\cA_I)$ are in the cone over $\Lambda_I$, and one can easily check in a local model that a covector $(0, \tau)$ in this cone must have $\tau = 0$. 

It now follows from (\ref{eq:invariance}) and Lemma \ref{lem:openorbitmain} that
$$ i^*_{T_N} ( CCC_{\cX_{\stackfano}}^{-1} ( K_{\{\Lambda_t\}}(\cA))) \cong CCC_{T_N}^{-1}(p_! \omega_{M_\bR} \star  K_{\{\Lambda_t\}}(\cA)) \cong CCC_{T_N}^{-1}(p_! \omega_{M_\bR} \star \cA) \cong \Gamma_c(p^*(\cA)).$$
Note that even though $\cA$ is not microsupported on $\Lambda_{\stackfan}$, the rightmost isomorphism follows from the same proof as the Lemma. 

Using the presentation (\ref{eq:altpres}) we obtain a triangle
$$ \Gamma_c(p^*(i_* \coeffs_B)) \to \Gamma_c(p^*(i_! \coeffs_W[2])) \to \Gamma_c(p^*(\cA)) \to \Gamma_c(p^*(i_* \coeffs_B))[1]. $$ 
 We immediately have $p^*(i_* \coeffs_B) \cong i_* \coeffs_{p^{-1}(B)}$ and $p^*(i_! \coeffs_W[2]) \cong i_! \coeffs_{p^{-1}(W)}[2]$. The preimages $p^{-1}(B)$ and $p^{-1}(W)$ are disjoint unions of contractible open sets in correspondence with $p^{-1}(\Gamma^b_0)$ and $p^{-1}(\Gamma^w_0)$, respectively. We identify $\pi_0( p^{-1}(B))$ and $\pi_0 (p^{-1}(W))$ with $M \times \Gamma^b_0$ and $M \times \Gamma^w_0$ as follows. The lifts $p^{-1}(\gamma_x)$, $p^{-1}(\gamma_y)$ carve $M_\bR$ into fundamental domains which each contain a unique element of $M$, and given $v \in \Gamma_0$ there is a unique point of $p^{-1}(v)$ in each such domain. 
 
Recalling that $\coeffs[T_N] \cong \coeffs M$, it follows that we have $M$-equivariant isomorphisms 
$$\Gamma_c(p^*(i_* \coeffs_B)) \cong \coeffs^{M \times \Gamma^b_0} \cong \coeffs[T_N]^{\Gamma^b_0}, \quad \Gamma_c(p^*(i_! \coeffs_W[2])) \cong \coeffs^{M \times \Gamma^w_0} \cong \coeffs[T_N]^{\Gamma^w_0}.$$ 
In particular, both are supported in cohomological degree zero. 
The fact that the first map in the resulting triangle
\[
\coeffs[T_N]^{\Gamma^b_0}  \to \coeffs[T_N]^{\Gamma^w_0}  \to \Gamma_c(p^*(\cA)) \to \coeffs[T_N]^{\Gamma^b_0}[1]
\]
is the Kasteleyn operator of $\cL$ now follows immediately from the way we associated $\cA$ to $\cL$ in the first place. Note also that since $K_\cL$ is injective the cone over it is just its cokernel. 
\end{proof}

The slightly different statement of Theorem \ref{thm:mainthmintro} follows easily from Theorem \ref{thm:mainresult}, just being reformulated in terms of the complete stack $\cX_{\stackfan}$. 

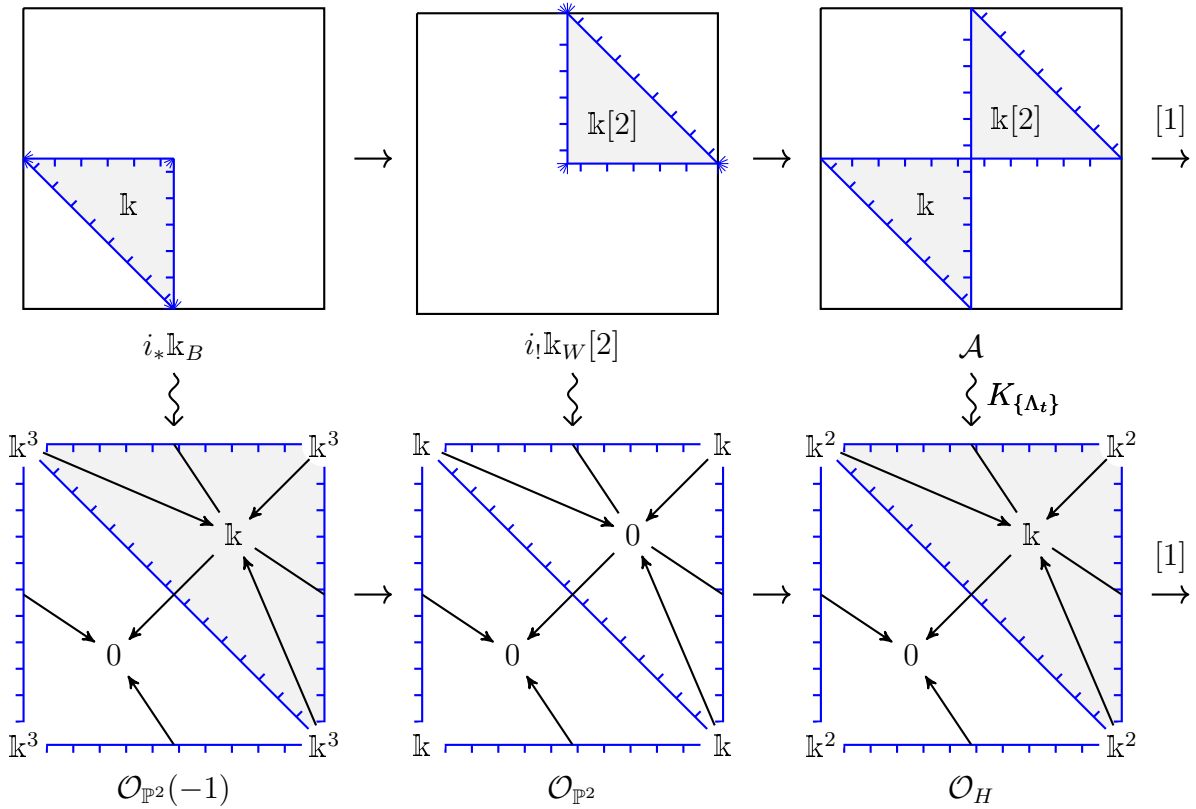
\begin{figure}
	\begin{tikzpicture}
	\newcommand*{\edgelen}{4}; \newcommand*{\wcircrad}{.3}
	\newcommand*{\xytrtri}{.7}; \newcommand*{\xybltri}{.3};
	\newcommand*{\ashort}{8}; \newcommand*{\bshort}{4};
	\newcommand*{\toprowheight}{5.8}; \newcommand*{\bspace}{5.3};
	\newcommand*{\rarrs}{2.4}; \newcommand*{\rarrt}{2.9};
	\newcommand*{\darrs}{2.95}; \newcommand*{\darrt}{2.25};
	\newcommand*{\vertrad}{.14}; \newcommand*{\angdelta}{22.5};
	
	\node (d) [matrix] at (-\bspace,\toprowheight) {
		\coordinate (tl) at (0,\edgelen); \coordinate (tr) at (\edgelen,\edgelen);
		\coordinate (bl) at (0,0); \coordinate (br) at (\edgelen,0);
		\coordinate (bm) at (\edgelen*.5,0); \coordinate (lm) at (0,\edgelen*.5);
		\coordinate (mm) at (\edgelen*.5,\edgelen*.5);
		\path[fill=\fcolor] (lm) -- (mm) -- (bm) -- (lm);
		\draw[thick] (tl) to (bl) to (br) to (tr) to (tl);
		\foreach \a/\b in {mm/bm, bm/lm} {\draw[asdstyle,righthairs] (\a) to (\b);}
		\foreach \a/\b in {mm/lm} {\draw[asdstyle,lefthairs] (\a) to (\b);}
		\node (bltri) at (.35*\edgelen, .35*\edgelen) {$\coeffs$};
		\foreach \ang in {8,...,12} {
			\draw[blue] ($(mm)+(\ang*\angdelta:\vertrad)$) to (mm);}
		\foreach \ang in {12,...,18} {
			\draw[blue] ($(lm)+(\ang*\angdelta:\vertrad)$) to (lm);}
		\foreach \ang in {2,...,8} {
			\draw[blue] ($(bm)+(\ang*\angdelta:\vertrad)$) to (bm);}
		\\};
	\node (e) [matrix] at (0,\toprowheight) {
		\coordinate (tl) at (0,\edgelen); \coordinate (tr) at (\edgelen,\edgelen);
		\coordinate (bl) at (0,0); \coordinate (br) at (\edgelen,0);
		\coordinate (tm) at (\edgelen*.5,\edgelen); \coordinate (rm) at (\edgelen,\edgelen*.5);
		\coordinate (mm) at (\edgelen*.5,\edgelen*.5);
		\path[fill=\fcolor] (tm) -- (mm) -- (rm) -- (tm);
		\draw[thick] (tl) to (bl) to (br) to (tr) to (tl);
		\foreach \a/\b in {mm/rm, rm/tm} {\draw[asdstyle,righthairs] (\a) to (\b);}
		\foreach \a/\b in {mm/tm} {\draw[asdstyle,lefthairs] (\a) to (\b);}
		\node (bltri) at (.65*\edgelen, .625*\edgelen) {$\coeffs[2]$};
		\foreach \ang in {8,...,12} {
			\draw[blue] ($(mm)+(\ang*\angdelta:\vertrad)$) to (mm);}
		\foreach \ang in {12,...,18} {
			\draw[blue] ($(rm)+(\ang*\angdelta:\vertrad)$) to (rm);}
		\foreach \ang in {2,...,8} {
			\draw[blue] ($(tm)+(\ang*\angdelta:\vertrad)$) to (tm);}
		\\};
	\node (f) [matrix] at (\bspace,\toprowheight) {
		\coordinate (tl) at (0,\edgelen); \coordinate (tr) at (\edgelen,\edgelen);
		\coordinate (bl) at (0,0); \coordinate (br) at (\edgelen,0);
		\coordinate (tm) at (\edgelen*.5,\edgelen); \coordinate (rm) at (\edgelen,\edgelen*.5);
		\coordinate (bm) at (\edgelen*.5,0); \coordinate (lm) at (0,\edgelen*.5);
		\coordinate (mm) at (\edgelen*.5,\edgelen*.5);
		\path[fill=\fcolor] (lm) -- (mm) -- (bm) -- (lm);
		\path[fill=\fcolor] (tm) -- (mm) -- (rm) -- (tm);
		\draw[thick] (tl) to (bl) to (br) to (tr) to (tl);
		\foreach \a/\b in {mm/rm, mm/bm, rm/tm, bm/lm} {\draw[asdstyle,righthairs] (\a) to (\b);}
		\foreach \a/\b in {mm/lm, mm/tm} {\draw[asdstyle,lefthairs] (\a) to (\b);}
		\node (bltri) at (.35*\edgelen, .35*\edgelen) {$\coeffs$};
		\node (bltri) at (.65*\edgelen, .625*\edgelen) {$\coeffs[2]$};
		\\};
	
	\node (a) [matrix] at (-\bspace,0) {
		\coordinate (tl) at (0,\edgelen); \coordinate (tr) at (\edgelen,\edgelen);
		\coordinate (bl) at (0,0); \coordinate (br) at (\edgelen,0);
		\path[fill=\fcolor] (tl) -- (tr) -- (br) -- (tl);
		\foreach \a/\b in {tl/bl, tr/br, br/tl} {\draw[asdstyle,righthairs] (\a) to (\b);}
		\foreach \a/\b in {tr/tl, br/bl} {\draw[asdstyle,lefthairs] (\a) to (\b);}
		\foreach \c in {tl, tr, bl, br} {
			\fill[white] (\c) circle (\wcircrad);
			\node at (\c) {$\coeffs^3$};}
		\node (trtri) at (\xytrtri*\edgelen, \xytrtri*\edgelen) {$\coeffs$};
		\node (bltri) at (\xybltri*\edgelen, \xybltri*\edgelen) {$0$};
		\draw[genmapstyle, shorten <=\ashort] (tl) to (trtri);
		\draw[genmapstyle, shorten <=\ashort] (br) to (trtri);
		\draw[genmapstyle, shorten <=\ashort, shorten >=-\bshort+1] (tr) to (trtri);
		\draw[thick] (trtri) to (\edgelen,.5*\edgelen); 
		\draw[genmapstyle] (0,.5*\edgelen) to (bltri);
		\draw[thick] (trtri) to (.5*\edgelen,\edgelen); 
		\draw[genmapstyle] (.5*\edgelen,0) to (bltri);
		\draw[genmapstyle, shorten >=-\bshort+2] (trtri) to (bltri);
		\\};
	\node (b) [matrix] at (0,0) {
		\coordinate (tl) at (0,\edgelen); \coordinate (tr) at (\edgelen,\edgelen);
		\coordinate (bl) at (0,0); \coordinate (br) at (\edgelen,0);
		\foreach \a/\b in {tl/bl, tr/br, br/tl} {\draw[asdstyle,righthairs] (\a) to (\b);}
		\foreach \a/\b in {tr/tl, br/bl} {\draw[asdstyle,lefthairs] (\a) to (\b);}
		\foreach \c in {tl, tr, bl, br} {
			\fill[white] (\c) circle (\wcircrad);
			\node at (\c) {$\coeffs$};}
		\node (trtri) at (\xytrtri*\edgelen, \xytrtri*\edgelen) {$0$};
		\node (bltri) at (\xybltri*\edgelen, \xybltri*\edgelen) {$0$};
		\draw[genmapstyle, shorten <=\ashort] (tl) to (trtri);
		\draw[genmapstyle, shorten <=\ashort] (br) to (trtri);
		\draw[genmapstyle, shorten <=\ashort, shorten >=-\bshort+1] (tr) to (trtri);
		\draw[thick] (trtri) to (\edgelen,.5*\edgelen); 
		\draw[genmapstyle] (0,.5*\edgelen) to (bltri);
		\draw[thick] (trtri) to (.5*\edgelen,\edgelen); 
		\draw[genmapstyle] (.5*\edgelen,0) to (bltri);
		\draw[genmapstyle, shorten >=-\bshort+2] (trtri) to (bltri);
		\\};
	\node (c) [matrix] at (\bspace,0) {
		\coordinate (tl) at (0,\edgelen); \coordinate (tr) at (\edgelen,\edgelen);
		\coordinate (bl) at (0,0); \coordinate (br) at (\edgelen,0);
		\path[fill=\fcolor] (tl) -- (tr) -- (br) -- (tl);
		\foreach \a/\b in {tl/bl, tr/br, br/tl} {\draw[asdstyle,righthairs] (\a) to (\b);}
		\foreach \a/\b in {tr/tl, br/bl} {\draw[asdstyle,lefthairs] (\a) to (\b);}
		\foreach \c in {tl, tr, bl, br} {
			\fill[white] (\c) circle (\wcircrad);
			\node at (\c) {$\coeffs^2$};}
		\node (trtri) at (\xytrtri*\edgelen, \xytrtri*\edgelen) {$\coeffs$};
		\node (bltri) at (\xybltri*\edgelen, \xybltri*\edgelen) {$0$};
		\draw[genmapstyle, shorten <=\ashort] (tl) to (trtri);
		\draw[genmapstyle, shorten <=\ashort] (br) to (trtri);
		\draw[genmapstyle, shorten <=\ashort, shorten >=-\bshort+1] (tr) to (trtri);
		\draw[thick] (trtri) to (\edgelen,.5*\edgelen); 
		\draw[genmapstyle] (0,.5*\edgelen) to (bltri);
		\draw[thick] (trtri) to (.5*\edgelen,\edgelen); 
		\draw[genmapstyle] (.5*\edgelen,0) to (bltri);
		\draw[genmapstyle, shorten >=-\bshort+2] (trtri) to (bltri);
		\\};
	
	\newcommand*{\utxtoff}{2.5}; \newcommand*{\btxtoff}{2.6};
	\node (atxt) at (-\bspace, -\btxtoff) {$\cO_{\bP^2}(-1)$};
	\node (btxt) at (0, -\btxtoff) {$\cO_{\bP^2}$};
	\node (ctxt) at (\bspace, -\btxtoff) {$\cO_H$};
	\node (dtxt) at (-\bspace, \toprowheight-\utxtoff) {$i_* \coeffs_B$};
	\node (etxt) at (0, \toprowheight-\utxtoff) {$i_! \coeffs_W[2]$};
	\node (ftxt) at (\bspace, \toprowheight-\utxtoff) {$\cA$};
	\foreach \c in {a,b,c,d,e,f} {
		\draw[thick,->] ($(\c)+(\rarrs,0)$) to ($(\c)+(\rarrt,0)$); }
	\foreach \c in {c,f} {\node at ($(\c)+(\bspace*.5,.5)$) {$[1]$};} 
	\foreach \c/\d in {a/d,b/e,c/f} {
		\draw[arrowstyle] ($(\c)+(0,\darrs)$) to ($(\c)+(0,\darrt)$);
		\draw[arrhdstyle] ($(\c)+(0,\darrt)$) to ($(\c)+(0,\darrt)+(1.2mm,1.2mm)$);
		\draw[arrhdstyle] ($(\c)+(0,\darrt)$) to ($(\c)+(0,\darrt)+(-1.2mm,1.2mm)$);
	\node at ($(c)+(0,\darrt)+(.7,.35)$) {$K_{\{\Lambda_t\}}$};}
	
	\end{tikzpicture}
	\caption{The action of a Legendrian isotopy from $\Lambda_\Gamma$ to $\Lambda_\Sigma$ on an alternating sheaf in $\Lambda_\Gamma$ (right column, top to bottom). Here $\Gamma$ is the hexagonal lattice projected to a minimal fundamental domain (as in Figure \ref{fig:P1graph}) and $\Sigma$ is the fan of $\bP^2$.}\label{fig:mainex}	
\end{figure}

\begin{example}\label{ex:mainex}
	We continue with Example \ref{ex:P1graph}. The top row of Figure \ref{fig:mainex} illustrates the triangle
	$$ i_* \coeffs_B \to i_! \coeffs_W[2] \to \cA \to i_* \coeffs_B[1] $$
	presenting an alternating sheaf $\cA$. The Hom sheaf $\sheafHom(i_* \coeffs_B, i_! \coeffs_W[2])$ is a sum of skyscraper sheaves at the three crossings of $\pi(\Lambda_\Gamma)$. Up to isomorphism $\cA$ is determined by a section of this sum which is nonvanishing at all three crossings. Note that whereas $\cA$ is an object of $\Sh^c_{\Lambda_\Gamma}(T^2)$, the sheaves $i_* \coeffs_B$ and $i_! \coeffs_W[2]$ are not: above the crossings of $\pi(\Lambda_\Gamma)$ their singular support contains an interval of codirections which lie outside $\Lambda_\Gamma$. Informally, these codirections ``cancel out'' upon taking the cone.

Let $\{\Lambda_t\}_{t \in I}$ be the Legendrian isotopy which carries $\Lambda_\Gamma$ to $\Lambda_{\Sigmao}$ by moving each front projection up and to the right in the pictured fundamental domain. That is, starting from the top right of Figure \ref{fig:mainex} we collapse the upper right triangle into the upper right corner while expanding the lower left triangle to take up the entire upper right half of the picture. In this case, using the main theorem of \cite{Zho18} one can in fact extend the associated GKS equivalence $K_{\{\Lambda_t\}}$ to an equivalence $\Sh^c_{\Lambda'_\Gamma}(T^2) \congto \Sh^c_{\Lambda_\Sigma}(T^2)$, where 
\begin{equation}\label{eq:Lambdaprime} \Lambda'_\Gamma := \Lambda_\Gamma \cup SS^\infty(i_* \coeffs_B) \cup SS^\infty(i_! \coeffs_W[2]). \end{equation}

Composing with the coherent-constructible correspondence for $\bP^2$, the choice of a section of $\sheafHom(i_* \coeffs_B, i_! \coeffs_W[2])$ becomes the choice of a linear equation for a hyperplane $H \subset \bP^2$. The nonvanishing condition at crossings translates to the condition that $H$ does not meet the $T_N$-fixed points in $\bP^2$. The alternating sheaf $\cA$ itself is mapped to the structure sheaf of $H$, while $i_* \coeffs_B$ and $i_! \coeffs_W[2]$ are mapped to $\cO_{\bP^2}(-1)$ and $\cO_{\bP^2}$, respectively.
\end{example}

\begin{remark}
One can extend $K_{\{\Lambda_t\}}$ to an equivalence $\Sh^c_{\Lambda'_\Gamma}(T^2) \congto \Sh^c_{\Lambda_{\stackfan}}(T^2)$ more generally, where $\Lambda'_\Gamma$ is as in (\ref{eq:Lambdaprime}), provided one can find an isotopy between the singular Legendrians $\Lambda'_\Gamma$ and $\Lambda_{\stackfan}$ which satisfies the criteria of \cite{Zho18}. However, we do not know whether this is possible for arbitrary $\Gamma$. When it is possible, the sheaves $i_* \coeffs_B$ and $i_! \coeffs_W[2]$ are taken by the composition of $K_{\{\Lambda_t\}}$ and the $CCC$ to direct sums of line bundles on $\cX_\stackfan$. 
\end{remark}

Given a quadrilateral face of a bipartite graph $\Gamma \subset T^2$, we can produce a new bipartite graph $\Gamma'$ by performing a square move. The new graph is obtained from $\Gamma$ by gluing in the local picture of Figure \ref{fig:squaremove}. The Legendrians $\Lambda_\Gamma$ and $\Lambda_{\Gamma'}$ are related by a Legendrian isotopy supported above a small open set containing the face. The Lagrangians $L_\Gamma$ and $L_{\Gamma'}$ discussed in the introduction are related as the two inequivalent Lagrangian surgeries on a singular Lagrangian which interpolates between them \cite{STWZ}. 

The relation between alternating sheaves defined with respect to the two graphs $\Gamma$ and $\Gamma'$ is naturally described in terms of face coordinates. That is, for any face $F$ of $\Gamma$ we have a function $X_F$ on $\Loc_1(\Gamma)$ whose value on a local system is its holonomy around $\partial F$ (taken counterclockwise in the local model of Figure \ref{fig:squaremove}). The coordinate ring of $\Loc_1(\Gamma)$ is the quotient of the Laurent polynomial ring in its face coordinates modulo the relation that the product of all face coordinates is 1. 

 Note that the dual graph of $\Gamma$ is naturally a quiver (we orient the dual graph so that any edge passes a white vertex on its right) and when $\Gamma$ undergoes a square move its dual graph undergoes a quiver mutation. A key result of \cite{STWZ} is that alternating sheaves before and after a square move are related by a commutative diagram
\[
\begin{tikzpicture}
[thick,>=\arrtip,every text node part/.style={align=center}]
\newcommand*{\xa}{2.8}; \newcommand*{\xb}{2.8}; \newcommand*{\xc}{2.5}; \newcommand*{\xd}{2.5}
\newcommand*{\ya}{1.5}; \newcommand*{\yb}{1.8};
\node[matrix] (left) at (0,0) {
	\node (a) at (0,0) {$\Loc_1(L_\Gamma)$};
	\node (b) at (0,-\ya) {$\Loc_1(L_{\Gamma'})$};
	\node (c) at (\xa,0) {$\Sh^c_{\Lambda_{\Gamma}}(T^2)$};
	\node (d) at (\xa,-\ya) {$\Sh^c_{\Lambda_{\Gamma'}}(T^2).$};
	\draw[dashed,->] (a) to  (b); \draw[->] (c) to node[below, rotate=90] {$\sim$} (d);
	\draw[right hook->] (b) to (d); \draw[right hook->] (a) to (c);\\};
\end{tikzpicture}
\]
Here the right map is the equivalence defined by the local isotopy $\Lambda_{\Gamma} \to \Lambda_{\Gamma'}$ and the left map is the cluster $\cX$-transformation associated to the mutation of the dual quiver. Explicitly this means the two families of alternating sheaves are related by the following rational map. Let $X_M$, $X'_M$ be the face coordinates of the middle faces of $\Gamma$, $\Gamma'$ in Figure \ref{fig:squaremove}, and $X_{SW}$, $X'_{SW}$ the face coordinates of the southwest faces (similarly for $X_{NE}$, etc...). Then the two sets of face coordinates are related by
\begin{gather*}
X'_M = X_M^{-1}, \quad X'_{SE} = X_{SE}(1+X_M),  \quad X'_{NW} = X_{NW}(1+X_M),\\
X'_{SW} = X_{SW}(1+X_M^{-1})^{-1},  \quad X'_{NE} = X_{NE}(1+X_M^{-1})^{-1}
\end{gather*}
and $X'_F = X_F$ if $F$ does not share an edge with the given square face. 

We also note from \cite{STWZ} that the appearance of positive signs in the above formula is equivalent to the sign condition on the Kasteleyn orientation $\kappa$ at the given square face. Had we specified the trivialization (\ref{eq:triv}) with a function $\Gamma_1 \to \{\pm 1\}$ whose product around this face was 1 we would instead see minus signs in the above coordinate change. From the preceding discussion and Theorem \ref{thm:mainresult} one immediately obtains Corollary \ref{cor:clusterintro}.

\begin{figure}
	\centering
	\begin{tikzpicture}
	\newcommand*{\off}{18};\newcommand*{\rad}{2.3};\newcommand*{\vrad}{1.0};
	\node (l) [matrix] at (0,0) {
		\coordinate (bl) at (180:\rad/2); \coordinate (br) at (0:\rad/2);
		\coordinate (wt) at (90:\rad*.8); \coordinate (wb) at (-90:\rad*.8);
		\coordinate (wr) at (0:\rad*.9); \coordinate (wl) at (180:\rad*.9);
		\coordinate (ur) at ($(-\off:\rad)!\ifac!(-180+\off:\rad)$); \coordinate (lr) at ($(\off-180:\rad)!\ifac!(-\off:\rad)$);
		\coordinate (ul) at ($(\off:\rad)!\ifac!(-180-\off:\rad)$); \coordinate (ll) at ($(-\off-180:\rad)!\ifac!(\off:\rad)$);
		\coordinate (mr) at ($(ur)!.5!(lr)$); \coordinate (ml) at ($(ul)!.5!(ll)$);
		\path[fill=\fcolor, name path=p1] (90+\off:\rad) -- (-90-\off:\rad) arc[start angle=-90-\off,end angle=-90+\off,radius=\rad] (-90+\off:\rad) -- (90-\off:\rad) arc[start angle=90-\off,end angle=90+\off,radius=\rad] (90+\off:\rad);
		\path[fill=\fcolor, name path=p2] (180-\off:\rad)  to[out=0,in=-180] (lr) to (ur) [out=0,in=-180] to (\off:\rad) arc[start angle=\off,delta angle=-2*\off,radius=\rad] (-\off:\rad) [out=-180,in=0] to (ul) to (ll) to[out=-180,in=0] (180+\off:\rad) arc[start angle=180+\off,delta angle=-2*\off,radius=\rad] (180-\off:\rad);
		\path[fill=white, name intersections={of=p1 and p2, name=i}] (i-1) -- (i-2) -- (i-4) -- (i-3) -- cycle;
		\draw[dashed] (0,0) circle (\rad);
		\draw[graphstyle] (wb) to (br) to (wt) to (bl) to (wb);
		\draw[graphstyle] (br) to (wr) to (0:\rad);
		\draw[graphstyle] (bl) -- (wl) -- (180:\rad);
		\draw[graphstyle] (wt) to (90:\rad); \draw[graphstyle] (wb) to (-90:\rad);
		\foreach \c in {bl,br,wt,wb,wr,wl} {\fill[black] (\c) circle (\bvertrad);}
		\foreach \c in {wt,wb,wr,wl} {\fill[white] (\c) circle (\wvertrad);}
		\draw[asdstyle,righthairs] (90+\off:\rad) -- (-90-\off:\rad);
		\draw[asdstyle,righthairs] (-90+\off:\rad) -- (90-\off:\rad);
		\draw[asdstyle,lefthairs] (mr) to (ur) [out=0,in=180] to (\off:\rad);
		\draw[asdstyle,righthairs] (mr) to (lr) to[in=0,out=180] (180-\off:\rad);
		\draw[asdstyle,righthairs] (ml) to (ul) [out=0,in=180] to (-\off:\rad);
		\draw[asdstyle,lefthairs] (ml) to (ll) to[in=0,out=180] (180+\off:\rad);\\};
	
	\node (m) [matrix] at (2.5*\rad,0) {
		\path[fill=\fcolor] (0,0) to[out=-90,in=90]  (-90-\off:\rad) arc[start angle=-90-\off,delta angle=2*\off,radius=\rad] (-90+\off:\rad) to[out=90,in=-90] (0,0);
		\path[fill=\fcolor] (0,0) to[out=0,in=180]  (-\off:\rad) arc[start angle=-\off,delta angle=2*\off,radius=\rad] (\off:\rad) to[out=180,in=0] (0,0);
		\path[fill=\fcolor] (0,0) to[out=90,in=-90]  (90-\off:\rad) arc[start angle=90-\off,delta angle=2*\off,radius=\rad] (90+\off:\rad) to[out=-90,in=90] (0,0);
		\path[fill=\fcolor] (0,0) to[out=180,in=0]  (180-\off:\rad) arc[start angle=180-\off,delta angle=2*\off,radius=\rad] (180+\off:\rad) to[out=0,in=180] (0,0);
		\draw[dashed] (0,0) circle (\rad);
		\draw[asdstyle,lefthairsnogap] (0,0) to[out=90,in=-90] (90+\off:\rad);
		\draw[asdstyle,righthairsnogap] (0,0) to[out=-90,in=90]  (-90-\off:\rad);
		\draw[asdstyle,righthairsnogap] (0,0) to[out=90,in=-90]  (90-\off:\rad);
		\draw[asdstyle,lefthairsnogap] (0,0) to[out=-90,in=90]  (-90+\off:\rad);
		\draw[asdstyle,lefthairsnogap] (0,0) to[out=0,in=180] (\off:\rad);
		\draw[asdstyle,righthairsnogap] (0,0) to[out=180,in=0] (180-\off:\rad);
		\draw[asdstyle,righthairsnogap] (0,0) to[out=0,in=180] (-\off:\rad);
		\draw[asdstyle,lefthairsnogap] (0,0) to[out=180,in=0] (180+\off:\rad);\\};
	
	\node (r) [matrix] at (5*\rad,0) {
		\coordinate (bl) at (180+90:\rad/2); \coordinate (br) at (0+90:\rad/2);
		\coordinate (wt) at (90+90:\rad*.8); \coordinate (wb) at (-90+90:\rad*.8);
		\coordinate (wr) at (0+90:\rad*.9); \coordinate (wl) at (180+90:\rad*.9);
		\coordinate (ur) at ($(-\off+90:\rad)!\ifac!(-90+\off:\rad)$); \coordinate (lr) at ($(\off-90:\rad)!\ifac!(90-\off:\rad)$);
		\coordinate (ul) at ($(\off+90:\rad)!\ifac!(-90-\off:\rad)$); \coordinate (ll) at ($(-\off-90:\rad)!\ifac!(90+\off:\rad)$);
		\coordinate (mr) at ($(ur)!.5!(lr)$); \coordinate (ml) at ($(ul)!.5!(ll)$);
		\path[fill=\fcolor, name path=p1] (90+90+\off:\rad) -- (-90+90-\off:\rad) arc[start angle=-\off,end angle=\off,radius=\rad] -- (180-\off:\rad) arc[start angle=180-\off,end angle=180+\off,radius=\rad];
		\path[fill=\fcolor, name path=p2] (180+90-\off:\rad)  to[out=90,in=-90] (lr) to (ur) [out=90,in=-90] to (\off+90:\rad) arc[start angle=90+\off,delta angle=-2*\off,radius=\rad] [out=-90,in=90] to (ul) to (ll) to[out=-90,in=90] (180+90+\off:\rad) arc[start angle=180+90+\off,delta angle=-2*\off,radius=\rad];
		\path[fill=white, name intersections={of=p1 and p2, name=i}] (i-1) -- (i-2) -- (i-4) -- (i-3) -- cycle;
		\draw[dashed] (0,0) circle (\rad);
		\draw[graphstyle] (wb) to (br) to (wt) to (bl) to (wb);
		\draw[graphstyle] (br) to (wr) to (90:\rad);
		\draw[graphstyle] (bl) -- (wl) -- (180+90:\rad);
		\draw[graphstyle] (wt) to (90+90:\rad); \draw[graphstyle] (wb) to (-90+90:\rad);
		\foreach \c in {bl,br,wt,wb,wr,wl} {\fill[black] (\c) circle (\bvertrad);}
		\foreach \c in {wt,wb,wr,wl} {\fill[white] (\c) circle (\wvertrad);}
		\draw[asdstyle,righthairs] (90+90+\off:\rad) -- (-90+90-\off:\rad);
		\draw[asdstyle,righthairs] (-90+90+\off:\rad) -- (90+90-\off:\rad);
		\draw[asdstyle,lefthairs] (mr) to (ur) [out=0+90,in=180+90] to (\off+90:\rad);
		\draw[asdstyle,righthairs] (mr) to (lr) to[in=0+90,out=180+90] (180+90-\off:\rad);
		\draw[asdstyle,righthairs] (ml) to (ul) [out=0+90,in=180+90] to (-\off+90:\rad);
		\draw[asdstyle,lefthairs] (ml) to (ll) to[in=0+90,out=180+90] (180+90+\off:\rad);\\};
	\coordinate (rw) at ($(r.west)+(-1mm,0)$); \coordinate (me) at ($(m.east)+(1mm,0)$);
	\coordinate (mw) at ($(m.west)+(-1mm,0)$); \coordinate (le) at ($(l.east)+(1mm,0)$);
	\foreach \ca/\cb in {le/mw,me/rw} {
		\draw[arrowstyle] (\ca) -- (\cb);
		\draw[arrhdstyle] ($(\cb)+(-1.2mm,1.2mm)$) -- (\cb); \draw[arrhdstyle] ($(\cb)+(-1.2mm,-1.2mm)$) -- (\cb);}
	\end{tikzpicture}
	\caption{The square move $\Gamma \to \Gamma'$ as a Legendrian isotopy $\Lambda_\Gamma \to \Lambda_\Gamma'$. Shaded regions indicate the images of the Lagrangians $L_\Gamma$, $L_{\Gamma'}$, and in the center frame of an immersed Lagrangian of which they are surgeries.}
	\label{fig:squaremove}
\end{figure}
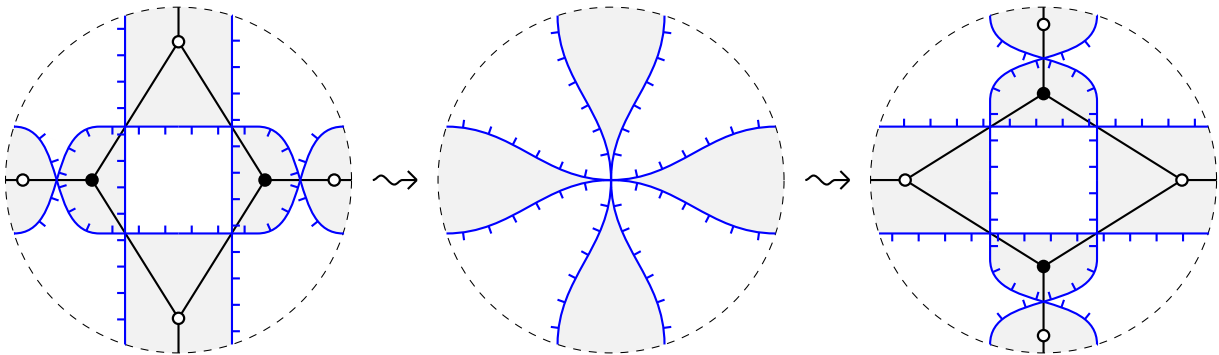

\section{Legendrian isotopies and discrete integrable systems}\label{sec:isotopies}
In stating Theorem \ref{thm:mainresult} we fixed an arbitrary choice of Legendrian isotopy $\Lambda_\Gamma \congto \Lambda_{\stackfano}$. However, there are many inequivalent choices, as the autoisotopy group of $\Lambda_{\stackfano}$ is disconnected. In this section we show that the action of autoisotopies on $\Sh^c_{\Lambda_{\stackfano}}(T^2)$ is mirrored by the action of tensoring by line bundles on coherent sheaves (hence this action preserves the restriction to $T_N$, as follows implicitly from Theorem \ref{thm:mainresult}). 

Recall from Section \ref{sec:CCC} that the components of $\Lambda_{\stackfano}$ are in correspondence with pairs of a ray $\rho \in \Sigma(1)$ and an element $[\chi] \in C_\rho$, where $C_\rho$ is the cokernel of the natural map $M \to M_\rho$. For each ray $\rho$, we let $\chi_\rho \in M_\rho$ be the generator whose value on a generator of $\rho \cap N$ is positive (this distinguishes $\chi_\rho$ from the generator $-\chi_\rho$). We choose a splitting $M_\rho \otimes \bR \into M_\bR$ of $M_\bR \to M_\rho \otimes \bR$, which lets us regard $\chi_\rho$ as a point in $M_\bR$ such that $|C_\rho|\chi_\rho \in M$. In the notation of Section \ref{sec:CCC} the locus $\rho^\perp_{n \chi_\rho} \subset M_\bR$ can then be written as the translate $\rho^\perp + n \chi_\rho$. Moreover, $\Lambda_{\stackfano}$ itself can be written as
$$ 
\Lambda_{\stackfano} := \bigcup_{\rho \in \Sigma(1)}\bigcup_{n=0}^{|C_\rho|-1} p(\rho^\perp + n \chi_\rho) \times [-\rho].
$$
Here we have identified $T^\infty T^2$ with $T^2 \times (N_\bR \smallsetminus \{0\})/\bR_+$, and given a subset $S$ of $N_\bR$ we write $[S]$ for its image in $(N_\bR \smallsetminus \{0\})/\bR_+$.  

Given a ray $\rho$ of $\Sigma$ we now define a Legendrian isotopy $\{\Lambda_t^\rho\}_{t \in I}$ with $\Lambda^\rho_0 = \Lambda^\rho_1 = \Lambda_{\stackfano}$ by setting
$$ 
\Lambda^\rho_t := \left( \bigcup_{\stackrel{\rho' \in \Sigma(1)}{\rho' \neq \rho}}\bigcup_{n=0}^{|C_{\rho'}|-1} p((\rho')^\perp + n \chi_{\rho'}) \times [-\rho']\right) \cup \left( \bigcup_{n=0}^{|C_{\rho}|-1} p(\rho^\perp + (n+t) \chi_{\rho}) \times [-\rho] \right).
$$
That is, the components of $\Lambda_{\stackfano}$ labeled by the ray $\rho$ move around $T^2$ and become cyclically permuted at the end of the isotopy, while the remaining components are left stationary. The associated GKS functor $K_{\{\Lambda^\rho_t\}}$ is an autoequivalence of $\Sh_{\Lambda_{\stackfano}}(T^2)$.

On the coherent side we will relate this autoequivalence to a line bundle $\cL_\rho$ on $\cX_{\stackfan}$. It is characterized by the condition $\cL_\rho^{|C_\rho|} \cong \pi^* \cO(-D_\rho)$, where $D_\rho$ is the toric divisor in $X_\Sigma$ associated to $\rho$ and $\pi$ is the projection from $\cX_{\stackfan}$ to its coarse moduli space $X_\Sigma$. It can also be characterized on the constructible side as follows. Let $\rho_L, \rho_R \in \Sigma(1)$ be the rays adjacent to $\rho$ in the cyclic order on $\Sigma(1)$. Then the lines $\rho^\perp + \chi_\rho, (\rho_L)^\perp,$ and $(\rho_R)^\perp$ bound a unique closed triangle $T_\rho \subset M_\bR$ (unless $\rho_L = - \rho_R$, in which case $T_\rho$ degenerates to an interval). Then $\cL_\rho$ is the unique line bundle such that the closure of the support of $\cT_\rho := CCC_{\cX_{\stackfan}}(\cL_\rho)$ is equal to $p(T_\rho) \subset T^2$. 

\begin{proposition}\label{prop:mirrorisotopies}
Let $\stackfan$ be a complete two-dimensional stacky fan, $\rho \in \Sigma(1)$ a ray, and $i: \cX_{\stackfano} \into \cX_{\stackfan}$ the inclusion. Then the following diagram of functors commutes.
\[
\begin{tikzpicture}[thick,>=\arrtip]
\node (a) at (0,0) {$\Perf_{prop}(\cX_{\stackfano})$};
\node (b) at (5,0) {$\Sh^c_{\Lambda_{\stackfano}}(T^2)$};
\node (c) at (0,-2) {$\Perf_{prop}(\cX_{\stackfano})$};
\node (d) at (5,-2) {$\Sh^c_{\Lambda_{\stackfano}}(T^2)$};
\draw[->] (a) to node[above] {$\CCC_{\cX_{\stackfano}}$} node[below] {$\sim$} (b);
\draw[->] (b) to node[below, rotate=-90] {$\sim$} node[right] {$K_{\{\Lambda^\rho_t\}}$} (d);
\draw[->] (a) to node[below, rotate=90] {$\sim$} node[left] {$ i^*\cL_\rho \otimes -$}(c);
\draw[->] (c) to node[above] {$\CCC_{\cX_{\stackfano}}$} node[below] {$\sim$} (d);
\end{tikzpicture}
\]
\end{proposition}

\begin{proof}
Recall that the coherent-constructible correspondence is a monoidal equivalence intertwining the standard tensor product on the coherent side with the convolution product on the constructible side. Thus we obtain a canonically commutative diagram from the one above by replacing $K_{\{\Lambda^\rho_t\}}$ with the functor $CCC_{\cX_{\stackfano}}(i^*\cL_\rho) \star -$. 

We first claim that if $\cF \in \Sh^c_{\Lambda_{\stackfano}}$ then $CCC_{\cX_{\stackfano}}(i^*\cL_\rho) \star \cF \cong \cT_\rho \star \cF$, where as above  $\cT_\rho := CCC_{\cX_{\stackfan}}(\cL_\rho)$; note that $CCC_{\cX_{\stackfano}}(i^*\cL_\rho)$ and $CCC_{\cX_{\stackfan}}(\cL_\rho)$ are not themselves isomorphic, as the former will have infinite-rank stalks (since unlike $\cL_\rho$ the restriction $i^*\cL_\rho$ does not have proper support). On the other hand, if $\cG \in \Perf_{prop}(\cX_{\stackfano})$ then we do have $CCC_{\cX_{\stackfano}}(\cG) \cong CCC_{\cX_{\stackfan}}(i_* \cG)$. Setting $\widehat{\cF} = CCC_{\cX_{\stackfano}}^{-1}(\cF)$, we use this together with the push-pull isomorphism $i_*(i^*\cL_\rho \otimes \widehat{\cF}) \cong \cL_\rho \otimes i_*\widehat{\cF}$ to obtain
\begin{align*}
CCC_{\cX_{\stackfano}}(i^*\cL_\rho) \star \cF & \cong CCC_{\cX_{\stackfano}}(i^*\cL_\rho \otimes \widehat{\cF})\\
& \cong CCC_{\cX_{\stackfan}}(\cL_\rho \otimes i_*\widehat{\cF}) \\
& \cong \cT_\rho \star \cF.
\end{align*}

The sheaf $\cT_\rho$ is an example of a twisted polytope sheaf \cite{Zho17}. In the case at hand we have $\cT_\rho \cong p_!\widehat{\cT}_\rho$ for the following sheaf $\widehat{\cT}_\rho$ on $M_\bR$. First, let $m_L, m_R \in M_\bR$ denote the points where the line $\rho^\perp + \chi_\rho$ crosses the lines $\rho_L^\perp, \rho_R^\perp$, respectively. We write $I_{[m_L, m_R]} \subset M_\bR$ for the interval with endpoints $m_L, m_R$, similarly for $I_{[0, m_L]}, I_{[0,m_R]}$. They form the boundary of the (possibly degenerate) triangle $T_\rho$ discussed before the Proposition. We also write $\sigma_{[\rho_L,\, \rho_R]} \subset N_\bR$ for the cone whose boundary rays are $\rho_L, \rho_R$ but which does not contain $\rho$ in its interior, similarly for $\sigma_{[\rho,\, \rho_L]}, \sigma_{[\rho,\, \rho_R]}$ --- note that $\sigma_{\rho_L, \rho_R}$ need not be convex. 

The singular support of $\widehat{\cT}_\rho$ will be
\begin{align*} SS(\widehat{\cT}_\rho) =& (T_\rho \times \{0\}) \cup (I_{[m_L, m_R]} \times \rho) \cup (I_{[0, m_R]} \times \rho_R) \cup (I_{[m0, m_L]} \times \rho_L) \\
& \quad \cup (\{0\} \times \sigma_{[\rho_L,\, \rho_R]}) \cup (\{m_L\} \times \sigma_{[\rho,\, \rho_L]}) \cup (\{m_R\} \times \sigma_{[\rho,\, \rho_R]}),
\end{align*}
where as usual we identify $T^*M_\bR$ with $M_\bR \times N_\bR$. If $\rho$ is in the positive span of $\rho_L, \rho_R$, then $\widehat{\cT}_\rho$ is characterized by this singular support condition and the property that its stalk at any interior point of $T_\rho$ is $\coeffs[1]$. Otherwise we have $\widehat{\cT}_\rho \cong j_* \coeffs_{T_\rho}$. Examples of the two cases are pictured below, hairs indicating codirections of singular support.
$$
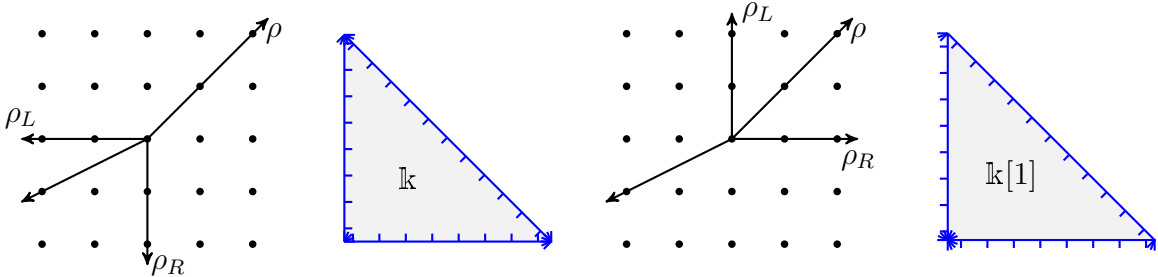
\begin{tikzpicture}
	\newcommand*{\edgelen}{5.5}; \newcommand*{\vertrad}{.12}; \newcommand*{\crad}{.05}
	\newcommand*{\gspace}{.7}; \newcommand*{\angdelta}{22.5};
	\node (c) [matrix] at (8,.29) {
		\foreach \x in {0,...,4} \foreach \y in {0,...,4}
			{\fill (\x*\gspace,\y*\gspace) circle (\crad);}
		\foreach \x/\y in {4.4/2,2/4.4,4.3/4.3,-.4/.8}
			{\draw[-stealth',thick] (2*\gspace,2*\gspace) to (\x*\gspace,\y*\gspace);}
		\node (e1) at (2.5*\gspace,4.4*\gspace) {$\rho_L$};
		\node (e2) at (4.4*\gspace,1.6*\gspace) {$\rho_R$};
		\node (e3) at (4.4*\gspace,4*\gspace) {$\rho$};
	\\};
	\node (d) [matrix] at (12,0) {
		\coordinate (tl) at (0,\edgelen); \coordinate (tr) at (\edgelen,\edgelen);
		\coordinate (bl) at (0,0); \coordinate (br) at (\edgelen,0); 
		\coordinate (tm) at (\edgelen/2,\edgelen); \coordinate (bm) at (\edgelen/2,0);
		\coordinate (mm) at (\edgelen/2,\edgelen/2);
		\path[fill=\fcolor] (bm) -- (mm) -- (br) -- (bm);
		\node at ($(bm)+(\edgelen/6.5,\edgelen/6.5)$) {$\coeffs[1]$};
		\foreach \a/\b in {mm/bm, bm/br,mm/br} {\draw[asdstyle,righthairs] (\a) to (\b);}
		\foreach \c in {bm} \foreach \ang in {12,...,24} 
			{\draw[blue,thick] ($(\c)+(\ang*\angdelta:\vertrad+.02)$) to (\c);}
		\foreach \ang in {8,...,10} 
			{\draw[blue,thick] ($(mm)+(\ang*\angdelta:\vertrad+.02)$) to (mm);}
		\foreach \ang in {10,...,12} 
			{\draw[blue,thick] ($(br)+(\ang*\angdelta:\vertrad+.02)$) to (br);}	\\};
\node (a) [matrix] at (0,-.1) {
		\foreach \x in {0,...,4} \foreach \y in {0,...,4}
			{\fill (\x*\gspace,\y*\gspace) circle (\crad);}
		\foreach \x/\y in {-.4/2,2/-.4,4.3/4.3,-.4/.8}
			{\draw[-stealth',thick] (2*\gspace,2*\gspace) to (\x*\gspace,\y*\gspace);}
		\node (e1) at (-.4*\gspace,2.4*\gspace) {$\rho_L$};
		\node (e2) at (2.4*\gspace,-.4*\gspace) {$\rho_R$};
		\node (e3) at (4.4*\gspace,4*\gspace) {$\rho$};
	\\};
\node (b) [matrix] at (4,0) {
		\coordinate (tl) at (0,\edgelen); \coordinate (tr) at (\edgelen,\edgelen);
		\coordinate (bl) at (0,0); \coordinate (br) at (\edgelen,0); 
		\coordinate (tm) at (\edgelen/2,\edgelen); \coordinate (bm) at (\edgelen/2,0);
		\coordinate (mm) at (\edgelen/2,\edgelen/2);
		\path[fill=\fcolor] (bm) -- (mm) -- (br) -- (bm);
		\node at ($(bm)+(\edgelen/6.5,\edgelen/6.5)$) {$\coeffs$};
		\foreach \a/\b in {bm/mm, br/bm,mm/br} {\draw[asdstyle,righthairs] (\a) to (\b);}
		\foreach \c in {bm} \foreach \ang in {0,...,4} 
			{\draw[blue,thick] ($(\c)+(\ang*\angdelta:\vertrad+.02)$) to (\c);}
		\foreach \ang in {-6,...,0} 
			{\draw[blue,thick] ($(mm)+(\ang*\angdelta:\vertrad+.02)$) to (mm);}
		\foreach \ang in {4,...,10} 
			{\draw[blue,thick] ($(br)+(\ang*\angdelta:\vertrad+.02)$) to (br);}	\\};
\end{tikzpicture}
$$
In each case, different choices of stacky structure on the pictured fan result in different scalings of the pictured triangle. 

We now define a sheaf $\cT_{\rho,I} \in \Sh(T^2 \times I)$ as follows. Let $s: M_\bR \times I \to M_\bR \times I$ be the scaling map $(m, t) \mapsto (tm, t)$. Then we set $\cT_{\rho,I} := (p \times id_I)_! s_!(\widehat{\cT}_\rho \boxtimes \coeffs_I)$. We have $i_{T^2_1}^*\cT_{\rho,I} \cong \cT_\rho$ by construction, where as usual we write $T^2_t$ for $T^2 \times \{t\} \subset T^2 \times I$. Moreover, $i_{T^2_0}^*\cT_{\rho,I}$ is the skyscraper sheaf $\coeffs_{\{0\}}$. Indeed, $i_{T^2_0}^*\cT_{\rho,I}$ is a priori a skyscraper at $0$ with stalk $\Gamma(\cT_\rho)$, but $\Gamma(\cF) \cong \coeffs$ whenever $CCC_{\cX_{\stackfan}}^{-1}(\cF)$ is a line bundle $\cL$ since $\Gamma(\cF) \cong \Hom(\coeffs_{T^2},\cF) \cong \Hom(\cO_{(1,1)}[-2], \cL).$ Taking singular support we obtain a family $\{ SS^\infty (i_{T^2_t}^*\cT_{\rho,I})\}$ of piecewise smooth Legendrians homeomorphic to $S^1$. 

Recall from Corollary \ref{cor:LegendrianGKS} that $K_{\{\Lambda^\rho_t\}}$ is defined by composing the equivalences
\begin{align*}\Sh^c_{\Lambda_0^\rho}(T^2) \xleftarrow[i_{T^2_0}^*]{\sim} \Sh^c_{\Lambda_I^\rho}(T^2 \times I) \xrightarrow[i_{T^2_1}^*]{\sim} \Sh^c_{\Lambda_1^\rho}(T^2). 
\end{align*}
On the other hand, proper pushforward along $m \times id_I: T^2 \times T^2 \times I \to T^2 \times I$
defines a pointwise convolution functor $\cT_{\rho,I} \star -: \Sh^c_{\Lambda_{\stackfano}}(T^2) \to \Sh^c(T^2 \times I)$. We claim this provides the inverse to the restriction $i_{T^2_0}^*$ above. This is equivalent to claiming that for any $\cF \in \Sh^c_{\Lambda_0^\rho}(T^2)$ the sheaf $\cT_{\rho,I} \star \cF$ has singular support on $\Lambda_I^\rho$ and satisfies $i_{T^2_0}^*(\cT_{\rho,I} \star \cF) \cong \cF$. The latter claim follows since convolution and $i_{T^2_0}^*$ commute by base change and since $i_{T^2_0}^* \cT_{\rho,I} \cong \coeffs_{\{0\}}$. 

The fact that $\cT_{\rho,I} \star \cF$ has singular support on $\Lambda_I^\rho$ will follow from the bound on singular support of a proper pushforward given by \cite[Prop. 5.4.4]{KS94}. Here it says that
\begin{equation}\label{eq:sscontainment}
\begin{aligned}
SS(\cT_{\rho,I} \star \cF) \subset &\{(p(m), n, t, \tau) \in T^*T^2 \times T^*I \text{ such that } \exists(p(m_1), n) \in SS(\cF),\\ &\quad  (p(m_2), n, t, \tau) \in SS(\cT_{\rho,I}) \text{ with }p(m_1+m_2)=p(m)\}.
\end{aligned}
\end{equation}
Let $\beta_\rho \in \rho$ be the generator of $\rho \cap \beta(\widehat{N}) \cong \bN$, i.e. the element of $\rho$ which pairs to 1 with $\chi_\rho$. Explicitly we can then write
\begin{gather*}
\Lambda^\rho_I = \left( \bigcup_{\stackrel{\rho' \in \Sigma(1)}{\rho' \neq \rho}}\bigcup_{n=0}^{|C_{\rho'}|-1} p((\rho')^\perp + n \chi_{\rho'}) \times [-\rho'] \times I \right) \cup \left( \bigcup_{n=0}^{|C_{\rho}|-1} p(\rho^\perp + (n+t) \chi_{\rho}) \times [-\beta_\rho + dt] \right).
\end{gather*} 

\begin{figure}
\begin{tikzpicture}
\newcommand*{\scl}{2}
\node (a) [matrix] at (0,0) {
\draw[thick, rotate=180] (\scl*1,\scl*0)--(\scl*1,\scl*2)--(\scl*3,\scl*2)--(\scl*3,\scl*0)--(\scl*1,\scl*0);
\draw[lefthairs, thick, blue, rotate=180] ($\scl*(1+.2,2)$)--($\scl*(1+.2,0)$);
\draw[lefthairs, thick, blue, rotate=180] ($\scl*(1.2+2*.333,2)$)--($\scl*(1+.2+2*.333,0)$);
\draw[lefthairs, thick, blue, rotate=180] ($\scl*(1.2+2*.666,2)$)--($\scl*(1+.2+2*.666,0)$);
\draw[lefthairs, thick, blue, rotate=180] ($\scl*(1,0.2)$)--($\scl*(3,0.2)$);
\draw[lefthairs, thick, blue, rotate=180] ($\scl*(1,1.2)$)--($\scl*(3,1.2)$);
\draw[lefthairs, thick, blue, rotate=180] ($\scl*(1.4,0)$)--($\scl*(1,0.4)$);
\draw[lefthairs, thick, blue, rotate=180] ($\scl*(3,0.4)$)--($\scl*(1.4,2)$);
\draw[lefthairs, thick, blue, rotate=180] ($\scl*(3,1.4)$)--($\scl*(2.4,2)$);
\draw[lefthairs, thick, blue, rotate=180] ($\scl*(2.4,0)$)--($\scl*(1,1.4)$);\\};
\node (b) [matrix] at (5.5,0) {
\draw[thick, rotate=180] (\scl*1,\scl*0)--(\scl*1,\scl*2)--(\scl*3,\scl*2)--(\scl*3,\scl*0)--(\scl*1,\scl*0);
\draw[lefthairs, thick, blue, rotate=180] ($\scl*(1.533,2)$)--($\scl*(1.533,0)$);
\draw[lefthairs, thick, blue, rotate=180] ($\scl*(1.533+2*.333,2)$)--($\scl*(1.533+2*.333,0)$);
\draw[lefthairs, thick, blue, rotate=180] ($\scl*(1.533+2*.666,2)$)--($\scl*(1.533+2*.666,0)$);
\draw[lefthairs, thick, blue, rotate=180] ($\scl*(1,0.2)$)--($\scl*(3,0.2)$);
\draw[lefthairs, thick, blue, rotate=180] ($\scl*(1,1.2)$)--($\scl*(3,1.2)$);
\draw[lefthairs, thick, blue, rotate=180] ($\scl*(1.4,0)$)--($\scl*(1,0.4)$);
\draw[lefthairs, thick, blue, rotate=180] ($\scl*(3,0.4)$)--($\scl*(1.4,2)$);
\draw[lefthairs, thick, blue, rotate=180] ($\scl*(3,1.4)$)--($\scl*(2.4,2)$);
\draw[lefthairs, thick, blue, rotate=180] ($\scl*(2.4,0)$)--($\scl*(1,1.4)$);\\};\node (c) [matrix] at (11,0) {
\draw[thick, rotate=180] (\scl*1,\scl*0)--(\scl*1,\scl*2)--(\scl*3,\scl*2)--(\scl*3,\scl*0)--(\scl*1,\scl*0);
\draw[lefthairs, thick, blue, rotate=180] ($\scl*(1+.2,2)$)--($\scl*(1+.2,0)$);
\draw[lefthairs, thick, blue, rotate=180] ($\scl*(1.2+2*.333,2)$)--($\scl*(1+.2+2*.333,0)$);
\draw[lefthairs, thick, blue, rotate=180] ($\scl*(1.2+2*.666,2)$)--($\scl*(1+.2+2*.666,0)$);
\draw[lefthairs, thick, blue, rotate=180] ($\scl*(1,0.2)$)--($\scl*(3,0.2)$);
\draw[lefthairs, thick, blue, rotate=180] ($\scl*(1,1.2)$)--($\scl*(3,1.2)$);
\draw[lefthairs, thick, blue, rotate=180] ($\scl*(1.4,0)$)--($\scl*(1,0.4)$);
\draw[lefthairs, thick, blue, rotate=180] ($\scl*(3,0.4)$)--($\scl*(1.4,2)$);
\draw[lefthairs, thick, blue, rotate=180] ($\scl*(3,1.4)$)--($\scl*(2.4,2)$);
\draw[lefthairs, thick, blue, rotate=180] ($\scl*(2.4,0)$)--($\scl*(1,1.4)$);\\};
\node (alabel) at (0,-2.7) {$\Lambda^\rho_0 := \Lambda_{\stackfano}$};
\node (blabel) at (5.5,-2.7) {$\Lambda^\rho_{t=\frac{1}{2}}$};
\node (clabel) at (11,-2.7) {$\Lambda^\rho_1 := \Lambda_{\stackfano}$};
\draw[arrowstyle] ($(a) + (2.2,0)$) to ($(b) + (-2.2,0)$);
\draw[arrhdstyle] ($(b) + (-2.2,0)$) to ($(b) + (-2.2,0)+(-1.2mm,-1.2mm)$);
\draw[arrhdstyle] ($(b) + (-2.2,0)$) to ($(b) + (-2.2,0)+(-1.2mm,1.2mm)$);
\draw[arrowstyle] ($(b) + (2.2,0)$) to ($(c) + (-2.2,0)$);
\draw[arrhdstyle] ($(c) + (-2.2,0)$) to ($(c) + (-2.2,0)+(-1.2mm,-1.2mm)$);
\draw[arrhdstyle] ($(c) + (-2.2,0)$) to ($(c) + (-2.2,0)+(-1.2mm,1.2mm)$);
\vspace{-50mm}
\end{tikzpicture}
\caption{The family $\Lambda^\rho_t$ for $\rho$ the horizontal ray (corresponding to the vertical blue strands) in a fan for a stacky $\bP^2$.}
\label{fig:lambdarhot}
\end{figure}
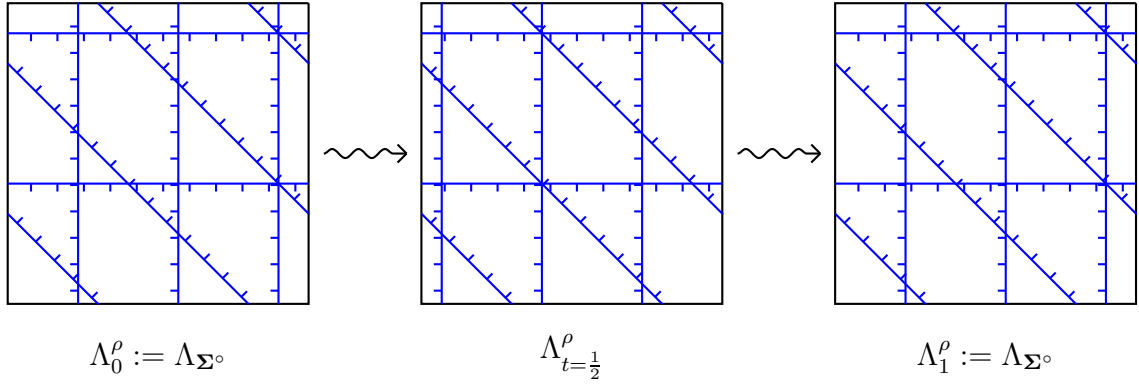

From this we observe that to show $SS(\cT_{\rho,I} \star \cF) \subset \Lambda_I^\rho$ it suffices to show that 
$$SS(\cT_{\rho,I}) \cap \big(T^2 \times \left(\bigcup_{\rho' \in \Sigma(1)} [-\rho']\right) \times T^*I\big) \subset \Lambda^\rho_I.$$
This follows since $1)$ the only covectors that appear in $SS(\cF) \subset \Lambda_{\stackfano}$ are of the form $[-\rho']$ for some $\rho' \in \Sigma(1)$, and $2)$ for any $t$ the subset $\bigcup_n p(\rho^\perp + (n+t) \chi_{\rho})$ is closed under multiplication by $\bigcup_n p(\rho^\perp + n \chi_{\rho})$ and for any $\rho'$ the subset $\bigcup_n p((\rho')^\perp + n \chi_{\rho'})$ is closed under multiplication by itself. 

One can now explicitly check that the left hand side of eq. \ref{eq:sscontainment} is in fact contained in 
\begin{equation*}
\left(\bigcup_{{\rho' \neq \rho, \rho_L, \rho_R}} \{0\} \times [-\rho'] \times I\right) \cup \left( \bigcup_{i \in \{L,R\}}p(\rho_i^\perp) \times [-\rho_i] \times I\right) \cup \big(p(\rho^\perp + t \chi_{\rho}) \times [-\beta_\rho + dt]\big),
\end{equation*}
which in turn is indeed contained in $\Lambda_{\rho,I}$ (it is straightforward to identify the left hand side more explicitly given the singular support of $\cT_{\rho,I}$, but the given bound is simpler to write and is sufficient). 

Now that we have shown $\cT_{\rho,I} \star -$ supplies the inverse to the restriction functor $i_{T^2_0}^*: \Sh^c_{\Lambda_I^\rho}(T^2 \times I) \to \Sh^c_{\Lambda_0^\rho}(T^2)$, the Proposition follows: by base change we have
$$K_{\{\Lambda^\rho_t\}} \cong i_{T^2_1}^*(\cT_{\rho,I} \star -) \cong i_{T^2_1}^*(\cT_{\rho,I}) \star - \cong \cT_\rho \star -,$$
and as explained at the beginning of the proof the functor on the right is intertwined with $i^*\cL_\rho \otimes -$ by the $CCC$.
\end{proof}

\section{Coarse moduli spaces and Legendrian degenerations}\label{sec:satellites}
\newcommand{\leftparen}{(}    
\newcommand{\rightparen}{)}  
\newcommand{\leftbracket}{[}
\newcommand{\rightbracket}{]}
\newcommand{\zeroone}{\leftbracket 0,1\rightparen}  
\newcommand{\zo}{\zeroone}
\newcommand{\Cyl}{\mathrm{Cyl}}

Suppose $\Gamma$ is a consistent bipartite graph whose Newton polygon has non-primitive edges. To obtain a faithful mirror description of $\Sh_{\La_\Gamma}(T^2)$ one must consider the toric stack $\cX_{\stackfano}$, which in this case is not isomorphic to its coarse moduli space, the toric variety $X_{\Sigmao}$. That is, the pushforward $\pi_*: \Coh(\cX_{\stackfano}) \to \Coh(X_{\Sigmao})$ is not fully faithful. In this section we explicitly describe the mirror functor $\Sh^w_{\Lambda_{\stackfano}}(T^2) \to \Sh^w_{\Lambda_{\Sigmao}}(T^2)$ as an example of a general class of functors associated to Legendrian satellites and degenerations. This allows one to reformulate the main content of Theorem \ref{thm:mainresult} purely in terms of ordinary varieties rather than stacks. Moreover, it illuminates the result of \cite{GK} that the cluster integrable systems considered in loc. cit. are in general finite covers of those considered by e.g. Beauville: the former are directly related to the stack $\cX_{\stackfan}$, the latter to the variety $X_{\Sigma}$.

We call a family $\{\Lambda_t\}_{t\in I}$ of Legendrians in $T^\infty M$ a Legendrian degeneration if 
\begin{enumerate}
\item the total space of the family in $T^*M \times I$ is closed, and
\item the family is an isotopy of smooth Legendrians for $t \in [0,1) \subset I$.
\end{enumerate}
In particular $\Lambda_1$ need not be homeomorphic to $\Lambda_0$ nor even smooth, though it will be smooth in our main example:  effectively, we generalize the notion of a Legendrian isotopy to allow more complicated behavior at $t=1$.

A special case of this notion is that of a Legendrian satellite. Suppose $p: K \onto K'$ is a covering space map of compact 1-manifolds, and $\Cyl(p)$ the mapping cylinder
\[
\Cyl(p) = ((I \times K) \amalg K' )/\sim \qquad (1,k) \sim p(k) \text{ for } k \in K.
\]
Then $\{\Lambda_t\}_{t\in I}$ is a Legendrian satellite if its total space in $T^*M \times I$ is the image of an embedding of $\Cyl(p)$ compatible with its projection to $I$. More precisely, for $t$ sufficiently close to $1$ the link $\Lambda_t$ lives in a tubular neighborhood of $\Lambda_1$ and is a Legendrian satellite of $\Lambda_1$ in the sense of e.g. \cite{Ng01}. The family $\{\Lambda_t\}_{t \in I}$ can then be thought of as a way of recording a particular realization of $\Lambda_t$ as a satellite.

To a Legendrian degeneration we may associate a functor
$$ K_{\{\Lambda_t\}}: \Sh_{\Lambda_0}(M) \to \Sh_{\Lambda_1}(M) $$
which generalizes the GKS equivalence attached to a Legendrian isotopy (though is no longer an equivalence in general). Recall as in Corollary \ref{cor:LegendrianGKS} that restriction from $M \times [0,1)$ to $M \times \{0\}$ yields an equivalence $\Sh_{\Lambda_{[0,1)}}(M \times [0,1)) \congto \Sh_{\Lambda_0}(M)$, where $\Lambda_{[0,1)} \subset T^*(M \times [0,1))$ is the Legendrian which lifts the isotopy $\{\Lambda_{t}\}_{t \in [0,1)}$. We now define $K_{\{\Lambda_t\}}$ to be the composition 
\begin{equation}\label{eq:blurdef} 
\Sh_{\Lambda_0}(M) \xrightarrow[(i_0^*)^{-1}]{\sim} \Sh_{\Lambda_{[0,1)}}(M \times [0,1)) \xrightarrow[(id_M \times i_{[0,1)})_*]{} \Sh(M \times I) \xrightarrow[(id_M \times i_1)^*]{} \Sh(M). 
\end{equation}

\begin{lemma}
\cite[Prop. 2.12]{Zho18} The composition (\ref{eq:blurdef}) takes values in $\Sh_{\Lambda_1}(M)$, hence defines a functor $K_{\{\Lambda_t\}}: \Sh_{\Lambda_0}(M) \to \Sh_{\Lambda_1}(M)$.
\end{lemma}
\begin{proof}
Suppose $\cF \in \Sh_{\Lambda_{[0,1)}}(M \times [0,1))$. By \cite[Thm.~6.3.1]{KS94} if $(p(m), [n], 1, \tau)$ is a point of $SS((id_M \times i_{[0,1)})_*\cF)$ then for some $\tau' \in T^*_1I$ the point $(p(m), [n], 1, \tau')$ is in the closure of $SS(\cF) \subset T^*(M \times [0,1))$ in $T^*(M \times I)$. Since the total space of $\{\Lambda_t\}_{t \in I}$ is closed it follows that $[n] \in \Lambda_1$. On the other hand, restriction to $M \times \{1\}$ acts on singular support by intersecting with $T^*M \times T^*_1 I$ then projecting to $T^*M$ \cite[Prop. 5.4.5]{KS94}, hence the result follows. 
\end{proof}

\begin{example}
The following figure illustrates a Legendrian degeneration of a $3$-strand braid to a single strand.  An object of $\Sh_{\Lambda_0}$ and its image under \eqref{eq:blurdef} are indicated on the left and right.
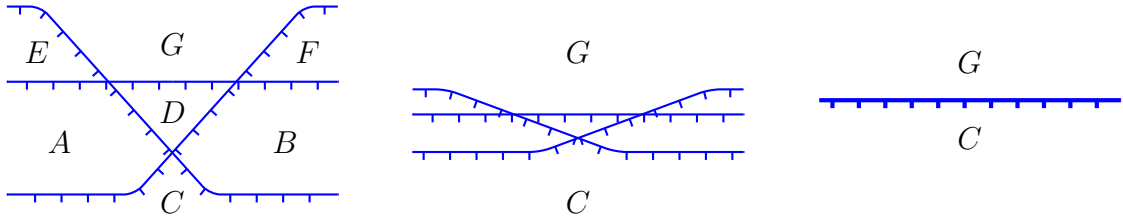
\vspace{4mm}
\begin{center}
	\begin{tikzpicture}
	\draw[righthairs,blue,thick,rounded corners] (-2.2,1.5)--(-1.75,1.5)--(.5,-1)--(2.2,-1);
	\draw[lefthairs,blue,thick,rounded corners] (2.2,1.5)--(1.75,1.5)--(-.5,-1)--(-2.2,-1);
	\draw[righthairs,blue,thick] (0,.5)--(2.2,.5);
	\draw[lefthairs,blue,thick] (0,.5)--(-2.2,.5);
	\node at (1.5,-.3) {$B$};
	\node at (-1.5,-.3) {$A$};
	\node at (0,-1.1) {$C$};
	\node at (0,.1) {$D$};
	\node at (-1.8,.9) {$E$};
	\node at (1.8,.9) {$F$};
	\node at (0,1) {$G$};
	\end{tikzpicture}
	\qquad
	\begin{tikzpicture}
	\draw[righthairs,blue,thick,rounded corners] (-2.2,1.5/3)--(-1.75,1.5/3)--(.5,-1/3)--(2.2,-1/3);
	\draw[lefthairs,blue,thick,rounded corners] (2.2,1.5/3)--(1.75,1.5/3)--(-.5,-1/3)--(-2.2,-1/3);
	\draw[righthairs,blue,thick] (0,.5/3)--(2.2,.5/3);
	\draw[lefthairs,blue,thick] (0,.5/3)--(-2.2,.5/3);
	
	\node at (0,-1) {$C$};
	\node at (0,1) {$G$};
	\end{tikzpicture}
	\qquad
	\begin{tikzpicture}
	\draw[righthairs,blue,ultra thick] (-2,0)--(2,0);
	\node at (0,-.5) {$C$};
	\node at (0,.5) {$G$};
	\node at (0,-1.5) {};
	\end{tikzpicture}
\end{center}
In the sheaf on the left, the several ways of defining a map from $C \to G$ by composing strand-crossing maps ($C \to B \to E \to G$ and $C \to A \to F \to G$, etc.) coincide; this composition supplies the strand-crossing map in the degenerated sheaf on the right.  The thicker paintbrush used in the right part of the figure figure conveys (to us) a sense in which a Legendrian degeneration $\Lambda_1$ is obtained by ``blurring'' some of the features of $\Lambda_0$ --- sheaves on $\Lambda_1$ are likewise obtained by blurring some of the features of sheaves on $\Lambda_0$.
\end{example}

Now we return to the setting of previous sections, with $M \cong \bZ^2$, $\stackfan$ a stacky fan, and $\Sigma$ its underlying ordinary fan. We do, however, allow $\stackfan$ to be non-complete. We have the Legendrian links $\Lambda_{\stackfano}$ and $\Lambda_{\Sigmao}$ in $T^\infty T^2$, the latter being a closed subset of the former. We realize $\Lambda_{\stackfano}$ as a Legendrian satellite of $\Lambda_{\Sigmao}$ as described below. Informally, each component of $\Lambda_{\stackfano}$
 is isotopic to a unique component of $\Lambda_{\Sigmao},$ and moves towards it at uniform speed as $t$ goes from $0$ to $1$, colliding with it at $t=1$.

 Recall again that the components of $\Lambda_{\stackfano}$ are in correspondence with pairs of a ray $\rho \in \Sigma(1)$ and an element $[\chi] \in C_\rho$, where $C_\rho$ is the cokernel of the natural map $M \to M_\rho$. As in Section \ref{sec:isotopies}, given a ray $\rho$ we let $\chi_\rho \in M_\rho$ be the generator whose value on a generator of $\rho \cap N$ is positive. We use a splitting $M_\rho \otimes \bR \into M_\bR$ of $M_\bR \to M_\rho \otimes \bR$ to regard $\chi_\rho$ as a point in $M_\bR$ such that $|C_\rho|\chi_\rho \in M$. We now consider the family 
\begin{equation}
\label{eq:Lambda_t}
\Lambda_t := \bigcup_{\rho \in \Sigma(1)} \bigcup_{n = 0}^{|C_\rho|-1} p(\rho^\perp + n t\chi_\rho) \times [-\rho].
\end{equation}
As before, we have identified $T^\infty T^2$ with $T^2 \times (N_\bR \smallsetminus \{0\})/\bR_+$, and given a subset $S$ of $N_\bR$ we write $[S]$ for its image in $(N_\bR \smallsetminus \{0\})/\bR_+$. 
We have $\Lambda_0 = \Lambda_{\stackfano}$ and $\Lambda_1 = \Lambda_{\Sigmao}$, while $\Lambda_t$ is isotopic to $\Lambda_{\stackfano}$ for $t \in (0,1)$. As a satellite this is somewhat trivial: for $t$ close to $1$ the link $\Lambda_t$ meets a tubular neighborhood of the component of $\Lambda_{\Sigmao}$ attached to $\rho$ along an unlink with $|C_\rho|$ strands. The main result of this section is to identify the counterpart of the associated degeneration functor 
$$ K_{\{\Lambda_t\}}: \Sh_{\Lambda_{\stackfano}}(T^2) \to \Sh_{\Lambda_{\Sigmao}}(T^2)$$
under the coherent-constructible correspondence.

\begin{figure}
\begin{tikzpicture}
\newcommand*{\scl}{2}
\node (a) [matrix] at (0,0) {
\draw[thick, rotate=180] (\scl*1,\scl*0)--(\scl*1,\scl*2)--(\scl*3,\scl*2)--(\scl*3,\scl*0)--(\scl*1,\scl*0);
\draw[lefthairs, thick, blue, rotate=180] ($\scl*(1+.2,2)$)--($\scl*(1+.2,0)$);
\draw[lefthairs, thick, blue, rotate=180] ($\scl*(1.2+2*.333,2)$)--($\scl*(1+.2+2*.333,0)$);
\draw[lefthairs, thick, blue, rotate=180] ($\scl*(1.2+2*.666,2)$)--($\scl*(1+.2+2*.666,0)$);
\draw[lefthairs, thick, blue, rotate=180] ($\scl*(1,0.2)$)--($\scl*(3,0.2)$);
\draw[lefthairs, thick, blue, rotate=180] ($\scl*(1,1.2)$)--($\scl*(3,1.2)$);
\draw[lefthairs, thick, blue, rotate=180] ($\scl*(1.4,0)$)--($\scl*(1,0.4)$);
\draw[lefthairs, thick, blue, rotate=180] ($\scl*(3,0.4)$)--($\scl*(1.4,2)$);
\draw[lefthairs, thick, blue, rotate=180] ($\scl*(3,1.4)$)--($\scl*(2.4,2)$);
\draw[lefthairs, thick, blue, rotate=180] ($\scl*(2.4,0)$)--($\scl*(1,1.4)$);\\};
\node (b) [matrix] at (5.5,0) {
\draw[thick, rotate=180] (\scl*1,\scl*0)--(\scl*1,\scl*2)--(\scl*3,\scl*2)--(\scl*3,\scl*0)--(\scl*1,\scl*0);
\draw[lefthairs, thick, blue, rotate=180] ($\scl*(1+.2,2)$)--($\scl*(1+.2,0)$);
\draw[lefthairs, thick, blue, rotate=180] ($\scl*(1.2+2*.25*.333,2)$)--($\scl*(1+.2+2*.25*.333,0)$);
\draw[lefthairs, thick, blue, rotate=180] ($\scl*(1.2+2*.25*.666,2)$)--($\scl*(1+.2+2*.25*.666,0)$);
\draw[lefthairs, thick, blue, rotate=180] ($\scl*(1,0.2)$)--($\scl*(3,0.2)$);
\draw[lefthairs, thick, blue, rotate=180] ($\scl*(1,1.2-.75)$)--($\scl*(3,1.2-.75)$);
\draw[lefthairs, thick, blue, rotate=180] ($\scl*(1.4,0)$)--($\scl*(1,0.4)$);
\draw[lefthairs, thick, blue, rotate=180] ($\scl*(3,0.4)$)--($\scl*(1.4,2)$);
\draw[lefthairs, thick, blue, rotate=180] ($\scl*(3,0.4-.25)$)--($\scl*(1.4-.25,2)$);
\draw[lefthairs, thick, blue, rotate=180] ($\scl*(2.4+.75-2,0)$)--($\scl*(1,1.4+.75-2)$);\\};
\node (c) [matrix] at (11,0) {
\draw[thick, rotate=180] (\scl*1,\scl*0)--(\scl*1,\scl*2)--(\scl*3,\scl*2)--(\scl*3,\scl*0)--(\scl*1,\scl*0);
\draw[lefthairs, thick, blue, rotate=180] ($\scl*(1+.2,2)$)--($\scl*(1+.2,0)$);
\draw[lefthairs,thick, blue, rotate=180] ($\scl*(1,0.2)$)--($\scl*(3,0.2)$);
\draw[lefthairs, thick, blue, rotate=180] ($\scl*(1.4,0)$)--($\scl*(1,0.4)$);
\draw[lefthairs, thick, blue, rotate=180] ($\scl*(3,0.4)$)--($\scl*(1.4,2)$);
\\};
\node (alabel) at (0,-2.7) {$\Lambda_0 := \Lambda_{\stackfano}$};
\node (blabel) at (5.5,-2.7) {$\Lambda_{t}$};
\node (clabel) at (11,-2.7) {$\Lambda_1 := \Lambda_{\Sigmao}$};
\draw[arrowstyle] ($(a) + (2.2,0)$) to ($(b) + (-2.2,0)$);
\draw[arrhdstyle] ($(b) + (-2.2,0)$) to ($(b) + (-2.2,0)+(-1.2mm,-1.2mm)$);
\draw[arrhdstyle] ($(b) + (-2.2,0)$) to ($(b) + (-2.2,0)+(-1.2mm,1.2mm)$);
\draw[arrowstyle] ($(b) + (2.2,0)$) to ($(c) + (-2.2,0)$);
\draw[arrhdstyle] ($(c) + (-2.2,0)$) to ($(c) + (-2.2,0)+(-1.2mm,-1.2mm)$);
\draw[arrhdstyle] ($(c) + (-2.2,0)$) to ($(c) + (-2.2,0)+(-1.2mm,1.2mm)$);
\vspace{-50mm}
\end{tikzpicture}
\caption{The Legendrian degeneration $\Lambda_{\stackfano} \to \Lambda_{\Sigmao}$ for a stacky $\bP^2$ with stabilizers of cardinality 3, 2, and 2 on its toric divisors.}
\label{fig:blurringfig}
\end{figure}

\begin{theorem}
The coherent-constructible correspondence intertwines $K_{\{\Lambda_t\}}$ with the pushforward $\pi_*: \Coh(\cX_{\stackfano}) \to \Coh(X_{\Sigmao})$. That is, $K_{\{\Lambda_t\}}$ restricts to a functor of wrapped sheaf categories and we have a commuting diagram of functors
\[
\begin{tikzpicture}
[thick,>=\arrtip]
\node (a) at (0,0) {$\Coh(\cX_{\stackfano})$};
\node (b) at (5,0) {$\Sh^w_{\Lambda_{\stackfano}}(T^2)$};
\node (c) at (0,-2) {$\Coh(X_{\Sigmao})$};
\node (d) at (5,-2) {$\Sh^w_{\Lambda_{\Sigmao}}(T^2).$};
\draw[->] (a) to node[above] {$\CCC_{\cX_{\stackfano}}$} node[below] {$\sim$} (b);
\draw[->] (b) to node[right] {$K_{\{\Lambda_t\}}$} (d);
\draw[->] (a) to node[left] {$\pi_*$}(c);
\draw[->] (c) to node[above] {$\CCC_{X_{\Sigmao}}$} node[below] {$\sim$} (d);
\end{tikzpicture}
\]
\end{theorem}

\begin{proof}
We begin by recalling that since $\pi$ is proper the coherent-constructible correspondence intertwines the pullback $\pi^*: \Coh(X_{\Sigmao}) \into \Coh(X_{\stackfano})$ with the trivial inclusion $id^*_{T^2}: \Sh^w_{\Lambda_{\Sigmao}}(T^2) \into \Sh^w_{\Lambda_{\stackfano}}(T^2)$ (by \cite[Prop. 9.3]{Kuw16} in the present generality, following \cite{FLTZ, Tre10, SS14}). That is, we have a diagram
\[
\begin{tikzpicture}
[thick,>=\arrtip]
\node (a) at (0,0) {$\Coh(\cX_{\stackfano})$};
\node (b) at (5,0) {$\Sh^w_{\Lambda_{\stackfano}}(T^2)$};
\node (c) at (0,-2) {$\Coh(X_{\Sigmao})$};
\node (d) at (5,-2) {$\Sh^w_{\Lambda_{\Sigmao}}(T^2)$};
\draw[->] (a) to node[above] {$\CCC_{\cX_{\stackfano}}$} node[below] {$\sim$} (b);
\draw[<-right hook] (b) to node[right] {$id^*_{T^2}$} (d);
\draw[<-right hook] (a) to node[left] {$\pi^*$}(c);
\draw[->] (c) to node[above] {$\CCC_{X_{\Sigmao}}$} node[below] {$\sim$} (d);
\end{tikzpicture}
\]

A priori $K_{\{\Lambda_t\}}$ gives a functor between the large sheaf categories $\Sh_{\Lambda_{\stackfano}}(T^2)$ and $\Sh_{\Lambda_{\Sigmao}}(T^2)$. The above diagram extends to one involving these categories on the right and $\IndCoh$ on the left. On the other hand, since $\pi_*$ is the right adjoint of $\pi^*$, the Theorem follows once we establish the corresponding adjunction between $K_{\{\Lambda_t\}}$ and $id^*_{T^2}$ --- in particular it will follow that $K_{\{\Lambda_t\}}$ preserves wrapped sheaf categories since $\pi_*$ preserves $\Coh$. We will show this adjunction directly in the degenerate case when $\Sigma$ has a single ray, then derive the general case by reducing to an affine cover.

When $\Sigma$ has a single ray $\rho$, consider the following continuous map $\psi: T^2 \to T^2$. Any point in $T^2$ can be written as $p(x + y \chi_\rho)$ for some $x \in \rho^\perp_0$ and a unique $y \in \bR$ with $0 \leq y < |C_\rho|$. We then define $\psi$ so that on such a point we have
\begin{equation*}
\psi(p(x + y \chi_\rho)) = \begin{cases} p(x + |C_\rho|(y+1-|C_\rho|)\chi_\rho) & |C_\rho|-1 \leq y <|C_\rho| \\
p(x) & 0 \leq y < |C_\rho|-1. \end{cases}
\end{equation*}
That is, $\psi$ retracts an annulus containing the front projection of $\Lambda_{\stackfano}$ onto the front projection of $\Lambda_{\Sigmao}$ (which consists of a single geodesic). The needed adjunction now follows from the straightforward observation that $K_{\{\Lambda_t\}} \cong \psi_*$, while the trivial inclusion is isomorphic to $\psi^*$.

For general $\Sigma$ we denote by $\Sigma_\rho$ the subfan consisting of $\{0\}$ and a single ray $\rho \in \Sigma(1)$, similarly for $\stackfan_\rho$. Consider the diagram formed by the dg categories $\IndCoh(\cX_{\stackfan_\rho})$ together with their restriction functors to $\IndCoh(T_N)$. By Zariski descent $\IndCoh(\cX_{\stackfano})$ is the limit of this diagram, similarly for $\IndCoh(X_{\Sigmao})$. In particular, the functor $\pi_*: \IndCoh(\cX_{\stackfano}) \to\IndCoh(X_{\Sigmao})$ is completely determined by the fact that under restriction it is intertwined with the local pushforwards $\IndCoh(\cX_{\stackfano_\rho}) \to\IndCoh(X_{\Sigmao_\rho})$ in a way further compatible with restriction to $T_N$. In particular, we have a commutative diagram
\[
\begin{tikzpicture}[thick,>=\arrtip]
\newcommand*{\width}{4.5}; \newcommand*{\hta}{2}; \newcommand*{\htb}{2};
\node (left) [matrix] at (8,0) {
	\node (a) at (0,0) {$\IndCoh(\cX_{\stackfano})$};
	\node (b) at (\width,0) {$\IndCoh(X_{\Sigmao})$};
	\node (c) at (0,-\hta) {$\IndCoh(\cX_{\stackfan_\rho})$};
	\node (d) at (\width,-\hta) {$\IndCoh(X_{\Sigma_\rho}).$};
	\draw[->] (a) to node[above] {$\pi_*$} (b);
	\draw[->] (c) to node[above] {$\pi_*$} (d);
	\draw[->] (a) to node[left] {$i_{\cX_{\stackfan_\rho}}^*$} (c); 
	\draw[->] (b) to node[right] {$i_{X_{\Sigma_\rho}}^*$} (d);
	\\};
\end{tikzpicture}
\]
To show that the CCC identifies $\pi_*$ and $K_{\{\Lambda_t\}}$ it thus suffices to show that $K_{\{\Lambda_t\}}$ has the corresponding compatibilities on the constructible side of the CCC, since we have shown the CCC identifies pushforward and degeneration for each $\stackfan_\rho$. 

More explicitly, we denote by $K_{\{ \Lambda_t^\rho\}}: \Sh_{\Lambda_{\stackfan_\rho}}(T^2) \to \Sh_{\Lambda_{\Sigma_\rho}}(T^2)$ the  degeneration functor corresponding to a single $\rho \in \Sigmao$. We also set $\Theta_\rho := CCC_{X_{\Sigma_\rho}}^{-1}(\cO_{X_{\Sigma_\rho}}) \in \Sh_{\Lambda_{\Sigma_\rho}}(T^2)$. The restriction functors from $\cX_{\stackfano}$ to $\cX_{\stackfan_\rho}$ and from $X_{\Sigmao}$ to $X_{\Sigma_\rho}$ are identified with convolution with $\Theta_\rho$ by the CCC. Thus what we must check explicitly is the commutativity of the following diagram, corresponding to the isomorphism $\pi_*i_{\cX_{\stackfan_\rho}}^* \cong i_{X_{\Sigma_\rho}}^* \pi_*$ of functors $\IndCoh(\cX_{\stackfan}) \to \IndCoh(X_{\Sigma_\rho})$.
\[
\begin{tikzpicture}[thick,>=\arrtip]
\newcommand*{\width}{4.5}; \newcommand*{\hta}{2}; \newcommand*{\htb}{2};
	\node (a) at (0,0) {$\Sh_{\Lambda_{\stackfano}}(T^2)$};
	\node (b) at (\width,0) {$\Sh_{\Lambda_{\Sigmao}}(T^2)$};
	\node (c) at (0,-\hta) {$\Sh_{\Lambda_{\stackfan_\rho}}(T^2)$};
	\node (d) at (\width,-\hta) {$\Sh_{\Lambda_{\Sigma_\rho}}(T^2)$};
	\draw[->] (a) to node[above] {$K_{\{\Lambda_t\}}$} (b);
	\draw[->] (c) to node[above] {$K_{\{\Lambda^\rho_t\}}$} (d);
	\draw[->] (a) to node[left] {$\Theta_\rho \star -$} (c); 
	\draw[->] (b) to node[right] {$\Theta_\rho \star -$} (d);

\end{tikzpicture}
\]

Consider the Legendrian $\Lambda_{[0,1)}\subset T^\infty(T^2 \times [0,1))$ associated to the isotopy $\{\Lambda_t\}_{t \in [0,1)}$. Given subsets $\Lambda \subset T^\infty X$, $\Lambda' \subset T^\infty Y$ we will write $\Lambda \Legtimes \Lambda' \subset T^\infty (X \times Y)$ for the subset whose cone in $T^* (X \times Y)$ is the product of the cones of $\Lambda$ and $\Lambda'$. We then let $\Lambda_I$ be the union in $T^\infty (T^2 \times I)$ of $\Lambda_{[0,1)}$ with $\Lambda_1 \Legtimes T^\infty_{\{1\}} I$.
We now consider the pointwise multiplication map $m: T^2 \times T^2 \times I \to T^2 \times I$ and claim the following about the interaction of $\Lambda_1 = \Lambda_{\Sigmao}$ with $\Lambda_I$: if $\cF \in \Sh_{\Lambda_1}(T^2)$ and $\cG \in \Sh_{\Lambda_I}(T^2 \times I)$, then $m_!(\cF \boxtimes \cG)$ is also in $\Sh_{\Lambda_I}(T^2 \times I)$. 

To see this, we once again apply \cite[Prop. 5.4.4]{KS94}, obtaining
\begin{equation*}
	\begin{aligned}
		SS(m_!(\cF \boxtimes \cG)) \subset &\{(p(m), n, t, \tau) \in T^*T^2 \times T^*I \text{ such that } \exists(p(m_1), n) \in SS(\cF),\\ &\quad  (p(m_2), n, t, \tau) \in SS(\cG) \text{ with }p(m_1+m_2)=p(m)\}.
	\end{aligned}
\end{equation*}
If $n$ is nonzero for such a point then $n$ must lie on the negative of some ray $\rho$ since $(p(m_1), n) \in SS(\cF)$ and $\cF \in \Sh_{\Lambda_{\Sigmao}}(T^2)$. On the other hand, we must then have $p(m_1) \in p(\rho^\perp_0)$ while $p(m_2)$ lies on
$$\bigcup_{n = 0}^{|C_\rho|-1} p(\rho^\perp + n t\chi_\rho) \subset T^2.$$
Since this second set is closed under addition by elements of $p(\rho^\perp_0)$ the claim follows.

It follows in particular that we have a well-defined arrow on the right side of the following square, which then commutes by base change.
\[
\begin{tikzpicture}
[thick,>=\arrtip]
\node (a) at (0,0) {$\Sh_{\Lambda_1 \Legtimes \Lambda_0}(T^2 \times T^2)$};
\node (b) at (5,0) {$\Sh_{\Lambda_1 \Legtimes \Lambda_{[0,1)}}(T^2 \times T^2 \times [0,1))$};
\node (c) at (0,-2) {$\Sh_{\Lambda_0}(T^2)$};
\node (d) at (5,-2) {$\Sh_{\Lambda_{[0,1)}}(T^2 \times [0,1))$};
\draw[<-] (a) to node[above] {$i_0^*$} (b);
\draw[->] (b) to node[right] {$m_!$} (d);
\draw[->] (a) to node[left] {$m_!$} (c);
\draw[<-] (c) to node[above] {$i_0^*$} (d);
\end{tikzpicture}
\]
Here and below we abbreviate e.g. $(id_{T^2 \times T^2} \times i_0)^*$ to $i_0^*$ when no ambiguity should result. 
We note again that since $\{\Lambda_t\}$ is an isotopy for $t \in [0,1)$ the horizontal arrows are in fact equivalences. 

We now consider the following diagram, its middle vertical arrow also being well-defined by the preceding discussion. 
\[
\begin{tikzpicture}[thick,>=\arrtip]
\newcommand*{\xb}{-.5}; 
\newcommand*{\xc}{5.5}; \newcommand*{\xd}{11};
\newcommand*{\hta}{2};
\node (b) at (\xb,0) {$\Sh_{\Lambda_1 \Legtimes \Lambda_{[0,1)}}(T^2 \times T^2 \times [0,1))$};
\node (c) at (\xc,0) {$\Sh_{\Lambda_1 \Legtimes \Lambda_{I}}(T^2 \times T^2 \times I)$};
\node (d) at (\xd,0) {$\Sh_{\Lambda_1 \Legtimes \Lambda_1}(T^2 \times T^2)$};
\node (f) at (\xb,-\hta) {$\Sh_{\Lambda_{[0,1)}}(T^2 \times [0,1))$};
\node (g) at (\xc,-\hta) {$\Sh_{\Lambda_{I}}(T^2 \times I)$};
\node (h) at (\xd,-\hta) {$\Sh_{\Lambda_1}(T^2)$};
\draw[->] (b) to node[above] {$(i_{[0,1)})_*$} (c); \draw[->] (c) to node[above] {$i_1^*$} (d);
\draw[->] (f) to node[above] {$(i_{[0,1)})_*$} (g); \draw[->] (g) to node[above] {$i_1^*$} (h);
\draw[->] (b) to node[left] {$m_!$} (f);
\draw[->] (c) to node[left] {$m_!$} (g); \draw[->] (d) to node[right] {$m_!$} (h);
\end{tikzpicture}
\]
The left square commutes since the two ways around the square are just different ways of factoring the product map $(m \times i)_*$ ($T^2$ being compact). The right square commutes by base change. Thus combining the total square with the one obtained above we obtain another commuting square
\[
\begin{tikzpicture}
[thick,>=\arrtip]
\node (a) at (0,0) {$\Sh_{\Lambda_1 \Legtimes \Lambda_0}(T^2 \times T^2)$};
\node (b) at (8,0) {$\Sh_{\Lambda_1 \Legtimes \Lambda_1}(T^2 \times T^2)$};
\node (c) at (0,-2) {$\Sh_{\Lambda_0}(T^2)$};
\node (d) at (8,-2) {$\Sh_{\Lambda_1}(T^2).$};
\draw[->] (a) to node[above] {$i_1^* \circ (i_{[0,1)})_* \circ (i_0^*)^{-1}$} (b);
\draw[->] (b) to node[right] {$m_!$} (d);
\draw[->] (a) to node[left] {$m_!$} (c);
\draw[->] (c) to node[above] {$i_1^* \circ (i_{[0,1)})_* \circ (i_0^*)^{-1}$} (d);
\end{tikzpicture}
\]
We now compare the two ways around the square after precomposing with the functor $$ \Theta_\rho \boxtimes -: \Sh_{\Lambda_0}(T^2) \to \Sh_{\Lambda_1 \Legtimes \Lambda_0}(T^2 \times T^2).$$
The top right path results in the functor $\Theta_\rho \star K_{\{\Lambda_t\}}(-)$ while the bottom left path results in $K_{\{\Lambda_t\}}(\Theta_\rho \star -)$. Thus we have identified the two functors we needed to identify, completing the proof. 
\end{proof}

\bibliographystyle{amsalpha}
\bibliography{bibliography}

\newcommand{\etalchar}[1]{$^{#1}$}
\providecommand{\bysame}{\leavevmode\hbox to3em{\hrulefill}\thinspace}
\providecommand{\MR}{\relax\ifhmode\unskip\space\fi MR }
\providecommand{\MRhref}[2]{%
  \href{http://www.ams.org/mathscinet-getitem?mr=#1}{#2}
}
\providecommand{\href}[2]{#2}
\begin{thebibliography}{FHM{\etalchar{+}}06}

\bibitem[BCS05]{BCS05}
L.~Borisov, L.~Chen, and G.~Smith, \emph{The orbifold {C}how ring of toric
  {D}eligne-{M}umford stacks}, J. Amer. Math. Soc. \textbf{18} (2005), no.~1,
  193--215.

\bibitem[BD16]{BD}
C.~Brav and T.~Dyckerhoff, \emph{Relative {C}alabi-{Y}au structures},
  arXiv:1606.00619 (2016).

\bibitem[Bea90]{B}
A.~Beauville, \emph{Jacobiennes des courbes spectrales et syst\`emes
  hamiltoniens compl\`etement int\'egrables}, Acta Math. \textbf{164} (1990),
  no.~3-4, 211--235.

\bibitem[BM09]{BM}
M.~Bender and S.~Mozgovoy, \emph{Crepant resolutions and brane tilings {II}:
  {T}ilting bundles}, arXiv:0909.2013 (2009).

\bibitem[Boc13]{Boc}
R.~Bocklandt, \emph{Calabi-{Y}au algebras and weighted quiver polyhedra}, Math.
  Z. \textbf{273} (2013), no.~1-2, 311--329.

\bibitem[Bon06]{Bon06}
A.~Bondal, \emph{Derived categories of toric varieties}, Convex and algebraic
  geometry, Oberwolfach conference reports, vol.~3, EMS Publishing House, 2006,
  pp.~284--286.

\bibitem[Bro12]{Bro}
N.~Broomhead, \emph{Dimer models and {C}alabi-{Y}au algebras}, Mem. Amer. Math.
  Soc. \textbf{215} (2012), no.~1011, viii+86.

\bibitem[CW12]{CW}
S.~Cherkis and R.~Ward, \emph{Moduli of monopole walls and amoebas}, J. High
  Energy Phys. (2012), no.~5, 090, front matter+36.

\bibitem[Dav11]{Dav}
B.~Davison, \emph{Consistency conditions for brane tilings}, J. Algebra
  \textbf{338} (2011), 1--23.

\bibitem[DM96]{DM}
R.~Donagi and E.~Markman, \emph{Spectral covers, algebraically completely
  integrable, {H}amiltonian systems, and moduli of bundles}, Integrable systems
  and quantum groups ({M}ontecatini {T}erme, 1993), Lecture Notes in Math.,
  vol. 1620, Springer, Berlin, 1996, pp.~1--119.

\bibitem[EFS12]{EFS}
R.~Eager, S.~Franco, and K.~Schaeffer, \emph{Dimer models and integrable
  systems}, J. High Energy Phys. (2012), no.~6, 106, front matter+24.

\bibitem[FHKV08]{FHKV}
B.~Feng, Y.-H. He, K.~Kennaway, and C.~Vafa, \emph{Dimer models from mirror
  symmetry and quivering amoebae}, Adv. Theor. Math. Phys. \textbf{12} (2008),
  no.~3, 489--545.

\bibitem[FHM{\etalchar{+}}06]{FHMSVW}
S.~Franco, A.~Hanany, D.~Martelli, J.~Sparks, D.~Vegh, and B.~Wecht,
  \emph{Gauge theories from toric geometry and brane tilings}, J. High Energy
  Phys. (2006), no.~1, 128, 40.

\bibitem[FHM16]{FHM}
S.~Franco, Y.~Hatsuda, and M.~Marino, \emph{Exact quantization conditions for
  cluster integrable systems}, J. Stat. Mech. Theory Exp. (2016), no.~6,
  063107, 30.

\bibitem[FHV{\etalchar{+}}06]{FHVWK}
S.~Franco, A.~Hanany, D.~Vegh, B.~Wecht, and K.~Kennaway, \emph{Brane dimers
  and quiver gauge theories}, J. High Energy Phys. (2006), no.~1, 096, 48.

\bibitem[FLS16a]{FLS16b}
S.~Franco, S.~Lee, and R.~Seong, \emph{Brane brick models and 2d {$(0, 2)$}
  triality}, J. High Energy Phys. (2016), no.~5, 20.

\bibitem[FLS16b]{FLS16a}
\bysame, \emph{Brane brick models, toric {C}alabi-{Y}au 4-folds and 2d
  {$(0,2)$} quivers}, J. High Energy Phys. (2016), no.~2, 047, front matter+66.

\bibitem[FLSV17]{FLSV}
S.~Franco, S.~Lee, R.~Seong, and C.~Vafa, \emph{Brane brick models in the
  mirror}, J. High Energy Phys. (2017), no.~2, 106, front matter+62.

\bibitem[FLTZ11]{FLTZ}
B.~Fang, C.-C.~M. Liu, D.~Treumann, and E.~Zaslow, \emph{A categorification of
  {M}orelli's theorem}, Invent. Math. \textbf{186} (2011), no.~1, 79--114.

\bibitem[FLTZ14]{FLTZ3}
\bysame, \emph{The coherent-constructible correspondence for toric
  {D}eligne-{M}umford stacks}, Int. Math. Res. Not. IMRN (2014), no.~4,
  914--954.

\bibitem[FM16]{FM}
V.~V. Fock and A.~Marshakov, \emph{Loop groups, clusters, dimers and integrable
  systems}, Geometry and quantization of moduli spaces, Adv. Courses Math. CRM
  Barcelona, Birkh\"auser/Springer, Cham, 2016, pp.~1--66.

\bibitem[FU10]{FU10}
Masahiro Futaki and Kazushi Ueda, \emph{Exact {L}efschetz fibrations associated
  with dimer models}, Math. Res. Lett. \textbf{17} (2010), no.~6, 1029--1040.

\bibitem[FU14]{FU14}
M.~Futaki and K.~Ueda, \emph{Tropical coamoeba and torus-equivariant
  homological mirror symmetry for the projective space}, Comm. Math. Phys.
  \textbf{332} (2014), no.~1, 53--87.

\bibitem[FV06]{FV}
S.~Franco and D.~Vegh, \emph{Moduli spaces of gauge theories from dimer models:
  proof of the correspondence}, J. High Energy Phys. (2006), no.~11, 054, 26.

\bibitem[GK13]{GK}
A.~B. Goncharov and R.~Kenyon, \emph{Dimers and cluster integrable systems},
  Ann. Sci. \'Ec. Norm. Sup\'er. (4) \textbf{46} (2013), no.~5, 747--813.

\bibitem[GKS12]{GKS}
S.~Guillermou, M.~Kashiwara, and P.~Schapira, \emph{Sheaf quantization of
  {H}amiltonian isotopies and applications to nondisplaceability problems},
  Duke Math. J. \textbf{161} (2012), no.~2, 201--245.

\bibitem[GS15]{GS15}
A.~Geraschenko and M.~Satriano, \emph{Toric stacks {I}: {T}he theory of stacky
  fans}, Trans. Amer. Math. Soc. \textbf{367} (2015), no.~2, 1033--1071.

\bibitem[GS17]{GS17}
B.~Gammage and V.~Shende, \emph{Mirror symmetry for very affine hypersurfaces},
  arXiv:1707.02959 (2017).

\bibitem[GSTV16]{GSTV}
M.~Gekhtman, M.~Shapiro, S.~Tabachnikov, and A.~Vainshtein, \emph{Integrable
  cluster dynamics of directed networks and pentagram maps}, Adv. Math.
  \textbf{300} (2016), 390--450.

\bibitem[Gui12]{Gui12}
S.~Guillermou, \emph{Quantizations of conic {L}agrangian submanifolds of
  cotangent bundles}, arXiv:1212.5818 (2012).

\bibitem[Han18]{Han18}
A.~Hanlon, \emph{Monodromy of fukaya-seidel categories mirror to toric
  varieties}, arXiv:1809.06001 (2018).

\bibitem[HHV06]{HHV}
A.~Hanany, C.~Herzog, and D.~Vegh, \emph{Brane tilings and exceptional
  collections}, J. High Energy Phys. (2006), no.~7, 001, 44.

\bibitem[Hic]{Hic18}
J.~Hicks, \emph{Tropical lagrangians and mirror symmetry}, in preparation.

\bibitem[HK05]{HK}
A.~Hanany and K.~Kennaway, \emph{Dimer models and toric diagrams},
  arXiv:hep-th/0503149 (2005).

\bibitem[HV00]{HV}
K.~Hori and C.~Vafa, \emph{Mirror symmetry}, arXiv:hep-th/0002222 (2000).

\bibitem[HV07]{HV07}
A.~Hanany and D.~Vegh, \emph{Quivers, tilings, branes and rhombi}, J. High
  Energy Phys. (2007), no.~10, 029, 35.

\bibitem[IU09]{IU09}
A.~Ishii and K.~Ueda, \emph{Dimer models and exceptional collections},
  arXiv:hep-th/0911.4529 (2009).

\bibitem[IU11]{IU}
\bysame, \emph{A note on consistency conditions on dimer models}, Higher
  dimensional algebraic geometry, RIMS K\^oky\^uroku Bessatsu, B24, Res. Inst.
  Math. Sci. (RIMS), Kyoto, 2011, pp.~143--164.

\bibitem[JT17]{JT17}
X.~Jin and D.~Treumann, \emph{Brane structures in microlocal sheaf theory},
  arXiv:1704.04291 (2017).

\bibitem[KO06]{KO}
R.~Kenyon and A.~Okounkov, \emph{Planar dimers and {H}arnack curves}, Duke
  Math. J. \textbf{131} (2006), no.~3, 499--524.

\bibitem[KOS06]{KOS}
R.~Kenyon, A.~Okounkov, and S.~Sheffield, \emph{Dimers and amoebae}, Ann. of
  Math. (2) \textbf{163} (2006), no.~3, 1019--1056.

\bibitem[KS94]{KS94}
M.~Kashiwara and P.~Schapira, \emph{Sheaves on manifolds}, Grundlehren der
  Mathematischen Wissenschaften [Fundamental Principles of Mathematical
  Sciences], vol. 292, Springer-Verlag, Berlin, 1994, With a chapter in French
  by Christian Houzel, Corrected reprint of the 1990 original.

\bibitem[Kuw16]{Kuw16}
T.~Kuwagaki, \emph{The nonequivariant coherent-constructible correspondence for
  toric stacks}, arXiv:1610.03214 (2016).

\bibitem[Kuw17]{Kuw17}
\bysame, \emph{The nonequivariant coherent-constructible correspondence for
  toric surfaces}, J. Differential Geom. \textbf{107} (2017), no.~2, 373--393.

\bibitem[Mat18a]{Mat18a}
D.~Matessi, \emph{Lagrangian pairs of pants}, arXiv:1802.02993 (2018).

\bibitem[Mat18b]{Mat18b}
\bysame, \emph{Lagrangian submanifolds from tropical hypersurfaces},
  arXiv:1804.01469 (2018).

\bibitem[Mik18]{Mik}
G.~Mikhalkin, \emph{Examples of tropical-to-lagrangian correspondence},
  arXiv:1802.06473 (2018).

\bibitem[MR10]{MR}
S.~Mozgovoy and M.~Reineke, \emph{On the noncommutative {D}onaldson-{T}homas
  invariants arising from brane tilings}, Adv. Math. \textbf{223} (2010),
  no.~5, 1521--1544.

\bibitem[Nad09]{Nad09}
D.~Nadler, \emph{Microlocal branes are constructible sheaves}, Selecta Math.
  (N.S.) \textbf{15} (2009), no.~4, 563--619.

\bibitem[Nad16]{Nad16}
\bysame, \emph{Wrapped microlocal sheaves on pairs of pants}, arXiv:1604.00114
  (2016).

\bibitem[Nag12]{Nag}
K.~Nagao, \emph{Derived categories of small toric {C}alabi-{Y}au 3-folds and
  curve counting invariants}, Q. J. Math. \textbf{63} (2012), no.~4, 965--1007.

\bibitem[Ng01]{Ng01}
Lenhard Ng, \emph{The legendrian satellite construction}, arXiv:math/0112105
  (2001).

\bibitem[NN11]{NN}
K.~Nagao and H.~Nakajima, \emph{Counting invariant of perverse coherent sheaves
  and its wall-crossing}, Int. Math. Res. Not. IMRN (2011), no.~17, 3885--3938.

\bibitem[NZ09]{NZ}
D.~Nadler and E.~Zaslow, \emph{Constructible sheaves and the {F}ukaya
  category}, J. Amer. Math. Soc. \textbf{22} (2009), no.~1, 233--286.

\bibitem[ORV06]{ORV}
A.~Okounkov, N.~Reshetikhin, and C.~Vafa, \emph{Quantum {C}alabi-{Y}au and
  classical crystals}, The unity of mathematics, Progr. Math., vol. 244,
  Birkh\"auser Boston, Boston, MA, 2006, pp.~597--618.

\bibitem[RS18]{RS18}
M.~Robalo and P.~Schapira, \emph{A lemma for microlocal sheaf theory in the
  {$\infty$}-categorical setting}, Publ. Res. Inst. Math. Sci. \textbf{54}
  (2018), no.~2, 379--391.

\bibitem[RSTZ14]{RSTZ}
H.~Ruddat, N.~Sibilla, D.~Treumann, and E.~Zaslow, \emph{Skeleta of affine
  hypersurfaces}, Geom. Topol. \textbf{18} (2014), no.~3, 1343--1395.

\bibitem[SS14]{SS14}
S.~Scherotzke and N.~Sibilla, \emph{The nonequivariant coherent-constructible
  correspondence and tilting}, arXiv:1402.3360 (2014).

\bibitem[ST16]{ST}
V.~Shende and A.~Takeda, \emph{Symplectic structures from topological {F}ukaya
  categories}, arXiv:1605.02721 (2016).

\bibitem[STWZ15]{STWZ}
V.~Shende, D.~Treumann, H.~Williams, and E.~Zaslow, \emph{Cluster varieties
  from {L}egendrian knots}, arXiv:1512.08942 (2015).

\bibitem[STZ17]{STZ}
V.~Shende, D.~Treumann, and E.~Zaslow, \emph{Legendrian knots and constructible
  sheaves}, Invent. Math. \textbf{207} (2017), no.~3, 1031--1133.

\bibitem[Sze08]{Sze}
B.~Szendr\"oi, \emph{Non-commutative {D}onaldson-{T}homas invariants and the
  conifold}, Geom. Topol. \textbf{12} (2008), no.~2, 1171--1202.

\bibitem[Tre10]{Tre10}
D.~Treumann, \emph{Remarks on the nonequivariant coherent-constructible
  correspondence for toric varieties}, arXiv:1006.5756 (2010).

\bibitem[Tyo12]{Tyo12}
I.~Tyomkin, \emph{Tropical geometry and correspondence theorems via toric
  stacks}, Math. Ann. \textbf{353} (2012), no.~3, 945--995.

\bibitem[UY13]{UY}
K.~Ueda and M.~Yamazaki, \emph{Homological mirror symmetry for toric orbifolds
  of toric del {P}ezzo surfaces}, J. Reine Angew. Math. \textbf{680} (2013),
  1--22.

\bibitem[Zho17]{Zho17}
P.~Zhou, \emph{Twisted polytope sheaves and coherent-constructible
  correspondence for toric varieties}, arXiv:1701.00689 (2017).

\bibitem[Zho18a]{Zho18b}
\bysame, \emph{Lagrangian skeleta of hypersurfaces in $(\mathbb{C}^\times)^n$},
  arXiv:1803.00320 (2018).

\bibitem[Zho18b]{Zho18}
\bysame, \emph{Sheaf quantization of {L}egendrian isotopy}, arXiv:1804.08928
  (2018).

\end{thebibliography}

\end{document}